\documentclass[11pt]{amsart}
\usepackage{graphicx}
\usepackage{amsmath}
\usepackage{amssymb}
\usepackage{amsfonts}
\usepackage{verbatim}
\usepackage{scalefnt}
\usepackage{multirow}
\usepackage{mathtools}
\usepackage{bbm}
\usepackage{float}
\usepackage{multicol}
\restylefloat{figure}

\usepackage{pstricks,tikz}
\usepackage{fullpage}
\usepackage{fancyhdr}
\usepackage{hyperref}

\begin{document}

\def\COMMENT#1{}
\let\COMMENT=\footnote

\newtheorem{theorem}{Theorem}[section]
\newtheorem{fact}[theorem]{Fact}
\newtheorem{lemma}[theorem]{Lemma}
\newtheorem{proposition}[theorem]{Proposition}
\newtheorem{corollary}[theorem]{Corollary}
\newtheorem{conjecture}[theorem]{Conjecture}
\newtheorem{claim}[theorem]{Claim}
\newtheorem{definition}[theorem]{Definition}
\newtheorem{Proposition}[theorem]{Proposition}
\newtheorem{firstthm}{Proposition}
\newtheorem{thm}[firstthm]{Theorem}
\newtheorem{prop}[firstthm]{Proposition}

\newtheorem{cor}[firstthm]{Corollary}
\newtheorem{problem}[firstthm]{Problem}
\newtheorem{defin}[firstthm]{Definition}
\newtheorem{conj}[firstthm]{Conjecture}

\newtheorem{ques}[firstthm]{Question} 
\numberwithin{equation}{section}

\def\eps{{\varepsilon}}
\newcommand{\cP}{\mathcal{P}}
\newcommand{\cT}{\mathcal{T}}
\newcommand{\cL}{\mathcal{L}}
\newcommand{\ex}{\mathbb{E}}
\newcommand{\eul}{e}
\newcommand{\pr}{\mathbb{P}}

\title{ON DEGREE SEQUENCES FORCING THE SQUARE OF A HAMILTON CYCLE}
\title[On degree sequences forcing the square of a Hamilton cycle]{On degree sequences forcing the square of a Hamilton cycle}
\author{Katherine Staden and Andrew Treglown}

\begin{abstract}
A famous conjecture of P\'osa from 1962 asserts that every graph on $n$ vertices and with minimum degree at least $2n/3$ contains the square of a Hamilton cycle. The conjecture was proven for large graphs in 1996 by Koml\'os,  S\'ark\"ozy and  Szemer\'edi~\cite{kssposa}. In this paper we prove a degree sequence version of P\'osa's conjecture:
Given any $\eta >0$, every graph $G$ of sufficiently large order $n$ contains the square of a Hamilton cycle if its degree sequence $d_1\leq \dots \leq d_n$ satisfies $d_i \geq (1/3+\eta)n+i$ for all $i \leq n/3$. The degree sequence condition here is asymptotically best possible. Our approach uses a hybrid of the Regularity-Blow-up method and the Connecting-Absorbing method.
\end{abstract}

\date{\today}
\maketitle


\section{Introduction}

One of the most fundamental results in extremal graph theory is Dirac's theorem~\cite{dirac} which states that every graph $G$ on $n \geq 3$ vertices with minimum degree $\delta (G)$ at least $n/2$ contains a Hamilton cycle.
It is easy to see that the minimum degree condition here is best possible. The \emph{square} of a Hamilton cycle $C$ is obtained from $C$ by adding an edge between every pair of vertices of distance two on $C$.
A famous conjecture of P\'osa from 1962~(see~\cite{posa}) provides an analogue of Dirac's theorem for the square of a Hamilton cycle.
\begin{conjecture}[P\'osa~\cite{posa}]\label{posaconj}
Let $G$ be a graph on $n$ vertices. If $\delta (G) \geq 2 n/3$,
then $G$ contains the square of a Hamilton cycle.
\end{conjecture}
Again, it is easy to see that the minimum degree condition in P\'osa's conjecture cannot be lowered. The conjecture was  intensively studied in the 1990s (see e.g.~\cite{fan0, fan1, fan2, fan3, fau}), culminating in its proof for large graphs $G$ by Koml\'os,  S\'ark\"ozy and  Szemer\'edi~\cite{kssposa}. The proof applies  Szemer\'edi's Regularity lemma and as such  the graphs $G$ considered are extremely large. More recently, the lower bound on the size of $G$ in this result has been significantly lowered (see~\cite{ cdk, lev}).

Although the minimum degree condition is best possible in Dirac's theorem, this does not necessarily mean that one cannot significantly strengthen this result. Indeed, Ore~\cite{ore} showed that a graph~$G$ of
order $n\ge 3$ contains a
Hamilton cycle if $d(x)+d(y) \geq n$ for all non-adjacent $x \not = y\in V(G)$. 
The following result of P\'osa~\cite{posathm} provides a \emph{degree sequence} condition that ensures Hamiltonicity.
\begin{theorem}[P\'osa~\cite{posathm}]\label{posathm}
Let $G$ be a graph on $n\ge 3$ vertices with degree sequence $d_1\leq \dots \leq d_n$. 
If $d_i \geq i+1$ for all $i<(n-1)/2$ and if additionally $d_{\lceil n/2\rceil} \geq \lceil n/2\rceil$ when~$n$ is odd, then $G$ contains a Hamilton cycle. 
\end{theorem}
Notice that Theorem~\ref{posathm} is significantly stronger than Dirac's theorem as it allows for almost half of the vertices of $G$ to have degree less than $n/2$. A theorem of Chv\'atal~\cite{ch} generalises Theorem~\ref{posathm} by  characterising all those degree sequences which
ensure the existence of a Hamilton cycle in a graph: Suppose that the degrees of a graph $G$
are $d_1\le \dots \le d_n$. If $n \geq 3$ and $d_i \geq i+1$ or $d_{n-i} \geq n-i$
for all $i <n/2$ then $G$ is Hamiltonian. Moreover, if $d_1 \leq \dots \leq d_n$
is a  degree sequence that does not satisfy this condition then there
exists a non-Hamiltonian graph $G$ whose degree sequence $d'_1 \leq  \dots \leq d'_n$ is such
that $d'_i \geq d_i$ for all $1\leq i \leq n$.

Recently there has been an  interest in generalising P\'osa's conjecture.
An `Ore-type' analogue of P\'osa's conjecture has been proven for large graphs in~\cite{chau, deore}. A random version of P\'osa's conjecture was proven by K\"uhn and Osthus in~\cite{ranposa}.
In~\cite{file}, Allen, B\"ottcher and Hladk\'y determined the minimum degree threshold that ensures a large graph  contains a square cycle of a given length. The problem of finding the square of a Hamilton cycle in a pseudorandom graph has  recently been studied in~\cite{new}.
The focus of this paper is to investigate degree sequence conditions that guarantee a graph contains the square of a Hamilton cycle. This problem was raised in the arXiv version of~\cite{bkt}.
The main result of this paper is the following degree sequence version of P\'osa's conjecture.
\begin{theorem}\label{mainthm}
Given any $\eta >0$ there exists an $n_0 \in \mathbb N$ such that the following holds. If $G$ is a graph on $n \geq n_0$ vertices whose degree sequence $d_1\leq \dots \leq d_n$ satisfies
$$ d_i \geq n/3+i+\eta n \  \text{ for all  $\ i\leq n/3$,}$$ then $G$ contains the square of a Hamilton cycle.
\end{theorem}
Note that Theorem~\ref{mainthm} allows  for almost $n/3$ vertices in $G$ to have degree substantially smaller than $2n/3$. However, it does not quite imply P\'osa's conjecture for large graphs due to the term $\eta n$.
A surprising facet of the problem is that the term $\eta n$ in Theorem~\ref{mainthm} cannot be omitted completely.
Indeed, an example from the arXiv version of~\cite{bkt} shows that the term $\eta n$ cannot even be replaced by  $o(\sqrt{n})$ for every $i\leq n/3$.  So in this sense Theorem~\ref{mainthm} is close to best possible. (Extremal examples for Theorem~\ref{mainthm} are discussed in more detail in Section~\ref{sec:extremal}.)
We suspect though that the degrees in Theorem~\ref{mainthm} can be capped at $2n/3$.
\begin{conjecture}\label{conjy}
Given any $\eta >0$ there exists an $n_0 \in \mathbb N$ such that the following holds. If $G$ is a graph on $n \geq n_0$ vertices whose degree sequence $d_1\leq \dots \leq d_n$ satisfies
$$ d_i \geq \min \{ n/3+i+\eta n , 2n/3 \} \  \text{ for all \ $i$,}$$ then $G$ contains the square of a Hamilton cycle.
\end{conjecture}
It would be of considerable interest to establish an analogue of Chv\'atal's theorem for the square of a Hamilton cycle, i.e., to characterise those degree sequences which force the square of a Hamilton cycle.
However, at present we do not have a conjecture for this. In general, we believe that it would be extremely difficult to strengthen Theorem~\ref{mainthm}, and it is likely that several further new ideas would be required.

A well-known result of Aigner and Brandt~\cite{aigner} and Alon and Fischer~\cite{fishy} states that if  $G$ is a graph on $n$ vertices  with minimum degree $\delta (G) \geq (2n-1)/3$ then $G$ contains every graph $H$ on $n$ vertices with maximum degree $\Delta (H) \leq 2$. This resolves a special case of the famous Bollob\'as--Eldridge--Catlin Conjecture~\cite{boleld,catlin}. (A conjecture of El-Zahar~\cite{elzahar}, that was proven for large graphs by Abbasi~\cite{abbasi}, implies  that for many graphs $H$ with  $\Delta (H) \leq 2$, the minimum degree condition here can be substantially lowered.)
Since a square path on $n$ vertices contains any such graph $H$, an immediate consequence of Theorem~\ref{mainthm} is the following degree sequence result.
\begin{corollary}\label{mainthm2}
Given any $\eta >0$ there exists an $n_0 \in \mathbb N$ such that the following holds. Suppose that $H$ is a graph on $n\geq n_0$ vertices such that $\Delta (H) \leq 2$. If $G$ is a graph on $n $ vertices whose degree sequence $d_1\leq \dots \leq d_n$ satisfies
$$ d_i \geq n/3+i+\eta n \  \text{ for all $\ i\leq n/3$,}$$ then $G$ contains $H$.
\end{corollary}

The case when $H$ is a triangle factor was proved in~\cite{triangle}, and in fact this result is used as a tool in the proof of Theorem~\ref{mainthm} (see Section~\ref{sectionalmost}).

A key component of the proof of Theorem~\ref{mainthm} is a special structural embedding lemma (Lemma~\ref{trianglecycle}), which is likely to be of independent interest. In particular, we believe that it could have applications to other embedding problems (see Section~\ref{concsec} for further discussion).

The proof of Theorem~\ref{mainthm} makes use of Szemer\'edi's Regularity lemma~\cite{szem} and the Blow-up lemma~\cite{blowuplem}. In Section~\ref{sketch} we give a detailed sketch of the proof. 
We discuss extremal examples for Theorem~\ref{mainthm} in Section~\ref{sec:extremal}.
After introducing some notation and preliminary results in Section~\ref{prelim}, we prove Theorem~\ref{mainthm} in Sections~\ref{sectionalmost}--\ref{sectionend}.


\section{Overview of the proof}\label{sketch}

Over the last few decades a number of powerful techniques have been developed for embedding problems in graphs. The Blow-up lemma~\cite{blowuplem}, in combination with the Regularity lemma~\cite{szem}, has been used to resolve a number of long-standing open problems, including P\'osa's conjecture for large graphs~\cite{kssposa}. More recently, the so-called Connecting-Absorbing method developed by R\"odl, Ruci\'nski and  Szemer\'edi~\cite{rrs2} 
has also proven to be highly effective in tackling such embedding problems.

Typically, both these approaches have been applied to graphs with `large' minimum degree. Our graph $G$ in Theorem~\ref{mainthm} may have minimum degree  $(1/3+o(1))n$. In particular, this is significantly smaller than the minimum degree threshold that forces the square of a Hamilton cycle in a graph (namely, $2n/3$). As we describe below, having vertices of relatively small degree makes the proof of Theorem~\ref{mainthm} highly involved and rather delicate. Indeed, our proof draws on ideas from both the Regularity-Blow-up method and the Connecting-Absorbing method. Further, we also develop a number of new ideas in order
to deal with these vertices of small degree.

\subsection{An approximate version of P\'osa's conjecture}
In order to highlight some of the difficulties in the proof of Theorem~\ref{mainthm}, we first give a sketch of a proof of an approximate version of P\'osa's conjecture. This is based on  the  proof of P\'osa's conjecture for large graphs given in~\cite{lev} that uses the absorbing method (rather than the regularity-based proof given in~\cite{kssposa}).

Let $0< \eps \ll \gamma \ll \eta$.
Suppose that $G$ is a sufficiently large graph on $n$ vertices with $\delta (G) \geq (2/3+\eta)n$.  We wish to find the square of a Hamilton cycle in $G$. The proof splits into 
three main parts.
\begin{itemize}
\item {\bf Step 1 (Absorbing square path):} Find an `absorbing' square path $P_A$ in $G$ such that $|P_A|\leq \gamma n$. $P_A$ has the property that given \emph{any} set $A \subseteq V(G) \setminus V(P_A)$ such that $|A|\leq 2 \eps n$, $G$ contains a square path $P$ with vertex set $V(P_A) \cup A$, where the first and last two vertices on $P$ are the same as the first and last two vertices on $P_A$.
\item {\bf Step 2 (Reservoir set):} Let $G':=G\setminus V(P_A)$. Find a `reservoir' set $\mathcal R \subseteq V(G')$ such that $|\mathcal R|\leq \eps n$. $\mathcal R$ has the property that, given \emph{arbitrary} disjoint ordered edges $ab,cd \in E(G)$, there are `many' short  square paths $P$ in $G$ so that: (i) The first two vertices on $P$ are $a,b$ respectively; (ii) The last two vertices on $P$ are $c,d$ respectively; (iii) $V(P)\setminus \{a,b,c,d\} \subseteq \mathcal R$.
\item {\bf Step 3 (Almost tiling with square paths):} Let $G'':= G' \setminus \mathcal R$. Find a collection $\mathcal P$ of a bounded number of vertex-disjoint square paths in $G''$ that together cover all but $\eps n$ of the vertices in $G''$.
\end{itemize}
Assuming that $\delta (G) \geq (2/3+\eta)n$, the proof of each of these three steps is not too involved. (Note though that the proof in~\cite{lev} is  more technical since there $\delta (G) \geq 2n/3$.)

After completing  Steps~1--3, it is straightforward to find the square of  a Hamilton cycle in $G$. Indeed, suppose $ab$ is the last edge on a square path $P_1$ from $\mathcal P$ and $cd$ is the first edge on a square path $P_2$ from $\mathcal P$. Then Step~2 implies that we can `go through' $\mathcal R$ to join $P_1$ and $P_2$ into a single square path in $G$. Repeating this process we can obtain a square cycle $C$ in $G$ that contains all the square paths from $\mathcal P$. Further, we may also incorporate the absorbing square path $P_A$ into $C$. $C$ now covers almost all the vertices of $G$. 
We then use $P_A$ to absorb all the vertices from $V(G)\setminus V(C)$  into $C$  to obtain the square of a Hamilton cycle.

\subsection{A degree sequence version of P\'osa's conjecture}
Suppose that $G$ is a sufficiently large graph on $n$ vertices as in the statement of Theorem~\ref{mainthm}. A result of the second author~\cite{triangle} guarantees that $G$ contains a collection of $\lfloor n/3 \rfloor$ vertex-disjoint triangles  (see Theorem~\ref{trianglepack}). Further, this result together with a simple application of the Regularity lemma implies that $G$ in fact contains 
a collection $\mathcal P$ of a bounded number of vertex-disjoint square paths  that together cover almost all of the vertices in $G$. So we can indeed prove an analogue of Step~3 in this setting. In particular, if we could  find a reservoir set $\mathcal R$ as above, then certainly we would be able to join together the square paths in $\mathcal P$ through $\mathcal R$, to obtain an almost spanning square cycle $C$ in $G$.

Suppose that $ab,cd \in E(G)$  and we wish to find a square path $P$ in $G$ between $ab$ and $cd$. If  $d_G (a), d_G (b) <n/2$ then it may be the case that  $a$ and $b$ have no common neighbours. Then it is clearly impossible to find such a square path $P$ between $ab$ and $cd$ (since $ab$ does not lie in a single square path!). The degree sequence condition on $G$ is such that almost $n/6$ vertices in $G$ may have degree 
less than $n/2$. Therefore we cannot hope to find a reservoir set precisely as in Step~2 above.

We overcome this significant problem as follows. We first show that $G$ contains a reservoir set $\mathcal R$  that can \emph{only} be used to find a square path between pairs of edges $ab,cd \in E(G)$ of \emph{large degree} (namely, at least $(2/3+\eta)n$). This turns out to be quite involved (the whole of Section~\ref{connectingsec} is devoted to constructing $\mathcal R$). In order to use $\mathcal R$ to join together the
square paths $P \in \mathcal P$ into an almost spanning square cycle, we now require that the first and last two vertices on each such $P $ have large degree.

To find such a collection of square paths $\mathcal P$ we first find a special collection $\mathcal F$ of so-called `folded paths' in a reduced graph $R$ of $G$. Roughly speaking,  folded paths are a generalisation of the notion of a square path. Each such folded path $F \in \mathcal F$ will act as a `guide' for embedding one of the paths $P \in \mathcal P$ into $G$.
More precisely, there is a homomorphism from a square path $P$ into a folded path $F$.
 In particular, the structure of $F$ will ensure that the first and last two vertices on $P$ are `mapped' to large degree vertices in $G$. This is achieved in Section~\ref{sectionalmost}. 

Given our new reservoir set $\mathcal R$ and collection of square paths $\mathcal P$, we again can obtain an almost spanning square cycle $C$ in $G$. Further, if we could construct an absorbing square path $P_A$ as in Step~1, we would be able to absorb the vertices in $V(G)\setminus V(C)$ to obtain the square of a Hamilton cycle. However, we were unable to construct such an absorbing square path, and do not believe there is a `simple' way to construct one. (Though, one could construct such a square path $P_A$ if one only requires $P_A$ to absorb vertices of large degree.) Instead, our method now turns towards the Regularity-Blow-up approach.

Using the results from Sections~\ref{sectionalmost} and~\ref{connectingsec} we can obtain an almost spanning square cycle in the reduced graph $R$ of $G$. In fact, we obtain a much richer structure $Z_{\ell}$ in $R$ called a `triangle cycle'  (see Section~\ref{sectc}). $Z_{\ell}$ is a special $6$-regular graph on $3\ell$ vertices that contains the square of a Hamilton cycle. In particular, $Z_{\ell}$ contains 
a collection of vertex-disjoint triangles $T_{\ell}$ that together cover all the vertices in $Z_{\ell}$. 
Structures similar to $Z_\ell$ have been used previously for embedding other spanning subgraphs (see~e.g.~\cite{bst}).
We then show that $G$ contains an almost spanning structure $\mathcal C$ that looks like the `blow-up' of $ Z_{\ell}$.  More precisely, if $V(Z_{\ell})=\{1,\dots, 3\ell\}$ and $V_1, \dots , V_{3\ell}$ are the corresponding clusters in $G$, then
\begin{itemize}
\item $V(\mathcal C)= V_1 \cup \dots \cup V_{3\ell}$;
\item $\mathcal C [V_i,V_j]$ is $\eps$-regular whenever $ij \in E(Z_{\ell})$;
\item If $ij$ is an edge in a triangle $T \in T_{\ell}$ then $\mathcal C [V_i,V_j]$ is $\eps$-superregular.
\end{itemize}
We call $\mathcal C$ a `cycle structure' (see Section~\ref{subsec:square} for the formal definition).
The initial structure of $\mathcal C$ is such that it contains a spanning square cycle. However, since $\mathcal C$ is not necessarily spanning in $G$, this does not correspond to the square of a Hamilton cycle in $G$. We thus need to incorporate the `exceptional vertices' of $G$ into this cycle structure $\mathcal C$ in a balanced way so that at the end $\mathcal C$ (and hence $G$)  contains the square of a Hamilton cycle. The rich structure of $Z_{\ell}$ and thus $\mathcal C$ is vital for this. Again particular care is needed when incorporating exceptional vertices of small degree into our cycle structure. This is achieved in Section~\ref{sec:square}. This part of the proof builds on ideas used in~\cite{3partite, bst}.

\section{Extremal examples for Theorem~\ref{mainthm}}\label{sec:extremal}
In this section we describe examples which show that Theorem~\ref{mainthm} is asymptotically best possible.

Given a fixed graph $H$, an \emph{$H$-packing} in a graph $G$ is a collection of vertex-disjoint copies of $H$ in $G$.
We say that an $H$-packing is \emph{perfect} if it contains $\lfloor |G|/|H| \rfloor$ copies of $H$ in $G$, i.e. the maximum number.
Observe that the square of a Hamilton cycle contains a perfect $K_3$-packing.
The following proposition is a special case  of Proposition~17 in~\cite{bkt}.
It implies that one cannot replace $\eta n$ with $-1$ in Theorem~\ref{mainthm}.

\begin{proposition}\label{extremal1}
Suppose that $n \in 3\mathbb N$, $k \in \mathbb{N}$ and $1\leq k < n/3$. Then there exists a graph $G$ on $n$ vertices whose
degree sequence $d_1\leq \dots \leq d_n$ satisfies
$$
d_i = \begin{cases} n/3+k-1 &\mbox{if }\ \ 1 \leq i \leq k \\
2n/3 & \mbox{if }\ \ k+1 \leq i \leq n/3+k\\
n-k-1 & \mbox{if }\ \  n/3+k+1 \leq i \leq n-k+1\\
n-1 & \mbox{if }\ \  n-k+2 \leq i \leq n, \end{cases} 
$$
but such that $G$ does not contain a perfect $K_3$-packing.
\end{proposition}

\begin{proof}
Construct $G$ as follows.
The vertex set of $G$ is the union of disjoint sets $V_1,V_2,A,B$ of sizes $n/3$, $2n/3-2k+1$, $k-1$, $k$ respectively.
Add all edges from $B \cup V_{2} \cup A$ to $V_1$. Further, add all edges with both endpoints in $V_{2} \cup A$. Add all possible edges between $A$ and $B$. 

Consider an arbitrary copy $T$ of $K_3$ in $G$ which contains $b \in B$.
Since $B$ is an independent set in $G$ and there are no edges between $B$ and $V_2$, we have that $V(T) \setminus \lbrace b \rbrace \subseteq A \cup V_1$.
But $V_1$ is an independent set in $G$, so $T$ contains at most one vertex in $V_1$ and hence at least one vertex in $A$.
But since $|B|>|A|$ this implies that $G$ does not
contain a perfect $K_3$-packing. Furthermore, it is easy to check that $G$ has our desired degree sequence.
\end{proof}
Note that Proposition~\ref{extremal1} shows that, if true, Conjecture~\ref{conjy} is close to best possible in the following sense: Given any $1 \leq k < n/3$, there is a graph $G$ on $n$ vertices with degree sequence $d_1\leq \dots \leq d_n$ such that (i) $G$ does not contain the square of a Hamilton cycle and (ii) $G$ satisfies the degree sequence condition in Conjecture~\ref{conjy} except for the terms $d_{k-\eta n}, \dots, d_k$ which only `just' fail to satisfy the desired condition.

At first sight,
one might think that the $\eta n$ term in Theorem~\ref{mainthm} is an artifact of our proof, but in fact it is a feature of the problem: indeed,
it cannot be replaced by $o(\sqrt n)$.
This is shown by an example in Proposition~22 in the arXiv version of~\cite{bkt}.

\section{Preliminaries}\label{prelim}

\subsection{Notation}

We write  $|G|$ for the  order of a graph $G$ and $\delta (G)$
and $\Delta (G)$ for its minimum and maximum degrees respectively. 
The degree of a vertex $x \in V(G)$ is denoted by $d_G(x)$ and its neighbourhood
by $N_G(x)$. 
Given $A \subseteq V(G)$, we write $N_G(A) := \bigcup_{a \in A}{N_G(a)}$. 
We will write $N(A)$, for example, if this is unambiguous.
For $x \in V(G)$ and $A \subseteq V(G)$ we write $d_G(x,A)$ for the number of edges $xy$ in $G$ with $y \in A$.
Given (not necessarily disjoint) $X,Y \subseteq V(G)$, we write $E(G[X,Y])$ for the collection of edges with one endpoint in $X$ and the other endpoint in $Y$.
Define $e_G(X,Y) := |E(G[X,Y])|$.
For each $k \in \mathbb{N}$,  we let $K_k$ denote the complete graph on $k$ vertices.

Given a graph $G$, $X \subseteq V(G)$ and an integer $k \leq |X|$, we define the \emph{$k$-neighbourhood of $X$ in $G$} by
$$
N^k_G(X) := \bigcup_{\stackrel{X' \subseteq X}{|X'|=k}}\bigcap_{x \in X'}N_G(x),
$$
that is, the set of all vertices in $G$ adjacent to at least $k$ members of $X$.
When $X = \lbrace x_1, \ldots, x_\ell \rbrace$, we will also write $N^k_G(x_1, \ldots, x_\ell) := N^k_G(X)$.
Observe that, if $X,Y \subseteq V(G)$ are disjoint, then
\begin{equation}\label{disjoint}
N^{|X|}_G(X) \cap N^{|Y|}_G(Y) = N^{|X|+|Y|}_G(X \cup Y).
\end{equation}
If $H \subseteq G$ we set $N_G^k(H) := N_G^k(V(H))$.
For $A \subseteq V(G)$ we define $N_A^k(X):=N_G^k(X) \cap A$, and (\ref{disjoint}) holds with $G$ replaced by $A$.

Given a graph $G$ and a subset $X \subseteq V(G)$, we write $G[X]$ for the subgraph of $G$ induced by $X$. 
We write $G\setminus X$ for the subgraph of $G$ induced by $V(G)\setminus X$. Given disjoint $X,Y \subseteq V(G)$ we let $G[X,Y]$ denote the graph with vertex set $X\cup Y$ whose edge set consists of all those edges $xy \in E(G)$ with $x \in X$ and $y \in Y$.

Given a function $f: D \rightarrow C$ and $D' \subseteq D$, we write $f(D') := \lbrace f(d) : d \in D' \rbrace \subseteq C$. 


Given a graph $H$, the \emph{square} of $H$ is obtained from $H$ by adding an edge between every pair of vertices of distance two in $H$. In particular,
we say that $P = v_1 \ldots v_k$ is a \emph{square path} if $V(P) = \lbrace v_1, \ldots, v_k \rbrace$ and
$$
E(P) = \lbrace v_iv_{i+1} : 1 \leq i \leq k-1 \rbrace \cup \lbrace v_iv_{i+2} : 1 \leq i \leq k-2 \rbrace.
$$    
So we always implicitly assume that a square path $P$ is equipped with an ordering.
We write $P^*=v_k \dots v_1$ for $P$ `ordered backwards'; so $P \neq P^*$.
Given vertices $x_1, \ldots, x_\ell \in V(G)$
such that $v_1 \ldots v_k x_1 \ldots x_\ell$ is a square path, we sometimes write $Px_1 \ldots x_k := v_1 \ldots v_k x_1 \ldots x_\ell$.
The square path $x_1 \ldots x_k P$ is defined similarly. Given sets $X_1, \dots , X_k$,
we write $P \in X_1 \times \ldots \times X_k$ if $v_i \in X_i$ for all $1 \leq i \leq k$.

Given a square path $P$ and a positive integer $\ell \leq |P|$, 
we say that $[P]_\ell$ is an \emph{$\ell$-segment} if it is an ordered set whose members are $\ell$ consecutive vertices of $P$, endowed with the ordering of $P$.
We usually write $a_1 \ldots a_\ell$ for the $\ell$-segment $(a_1, \ldots, a_\ell)$.
We define the \emph{final $\ell$-segment} $[P]^+_\ell$ of $P$ to be the ordered set of the final $\ell$ vertices in $P$, whose order is inherited from $P$.
The \emph{initial $\ell$-segment} $[P]^-_\ell$ is defined analogously.
We write $(P)^+_\ell, (P)^-_\ell$ for the unordered versions. 
By a slight abuse of notation, we also write $[P]_\ell$ for the square path $P[(P)_\ell]$ and similarly for $[P]^\pm_\ell$.

Throughout we will omit floors and ceilings where the argument is unaffected. The constants in the hierarchies used to state our results are chosen from right to left.
For example, if we claim that a result holds whenever $0<1/n\ll a\ll b\ll c\le 1$ (where $n$ is the order of the graph), then 
there are non-decreasing functions $f:(0,1]\to (0,1]$, $g:(0,1]\to (0,1]$ and $h:(0,1]\to (0,1]$ such that the result holds
for all $0<a,b,c\le 1$ and all $n\in \mathbb{N}$ with $b\le f(c)$, $a\le g(b)$ and $1/n\le h(a)$. 
Note that $a \ll b$ implies that we may assume in the proof that e.g. $a < b$ or $a < b^2$.
Given $n,n' \in \mathbb{N}$ with $n \leq n'$, we write $[n,n'] := \lbrace n, \ldots, n' \rbrace$ and $[n] := [1,n]$.
We write $a\mathbb{N} := \lbrace an : n \in \mathbb{N} \rbrace$.
We also write $a = b \pm \eps$ for $a \in [b - \eps , b + \eps]$.

We will need the following simple consequence of the inclusion-exclusion principle.

\begin{Proposition}\label{2ndnbrhd}
Let $G$ be a graph on $n$ vertices and let $w,x,y \in V(G)$ be distinct. Then
\begin{itemize}
\item[(i)] $|N^2_G(x,y)| \geq d_G(x)+d_G(y) - n$;
\item[(ii)] $|N^2_G(w,x,y)| + |N^3_G(w,x,y)| \geq d_G(w)+d_G(x)+d_G(y) - n$.
\end{itemize}
\end{Proposition}

\begin{proof}
We will only prove (ii).
Observe that
$$
n \geq |N_G(w) \cup N_G(x) \cup N_G(y)| = d_G(w) + d_G(x) + d_G(y) - |N^2_G(w,x,y)| - |N^3_G(w,x,y)|,
$$
as required.
\end{proof}

\subsection{The Regularity and Blow-up lemmas}\label{regblow}
In the proof of Theorem~\ref{mainthm}
we apply Szemer\'edi's Regularity lemma~\cite{szem}. To state it we need some more definitions. We write $d_G(A,B)$ for 
the \emph{density} $\frac{e_G(A,B)}{\vert A \vert \vert B \vert}$ of a bipartite graph $G$ with vertex classes $A$ and $B$. 
Given $\eps > 0$ we say that $G$ is 
$\eps$\emph{-regular} if every $X \subseteq A$ and $Y \subseteq B$ with $\vert X \vert \geq \eps \vert A \vert$ 
and $\vert Y \vert \geq \eps \vert B \vert$ satisfy $\vert d(A,B) - d(X,Y) \vert \leq \eps$.  
Given $\eps, d \in (0,1)$ we say that $G$ is \emph{$(\eps, d)$-regular} if $G$ is 
$\eps$-regular and $d_G(A,B) \geq d$.

\begin{fact}\label{supereasy}
 Given an $(\eps,d)$-regular bipartite graph $G[A,B]$ and $X \subseteq A$ with $|X| \geq \eps|A|$, there are less than $\eps|B|$ vertices in $B$ which have less than $(d-\eps)|X|$ neighbours in $X$.
\end{fact}
We say that $G$ is $(\eps , d )$\emph{-superregular} if both of the following hold:
\begin{itemize} 
\item $G$ is $(\eps,d)$-regular;
\item $d_G(a) \geq d|B|$ and $d_G(b) \geq d|A|$ for all $a \in A$, $b \in B$.
\end{itemize}

We will use the degree form of the Regularity lemma, which can be easily derived from the standard version~\cite{szem}.

\begin{lemma} \label{reg}
\emph{(Degree form of the Regularity lemma)} For every $\eps \in (0,1)$ and every $M^\prime \in \mathbb{N}$ there exist $M, n_0 \in \mathbb{N}$ such that if $G$ is a graph on $n \geq n_0$ vertices and $d \in [0,1]$ is any real number, then there is a partition of the vertex set of $G$ into $V_0, V_1, \ldots , V_{L}$ and a spanning subgraph $G^\prime$ of $G$ such that the following holds:
\begin{itemize}
\item[(i)] $M^\prime \leq L \leq M$;
\item[(ii)] $\vert V_0 \vert \leq \eps n$;
\item[(iii)] $\vert V_1 \vert = \ldots = \vert V_{L} \vert =: m$;
\item[(iv)] $d_{G^\prime}(x) > d_{G}(x) - (d + \eps) n$ for all $x \in V(G)$;
\item[(v)] for all $1 \leq i \leq L$ the graph $G^\prime[V_i]$ is empty;
\item[(vi)] for all $1 \leq i < j \leq L$, $G'[V_i , V_j]$ is $\eps$-regular and has density either 0 or at least $d$.
\end{itemize}
\end{lemma}

We call $V_1 , \ldots , V_{L}$ \emph{clusters}, $V_0$ the \emph{exceptional set} and the vertices in $V_0$ \emph{exceptional vertices}. We refer to $G^\prime$ as the \emph{pure graph}. The last condition of the lemma says that all pairs of clusters are $\eps$-regular (but possibly with different densities). The \emph{reduced graph} $R$ of $G$ with parameters $\eps$, $d$ and $M^\prime$ is the graph whose vertices are $1,\ldots,L$ and in which $ij$ is an edge precisely when $G'[V_i , V_j]$ is $\eps$-regular and has density at least $d$. 

The following simple observation is well known.

\begin{proposition}\label{superslice2} 
Suppose that $0< \eps \le d' \le d \leq 1$ and $ d' \leq d/2, 1/6$. Let $G$ be a bipartite graph with
vertex classes
$A$ and $B$ of size $(1 \pm \eps)m$. Suppose that $G'$ is obtained from $G$ by removing at most
$d'm$
vertices from each vertex class.
\begin{itemize}
\item[(i)]
If $G$ is $(\eps,d)$-regular then $G'$ is $(2d',d-\eps)$-regular.
\item[(ii)] If $G$ is $(\eps,d)$-superregular then $G'$ is $(2d',d-2d')$-superregular.
\end{itemize}
\end{proposition}

\begin{proof}
To prove (i), let $A' \subseteq A$ and $B' \subseteq B$ denote the vertex classes of $G'$.
Since $|A'|,|B'| \geq \eps |A|,\eps |B|$ and $G$ is $(\eps ,d)$-regular we have that 
\begin{align}\label{1}
d_{G'} (A',B')=d_{G} (A',B')= d_G (A,B) \pm \eps
\end{align} and 
\begin{align}\label{2}
d_{G'} (A',B') \geq d-\eps.
\end{align} 
Suppose that $X \subseteq A'$, $Y \subseteq B'$ are such that $|X| \geq 2{d'}|A'| \geq \eps |A|$ and $|Y| \geq 2{d'}|B'|\geq \eps |B| $.
Then as $G$ is $(\eps ,d)$-regular
$d_{G'}(X,Y) = d_G (A,B) \pm \eps $. Together with (\ref{1}) this implies that 
$|d_{G'} (A',B')-d_{G'} (X,Y) | \leq 2 \eps \leq 2 {d'}$, and so by (\ref{2}) $G'$ is $(2{d'},d-\eps)$-regular.
To see (ii), note further that, for any $a \in A'$ we have
$$
d_{G'}(a,B') \geq d|B|-d'm \geq \left(d-\frac{d'}{1-\eps}\right)|B| \geq (d-2d')|B'|,
$$
and similarly $d_{G'}(b,A') \geq (d-2d')|A'|$ for all  $b\in B'$.
\end{proof}

The next proposition appears as Proposition~8 in~\cite{3partite}, and is a slight variant of Proposition~\ref{superslice2}.

\begin{proposition}\label{newsuperslice}
Let $G$ be a graph with $A,B \subseteq V(G)$ disjoint.
Suppose that $G[A,B]$ is $(\eps,d)$-regular and let $A',B' \subseteq V(G)$ be such that $|A \triangle A'| \leq \alpha|A'|$ and $|B \triangle B'| \leq \alpha|B'|$ for some $0 \leq \alpha < 1$.
Then $G[A',B']$ is $(\eps',d')$-regular, with
$$
\eps' := \eps + 6\sqrt{\alpha} \ \ \text{ and }\ \ d' := d-4\alpha.
$$
If, moreover, $G[A,B]$ is $(\eps,d)$-superregular and each vertex $x \in A'$ has at least $d'|B'|$ neighbours in $B'$ and each vertex $x \in B'$ has at least $d'|A'|$ neighbours in $A'$, then $G[A',B']$ is $(\eps',d')$-superregular with $\eps'$ and $d'$ as above.
\end{proposition}

The following lemma is well known in several variations.
The version here is almost identical to Proposition~8 in~\cite{maxplanar}.
\begin{lemma}\label{superslice}
Let $L \in \mathbb{N}$ and suppose that $0 < 1/m \ll 1/L \ll \eps \ll d, 1/\Delta \leq 1$.
Let $R$ be a graph with $V(R) = [L]$.
Let $G$ be a graph with vertex partition $V_1, \ldots, V_{L}$ such that $|V_i|=(1 \pm \eps)m$ for all $1 \leq i \leq L$, and in which $G[V_i,V_j]$ is $(\eps,d)$-regular whenever $ij \in E(R)$.
Let $H$ be a subgraph of $R$ with $\Delta(H) \leq \Delta$.
Then for each $i \in V(H)$, $V_i$ contains a subset $V_i'$ of size $(1-\sqrt{\eps})m$ such that for every edge $ij$ of $H$, the graph $G[V_i',V_j']$ is $(4\sqrt{\eps},d/2)$-superregular.
\end{lemma}

\begin{proof}
Consider an edge $ij \in E(H)$.
By ($\star$), there are less than $\eps|V_i|$ vertices in $V_i$ which have less than $(d-\eps)(1-\eps)m \geq (d-2\eps)m$ neighbours in $V_j$.
So for every vertex $i$ of $H$ we can choose a set $V_i' \subseteq V_i$ of size $(1-\sqrt{\eps})m \leq (1-\eps\Delta)|V_i|$ such that for each neighbour $j$ of $i$ in $H$, all vertices $x \in V_i'$ have at least $(d-2\eps)m$ neighbours in $V_j$.
Proposition~\ref{superslice2}(i) with $2\sqrt{\eps}$ playing the role of $d'$ implies that, for each edge $ij \in E(H)$, $G[V_i',V_j']$ is $(4\sqrt{\eps},d-\eps)$-regular, and hence $(4\sqrt{\eps},d/2)$-regular.
Moreover, for each $x \in V_i'$, $d_{G}(x,V_j') \geq (d-2\eps)m-2\sqrt{\eps}m \geq d(1-\sqrt{\eps})m/2$.
Therefore $G[V_i',V_j']$ is $(4\sqrt{\eps},d/2)$-superregular.
\end{proof}

The following proposition is an easy consequence of $(\eps,d)$-regularity, so we only sketch the proof.

\begin{proposition}\label{triangleembed}
Let $0 < 1/m \ll \eps \ll c,d < 1$.
Let $G$ be a graph with vertex partition $X_1,X_2,X_3$ where $|X_i|=(1 \pm \eps) m$ for all $1 \leq i \leq 3$ and such that $G[X_i,X_j]$ is $(\eps,d)$-regular for all $1 \leq i < j \leq 3$.
For each $i=1,2$, let $A_i,B_i \subseteq X_i$, where $|A_i|,|B_i| \geq cm$.
Let $W \subseteq V(G)$ be such that $|W \cap X_i| \leq \eps m/2$ for all $1 \leq i \leq 3$.
Then there exists a square path $P \in A_1 \times A_2 \times X_3 \times B_1 \times B_2$ with $V(P) \cap W = \emptyset$.
\end{proposition}

\medskip
\noindent
\emph{Sketch proof.}
For $j=1,2$, let $A'_j \subseteq A_j$, $B_j' \subseteq B_j$, $X_3' \subseteq X_3$ be such that $A_j',B_j',W$ are pairwise disjoint, $X_3' \cap W = \emptyset$, and $|A_j'|,|B_j'|,|X_3'| \geq cm/4$.
So $\mathcal{X} := \lbrace A_1',A_2',B_1',B_2',X_3' \rbrace$ is a collection of vertex-disjoint susbets of $V(G)$, and $W \cap Y = \emptyset$ for all $Y \in \mathcal{X}$.
Let $R$ be the graph whose vertices are elements of $\mathcal{X}$, in which, for all $Y,Z \in \mathcal{X}$, we have $YZ \in E(R)$ whenever $G[Y,Z]$ is $(\sqrt{\eps},d/2)$-regular.
Then $R$ is a complete tripartite graph with vertex classes $\lbrace A_1',B_1' \rbrace, \lbrace A_2',B_2' \rbrace, \lbrace X_3' \rbrace$. 
So $R$ contains the square path $P' := A_1'A_2'X_3'B_1'B_2'$.
It is a simple consequence of regularity that therefore $G$ contains a square path $P \in A_1' \times A_2' \times X_3' \times B_1' \times B_2'$.
Note that $P$ has the required properties.
\hfill$\square$

\medskip
Given two graphs $H,G$, we say that a function $\phi : V(H) \rightarrow V(G)$ is a \emph{graph homomorphism} if, for all edges $uv \in E(H)$, we have that $\phi(u)\phi(v) \in E(G)$.
If $\phi$ is injective, then we call it an \emph{embedding}, in which case we write $H \subseteq G$.

We need the following result from~\cite[Lemma 10]{3partite} which, given a homomorphism from a graph $H$ into the reduced graph $R$, allows us to embed $H$ into $G$.
Furthermore, under certain conditions we can guarantee that a small fraction of the vertices of $H$ are mapped into specific sets.
A similar result was first obtained by Chv\'atal, R\"odl, Szemer\'edi and Trotter~\cite{crst}.

\begin{lemma}\label{partialembed} (Partial embedding lemma~\cite{3partite})
Suppose that $L \in \mathbb{N}$ and $0 < 1/m \ll 1/L \ll \eps \ll c \ll d,1/\Delta < 1$.
Let $R$ be a graph with $V(R) = [L]$.
Let $G$ be a graph with vertex partition $V_1, \ldots, V_{L}$ such that $|V_i|=(1 \pm \eps)m$ for all $1 \leq i \leq L$, and in which $G[V_i,V_j]$ is $(\eps,d)$-regular whenever $ij \in E(R)$.

Let $H$ be a graph with vertex partition $X,Y$ and let $f : V(H) \rightarrow V(R)$ be a graph homomorphism (so $f(h)f(h') \in E(R)$ whenever $hh' \in E(H)$).

Then, if $|H| \leq \eps m$ and $\Delta(H) \leq \Delta$, there exists an injective mapping $\tau : X \rightarrow V(G)$ with $\tau(x) \in V_{f(x)}$ for all $x \in X$ and there exist sets $C_y \subseteq V_{f(y)} \setminus \tau(X)$ for all $y \in Y$, so that 
the following hold:
\begin{itemize}
\item[(i)] if $x,x' \in X$ and $xx' \in E(H)$, then $\tau(x)\tau(x') \in E(G)$;
\item[(ii)] for all $y \in Y$ we have that $C_y \subseteq N_G(\tau(x))$ for all $x \in N_H(y) \cap X$;
\item[(iii)] $|C_y| \geq c|V_{f(y)}|$ for all $y \in Y$.
\end{itemize}
\end{lemma}

In its simplest form, the Blow-up lemma of Koml\'os, S\'ark\"ozy and Szemer\'edi~\cite{blowuplem} states that for the purposes of embedding a spanning bipartite graph of bounded degree, a superregular pair behaves like a complete bipartite graph.

\begin{theorem}\label{blowup} (Blow-up lemma~\cite{blowuplem})
For every $d,\Delta,c > 0$ and $k \in \mathbb{N}$ there exist constants $\eps_0$ and $\alpha$ such that the following holds.
Let $n_1, \ldots, n_k$ be positive integers,  $0 < \eps < \eps_0$, and $G$ be a $k$-partite graph with vertex classes $V_1, \ldots, V_k$ where $|V_i|=n_i$ for $i \in [k]$.
Let $J$ be a graph on vertex set $[k]$ such that $G[V_i,V_j]$ is $(\eps,d)$-superregular whenever $ij \in E(J)$.
Suppose that $H$ is a $k$-partite graph with vertex classes $W_1, \ldots, W_k$ of size  at most $n_1, \ldots, n_k$ respectively with $\Delta(H) \leq \Delta$.
Suppose further that there exists a graph homomorphism $\phi : V(H) \rightarrow V(J)$ such that $|\phi^{-1}(i)| \leq n_i$ for every $i \in [k]$.
Moreover, suppose that in each class $W_i$ there is a set of at most $\alpha n_i$ special vertices $y$, each of them equipped with a set $S_y \subseteq V_i$ with $|S_y| \geq cn_i$.
Then there is an embedding of $H$ into $G$ such that every special vertex $y$ is mapped to a vertex in $S_y$.
\end{theorem}%

Notice that in Theorem~\ref{blowup}, $\eps$ depends on $k$ (that is, $\eps \ll 1/k$).
However, in one of our applications of the Blow-up lemma, we do not have this.
The following modified version of the Blow-up lemma is instead applied; it is a very special case of a Blow-up lemma of Csaba~\cite[Lemma 5]{csaba} and also B\"ottcher, Kohayakawa, Taraz and W\"urfl~\cite[Theorem 4]{bktw}.

\begin{theorem}\label{blowup2} (Alternative Blow-up lemma)
For every $d,\Delta,c > 0$ there exists a constant $\eps_0$ such that for all $k \in \mathbb{N}$, there exists $n_0>0$ such that the following holds.
Let $n \geq n_0$ be an integer, $0 < \eps < \eps_0$, and $G$ an $n$-vertex $k$-partite graph with vertex classes $V_1, \ldots, V_k$ where $|V_i|=n/k$ for $i \in [k]$.
Let $J$ be a graph on vertex set $[k]$ with $\Delta(J) \leq \Delta$ such that $G[V_i,V_j]$ is $(\eps,d)$-superregular whenever $ij \in E(J)$.
Suppose that $H$ is a $k$-partite graph with vertex classes $W_1, \ldots, W_k$ each of size  at most $n/k$ with $\Delta(H) \leq \Delta$.
Suppose further that there exists a graph homomorphism $\phi : V(H) \rightarrow V(J)$ such that $|\phi^{-1}(i)| \leq n/k$ for every $i \in [k]$.
Then there is an embedding of $H$ into $G$.
\end{theorem}%

\subsection{$\eta$-good degree sequences}

We will often think of the collection of degrees of the vertices of a graph $G$ as a function $d_G : V(G) \rightarrow \lbrace 0, 1, \ldots, n-1 \rbrace$.
The notation $d_G$ will always be used in this way.
Later we will define a different notion of degree, a function whose image is not necessarily a subset of $\mathbb{N}\cup \{0\}$. 

\begin{definition}\emph{($\eta$-goodness)}
Given $\eta > 0$, $n \in \mathbb{N}$, a  finite set $V$, and a function $d : V \rightarrow \mathbb{R}$, 
let $v_1, \ldots, v_{|V|}$ be an ordering of the elements of $V$ such that $d(v_i) \leq d(v_j)$ whenever $1 \leq i \leq j \leq |V|$.
We say that $d$  is \emph{$(\eta,n)$-good} if $d(v_i) \geq (1/3+\eta)n + i+1$ for all $1 \leq i \leq |V|/3$.
If $V$ is the vertex set of a graph $G$, and $d(v)$ is the degree of $v \in V$ in $G$, we say that $G$ is \emph{$(\eta,n)$-good}.
If $|V|=n$ we say that $G$ is \emph{$\eta$-good}.
\end{definition}

The next simple proposition is very useful.
Its proof follows immediately from the definition of $(\eta,n)$-good, so we omit it.

\begin{proposition}\label{vertexdeg}
Let $\eta > 0$ and $n,k \in \mathbb{N}$.
Let $G$ be a graph on $n$ vertices and let $d : V(G) \rightarrow \mathbb{R}$ be an $(\eta,n)$-good function.
Then the following hold:
\begin{itemize}
\item[(i)] for all $X \subseteq V(G)$ with $|X| \geq n/3$, there exist at least $|X|-n/3$ vertices $x \in X$ with $d(x) \geq (2/3+\eta)n$;
\item[(ii)] for all $X \subseteq V(G)$ with $k \leq |X| \leq n/3$, there exist at least $k$ vertices $x \in X$ with
$d(x) \geq (1/3+\eta)n + |X|-k+2$.
\end{itemize}
\end{proposition}

Given a graph $G$ on $n$ vertices and a set $X \subseteq V(G)$, we write
$$
X_{\eta} := \lbrace x \in X : d_G(x) \geq (2/3+\eta)n \rbrace.
$$
Observe that, if $G$ is $\eta$-good, then
\begin{equation}\label{etagood}
|V(G)_\eta| \geq 2n/3.
\end{equation}

The following proposition collects together some useful facts about $\eta$-good graphs.

\begin{proposition}\label{largeset}
Let $n,k \in \mathbb{N}$ and $\eta >0$ such that $0 \leq 1/n \leq 1/k,\eta \leq 1$.
Let $G$ be an $\eta$-good graph on $n$ vertices and let $X,Y \subseteq V(G)$.
Then the following hold:
\begin{itemize}
\item[(i)] if $X_{\eta} = \emptyset$, then $|X| < n/3$;
\item[(ii)] if $|X_{\eta}| \geq (1/3-\eta/2)n$, then there are no isolated vertices in $G[X_{\eta}]$;
\item[(iii)] if $|X| \geq n/3+k$, then $e_G(X) > k^2/2$;
\item[(iv)] if $X,Y \neq \emptyset$ and $E(G[X,Y])=\emptyset$, then $|X|+|Y| < (2/3-\eta)n$.
\end{itemize}
\end{proposition}

\begin{proof}
First note that (i) and (ii) follow immediately from the definition of $X_{\eta}$ and $\eta$-goodness.

We now prove (iii).
By (i), $|X_\eta| \geq k$.
For each $x \in X_\eta$ we have
$$
d_G(x,X) \geq d_G(x) - (n-|X|) \geq (2/3+\eta)n - (2n/3-k) > k.
$$
So $e(G[X]) \geq \frac{1}{2}\sum_{x \in X}d_G(x,X) > k^2/2$, as required.

To prove (iv),
suppose, without loss of generality, that $|X| \leq |Y|$.
Note that $|X| \leq n/3$ otherwise (i) implies that $X_\eta \neq \emptyset$ and then since $|Y|\geq |X| \geq n/3$ we have that $e_G(X,Y)  > 0$, a contradiction.
Let $x_0 \in X$ be such that $\max_{x \in X}\lbrace d_G(x) \rbrace=d_G(x_0)$.
Proposition~\ref{vertexdeg}(ii) applied with $k := 1$ implies that
$$
(1/3+\eta)n+|X|+1 \leq d_G(x_0) \leq n-|Y|,
$$
and so $|X|+|Y| < (2/3-\eta)n$, as desired.
\end{proof}

We now define what it means for a square path to be head- or tail-heavy.
We will show in Section~\ref{connectingsec} that if $P$ is a tail-heavy square path and $Q$ is a head-heavy square path, then we can `connect' them in an appropriate manner.

\begin{definition}\emph{($\eta$-heaviness)}
Let $n \in \mathbb{N}$ and  $\eta > 0$.
Let $G$ be an $\eta$-good graph on $n$ vertices containing a square path $P$.
We say that $P$ is \emph{$\eta$-tail-heavy} if $[P]^+_2 \in V(G)_\eta \times V(G)_\eta$.
We say that $P$ is \emph{$\eta$-head-heavy} if $[P]^-_2 \in V(G)_\eta \times V(G)_\eta$.
If $P$ is both $\eta$-head- and $\eta$-tail-heavy, we say that it is \emph{$\eta$-heavy}.
We omit the prefix $\eta$- if it is clear from the context.
\end{definition}

Equivalently, $P$ is $\eta$-tail-heavy if $d_G(x) \geq (2/3+\eta)n$ for all $x \in (P)^+_2$, and analogously for head-heavy.
Note that $P$ is $\eta$-tail-heavy if and only if $P^*$ is $\eta$-head-heavy.

\subsection{Core degree}\label{coredegree}
Suppose that $R$ is the reduced graph (with parameters $\eps, d$ and $M'$) of a graph $G$. If $G$ is $\eta$-good then we will show that $R$ `inherits' this property (see Lemma~\ref{reduced}(ii)). Note though that the degree of a vertex $i \in V(R)$ does not provide precise information about the degrees of the vertices $x \in V_i$ in $G$. In particular, if $d$ is small it is possible for $i$ to have `large' degree in $R$ but for \emph{every} vertex $x \in V_i$ to have `small' degree in $G$. In the proof of Theorem~\ref{mainthm} it will be important to ensure that certain clusters contain a `significant' number of vertices of `large' degree in $G$. For this, we introduce the notion of the `core degree' of a cluster in $R$.



Given $0 < \alpha \leq 1$, a graph $G$ on $n$ vertices and a collection $\mathcal{R}$ of disjoint subsets of $V(G)$,
we define the \emph{$\alpha$-core degree $d^\alpha_{\mathcal{R},G}(X)$ of $X \in \mathcal{R}$ (with respect to $G$)} as follows.
Let $d_1 \leq \ldots \leq d_{|X|}$ be the vertex degrees in $G$ of the vertices in $X$.
Then we let
$$
d^\alpha_{\mathcal{R},G}(X) := d_{\lfloor(1-\alpha)|X|\rfloor+1}|\mathcal{R}|/n.
$$
So $d_{\mathcal{R},G}^\alpha(X) \geq k|\mathcal{R}|$ if and only if there are at least $\alpha|X|$ vertices  $x \in X$ with $d_G(x) \geq kn$.
Note that whenever $\alpha' \leq \alpha$ we have that $d^{\alpha'}_{\mathcal{R},G}(X) \geq d^\alpha_{\mathcal{R},G}(X)$ for all $X \in \mathcal{R}$.

Suppose that $\mathcal{R} := \lbrace V_1, \dots , V_k \rbrace$. If $R$ is a graph such that each $j \in V(R)$ corresponds to the set $V_j \in \mathcal{R}$, we define
$$
d^\alpha_{R,G}(j) :=d^\alpha _{\mathcal{R},G}(V_j).
$$
(Typically $R$ will be a reduced graph and so its vertex set $\{1, \dots, k\}$ naturally corresponds to clusters $V_1, \dots ,V_k$ in $G$.)
In this case,
we often think of $d^\alpha_{R,G}$ as a function which maps each vertex of $R$ to some rational less than $|R|$,
and call this function the \emph{$\alpha$-core degree function of $R$ (with respect to $G$)}.

The next lemma shows that the reduced graph $R$ and the function $d^\alpha_{R,G}$ `inherit' the degree sequence of $G$.

\begin{lemma} \label{reduced}
Let $0 < 1/n \ll 1/M^\prime \ll \eps \ll d,\alpha \ll \eta < 1$ and let $G$ be a graph of order $n$ which is $\eta$-good. Apply Lemma~\ref{reg} with parameters $\eps, d$ and $M^\prime$ to obtain a pure graph $G^\prime$ and a reduced graph $R$ of $G$.
Then
\begin{itemize}
\item[(i)] $d_R(j) \geq (1-6d)d^\alpha_{R,G}(j)$ for all $j \in V(R)$;
\item[(ii)] $d^\alpha_{R,G}$ and $R$ are both $(\eta/2,|R|)$-good.
\end{itemize}
\end{lemma}

\begin{proof}
Let $L := |R|$ and $\mathcal{R} := \lbrace V_1, \ldots, V_L \rbrace$ be the clusters of $G$ such that $V_j$ is associated with $j \in V(R)$.
Set $m := |V_1| = \ldots = |V_L|$.
Lemma~\ref{reg}(ii) implies that
\begin{equation}\label{nmL}
mL \leq n = mL + |V_0| \leq mL + \eps n.
\end{equation}
To prove (i), fix $j \in V(R)$ and let $D:= d^\alpha_{R,G}(j)$.
Note first that $D \geq \delta(G)L/n \geq (1/3+\eta)L$ since $G$ is $\eta$-good.
By the definition of $d_{R,G}^\alpha$, there is a set $X_j \subseteq V_j$ such that $|X_j| \geq \alpha m$ and $d_G(x) \geq Dn/L$ for all $x \in X_j$.
So by Lemma~\ref{reg}(iv), $d_{G'}(x) > Dn/L-(d+\eps)n$.
Given any vertex $x \in X_j$, the number of clusters $V_i \in \mathcal{R}$ containing a neighbour of $x$ in $G'$ is at least 
$$
\frac{Dn/L - (d+2\eps)n}{m} \stackrel{(\ref{nmL})}{\geq} D - \frac{(d+2\eps)n}{m} \stackrel{(\ref{nmL})}{\geq} D - 2dL \geq D(1-6d).
$$
(In the last inequality we used that $D \geq (1/3+\eta)L$.)
Lemma~\ref{reg}(vi) implies that $j$ is adjacent to each of the vertices corresponding to these clusters in $R$.
So $d_R(j) \geq D(1-6d)$, proving (i).

To prove (ii), fix $1 \leq i \leq L/3$ and $\mathcal{X} \subseteq V(R)$ with $|\mathcal{X}| = i$.
Let $X' := \bigcup_{j \in \mathcal{X}}V_j \subseteq V(G)$.
Then $|X'| = im \leq Lm/3 \leq n/3$ by (\ref{nmL}).
Since $G$ is $\eta$-good, Proposition~\ref{vertexdeg}(ii) implies that there is a subset $Y$ of $X'$ with $|Y| \geq \alpha i m $ such that
$\min_{y \in Y} \lbrace d_G(y)\rbrace \geq (1/3 + \eta)n + (1-\alpha)im+2$.
Observe further that there exists some $j \in \mathcal{X}$ such that $|Y \cap V_j| \geq \alpha m$.
Thus
\begin{eqnarray*}
d^\alpha_{R,G}(j) &\geq& \min_{y \in Y}\lbrace d_G(y) \rbrace \frac{L}{n} \geq \left(\frac{1}{3}+\eta\right)L+ \frac{(1-\alpha)imL}{n}\\
&\stackrel{(\ref{nmL})}{\geq}& \left(\frac{1}{3}+\eta\right)L + (1-\alpha)(1-\eps)i \geq \left(\frac{1}{3} + \frac{2\eta}{3}\right)L + i+1.
\end{eqnarray*}
Since $\mathcal{X}$ was arbitrary, this proves that $d^\alpha_{R,G}$ is $(2\eta/3,L)$-good and hence $(\eta/2,L)$-good.

Let $1\leq i\leq L/3$.
Now, by (i), the vertex $j_i$ of $R$ with $i$th smallest degree satisfies
$$
d_R(j_i) \geq (1-6d)\left(\left(\frac{1}{3}+\frac{2\eta}{3}\right)L+i+1\right) \geq \left(\frac{1}{3}+\frac{\eta}{2}\right)L + i+1.
$$
So $R$ is $(\eta/2,L)$-good, completing the proof of (ii).
\end{proof}

The next proposition shows that, given an $(\eta,n)$-good function $d$, after arbitrarily shrinking the domain of $d$ a little or by slightly reducing each of the values that $d$ takes, the function that remains is $(\eta/2,n)$-good.

\begin{proposition}\label{stoneage}
Let $n \in \mathbb{N}$ and $\eta > 0$ such that $0 < 1/n \ll \eta < 1$.
Let $V$ be a set of order $n$ and let $d : V \rightarrow \mathbb{R}$ be $(\eta,n)$-good.
Let $V' \subseteq V$ with $|V'| \geq (1-\eta/4)n$ and let $d' : V' \rightarrow \mathbb{R}$ be such that $d'(v) \geq d(v) - \eta n/4$ for all $v \in V'$.
Then $d'$ is $(\eta/2,n)$-good.
In particular, any graph obtained from an $\eta$-good graph $G$ on $n$ vertices by removing at most $\eta n/4$ vertices and $\eta n/4$ edges from each vertex is $(\eta/2,n)$-good. 
\end{proposition}

\begin{proof}
Let $v_1, \ldots, v_n$ be an ordering of $V$ such that $d(v_i) \leq d(v_j)$ whenever $1 \leq i \leq j \leq n$.
Then $d(v_i) \geq (1/3+\eta)n + i+1$ for all $1 \leq i \leq n/3$.

Let $i_1, \ldots, i_k$ be the subsequence of $1, \ldots, n$ corresponding to the vertices in $V'$.
So $k := |V'| \geq (1-\eta/4)n$.
Let $1 \leq j \leq k/3$ be arbitrary.
Since $j \leq k/3\leq n/3$ and $i_j \geq j$ we have
\begin{align*}
d'(v_{i_j}) &\geq d(v_{i_j})-\eta n/4 \geq d(v_{j})-\eta n/4 \geq (1/3+\eta)n + (j+1)-\eta n/4\\
&\geq (1/3+\eta/2)n + j+1.
\end{align*}
This implies that $d'$ is $(\eta /2,n)$-good.
The final assertion follows by taking $d := d_G$.
\end{proof}


\section{An almost perfect packing of heavy square paths}\label{sectionalmost}

The aim of this section is to prove the following lemma, which ensures that every sufficiently large $\eta$-good graph $G$ on $n$ vertices contains an almost perfect packing of square paths, and the number of these paths is bounded. As mentioned in Section~\ref{sketch}, a relatively simple application of Lemma~\ref{reg} and Theorems~\ref{blowup} and~\ref{trianglepack}
can achieve this.
However, we also require that the first and last two vertices of each of these paths have degree at least $(2/3+\eta)n$ in $G$, for which considerably more work is needed.
This property is crucial when, in Section~\ref{connectingsec}, we connect these paths to obtain an almost spanning square cycle.

\begin{lemma}\label{almostpath}
Let $0<\eps, \eta \ll 1$.
Then there exist $n_0,M \in \mathbb{N}$ such that the following holds.
For every $\eta$-good graph $G$ on $n \geq n_0$ vertices, $G$ contains a collection $\mathcal{P}$ of at most $M$ vertex-disjoint $\eta$-heavy square paths such that
$\sum_{P \in \mathcal{P}}|P| \geq (1-\eps)n$.
\end{lemma}

To prove Lemma~\ref{almostpath},
we will use the following result of the second author~\cite{triangle} which guarantees a perfect triangle packing in a sufficiently large $\eta$-good graph.

\begin{theorem}\label{trianglepack}\cite{triangle}
For every $\eta >0$, there exists $n_0 \in \mathbb{N}$ such that every $\eta$-good graph $G$ on $n \geq n_0$ vertices contains a perfect $K_3$-packing.
\end{theorem}
Theorem~\ref{trianglepack} is a special case of a more general result from~\cite{triangle}  on a degree sequence condition that forces a graph to contain a  perfect $H$-packing for arbitrary $H$.

To find a bounded number of vertex-disjoint square paths which together cover almost every vertex of $G$, we apply Szemer\'edi's Regularity lemma to $G$ and then apply Theorem~\ref{trianglepack} to the reduced graph $R$ of $G$ to find a perfect triangle packing $(T_j)_j$ in $R$.
Then we use the Blow-up lemma  to find a square path in $G$ for each triangle $T_j$ that covers almost all of the vertices in the clusters of $T_j$.

However, to guarantee that our paths are $\eta$-heavy, more work is needed.
We extend each triangle $T_j$ in $R$ into two `folded paths' $F_j, F' _j$. A folded path is essentially a sequence of triangles such that the $i$th triangle shares exactly two vertices with the $(i-1)$th triangle.
A folded path is therefore a generalisation of a square path.
We choose both  $F_j$ and $F' _j$ so that their final two clusters each contain many vertices of degree at least $(2/3+\eta)n$. Further, the initial triangle of both $F_j$ and $F' _j$ 
is $T_j$.
(These folded paths $F_j,F_j'$ are obtained by applying Lemma~\ref{findfpath}.)
These properties will allow us to find a square path $Q_j$ in $G$ so that:
\begin{itemize}
\item[(i)] $Q_j$ only contains vertices from the clusters in $F_j$ and $F' _j$;
\item[(ii)] $Q_j$ contains most of the vertices from the clusters in $T_j$;
\item[(iii)] $Q_j$ is $\eta$-heavy.
\end{itemize} 
To ensure (ii), we wind around the clusters of $T_j$, using almost all of their vertices, to find the large central part of $Q_j$.
To ensure (iii), we extend this square path in both directions, using a small number of additional vertices in clusters of $F_j$ and $F_j'$.
Clearly (i) is also satisfied.
Note that for distinct $T_j$, $T_{j'}$ in $R$, the folded paths $F_j$, $F' _j$, $F_{j'}$ and $F' _{j'}$ may intersect. Thus care is needed to ensure the square paths $Q_j, Q_{j'}$ constructed are vertex-disjoint: this is possible since only a small number of vertices in each $Q_j$ lie outside of the clusters of $T_j$.
The $Q_j$ are constructed in Lemmas~\ref{babyhorror} and~\ref{horror}.



\subsection{Folded paths}

Here we define a structure -- a `folded path' -- which will be useful when embedding square paths.
Indeed, if the reduced graph of $G$ contains a short folded path $F$, we can embed a short square path into $G$ using only vertices lying in the clusters which form the vertex set of $F$.

\usetikzlibrary{calc}

\newcommand*\rowsa{3}
\newcommand*\rowsb{2}
\newcommand*\rowsc{1}
\newcommand*\rowsd{3}

\begin{center}
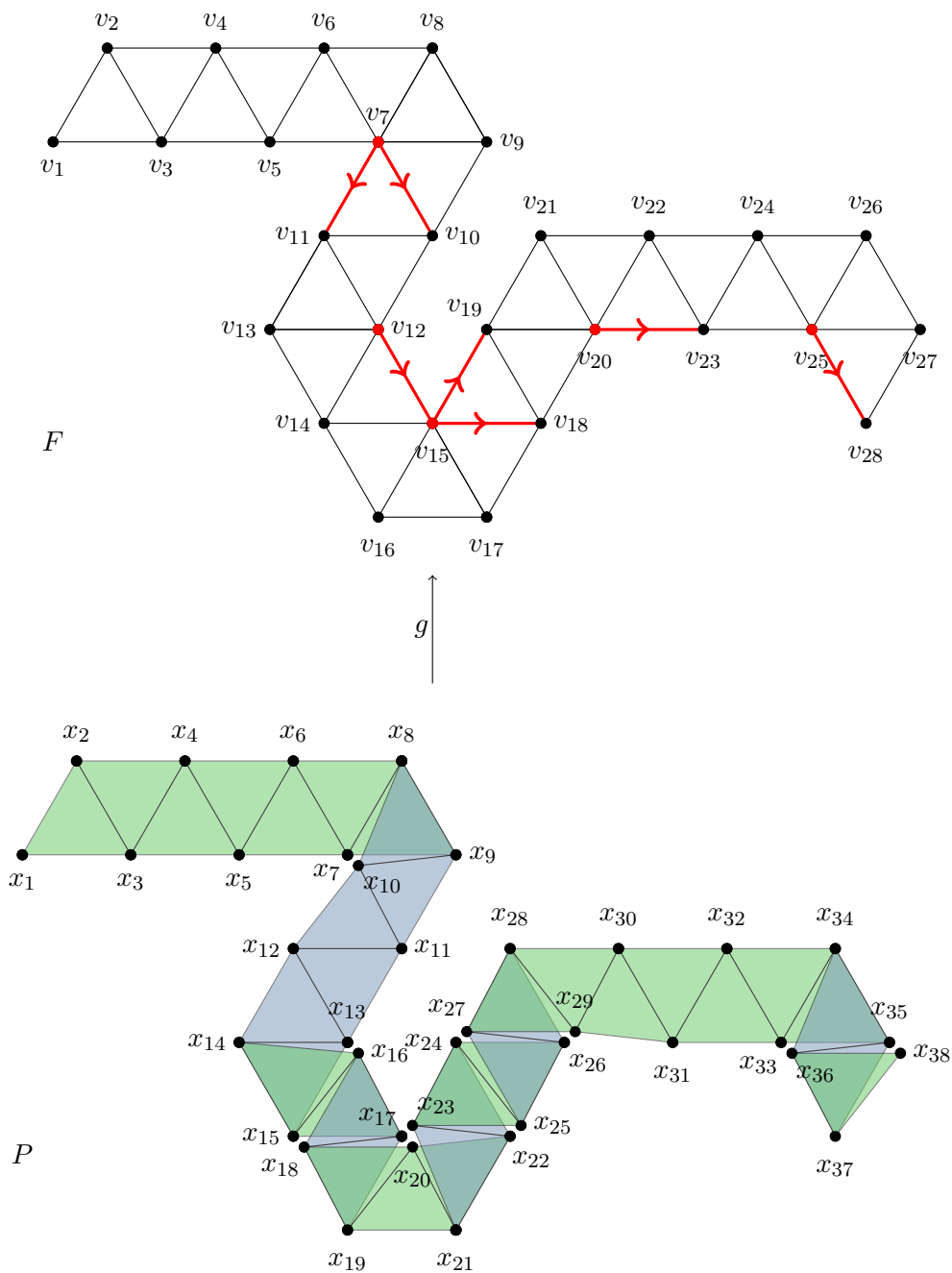
\begin{figure}

\begin{tikzpicture}[every node/.style={draw,circle,fill=black,inner sep=0.5mm},scale=1.5]

\begin{scope}
    \foreach \row in {0, 1, ...,\rowsa} {

    \node[] (\row one) at ($\row*(1, 0)$) {};    
    \node[] (\row two) at ($\row*(1,0)+(0.5,{0.5*sqrt(3)})$) {};
        
        \draw ($\row*(1, 0)$) -- ($\row*(1,0)+(0.5,{0.5*sqrt(3)})$);        \draw (${\row+1}*(1, 0)$) -- (${\row+1}*(1,0)+(-0.5,{0.5*sqrt(3)})$);
    }
    \foreach \row in {0, 1} {
        \draw ($\row*(0.5, {0.5*sqrt(3)})$) -- ($({\rowsa+1},0)+\row*(-0.5, {0.5*sqrt(3)})$);
        }
        
\node[draw=none,label=below:$v_{1}$] at (0one) {};        
\node[draw=none,label=below:$v_3$] at (1one) {};
\node[draw=none,label=below:$v_5$] at (2one) {};
\node[draw=none,label=above:$v_7$] at (3one) {};

\node[draw=none,label=above:$v_2$] at (0two) {};        
\node[draw=none,label=above:$v_4$] at (1two) {};
\node[draw=none,label=above:$v_6$] at (2two) {};
\node[draw=none,label=above:$v_8$] at (3two) {};

\end{scope}

\begin{scope}[shift={($\rowsa*(1,0)+(0.5,{0.5*sqrt(3)})$)},rotate=240]

    \foreach \row in {0, 1, ...,\rowsb} {

    \node[] (\row three) at ($\row*(1, 0)$) {};
    \node[] (\row four) at ($\row*(1,0)+(0.5,{0.5*sqrt(3)})$) {};

        \draw ($\row*(1, 0)$) -- ($\row*(1,0)+(0.5,{0.5*sqrt(3)})$);        \draw (${\row+1}*(1, 0)$) -- (${\row+1}*(1,0)+(-0.5,{0.5*sqrt(3)})$);
    }
    \foreach \row in {0, 1} {
        \draw ($\row*(0.5, {0.5*sqrt(3)})$) -- ($({\rowsb+1},0)+\row*(-0.5, {0.5*sqrt(3)})$);
        }

\node[color=red,fill=red] at (1,0) {};

\node[draw=none,label=left:$v_{11}$] at (2three) {};

\node[draw=none,label=right:$v_{9}$] at (0four) {};        
\node[draw=none,label=right:$v_{10}$] at (1four) {};
\node[draw=none,label=right:$v_{12}$] at (2four) {};

\begin{scope}[shift={(\rowsb+1,0)},rotate=60]

    \foreach \row in {0, 1, ...,\rowsb} {

    \node[] (\row five) at ($\row*(1, 0)$) {};
    \node[] (\row six) at ($\row*(1,0)+(0.5,{0.5*sqrt(3)})$) {};

        \draw ($\row*(1, 0)$) -- ($\row*(1,0)+(0.5,{0.5*sqrt(3)})$);        \draw (${\row}*(1, 0)$) -- (${\row}*(1,0)+(-0.5,{0.5*sqrt(3)})$);
    }
    
        \draw (0,0) -- (({\rowsb},0);
        \draw ($(0.5, {0.5*sqrt(3)})$) -- ($({\rowsb+1},0)+(-0.5, {0.5*sqrt(3)})$);

\node[color=red,fill=red] at ($(0,0)+(0.5,{0.5*sqrt(3)})$) {};

\node[draw=none,label=below:$v_{15}$] at (1six) {};
\node[draw=none,label=below:$v_{17}$] at (2six) {};

\node[draw=none,label=left:$v_{13}$] at (0five) {};        
\node[draw=none,label=left:$v_{14}$] at (1five) {};
\node[draw=none,label=below:$v_{16}$] at (2five) {};

\begin{scope}[shift={($\rowsb*(1,0)+(0.5,{0.5*sqrt(3)})$)},rotate=120]

    \foreach \row in {0, 1, ...,\rowsc} {
    
        \node[] (\row seven) at ($\row*(1, 0)$) {};
    \node[] (\row eight) at ($\row*(1,0)+(0.5,{0.5*sqrt(3)})$) {};
    
        \draw ($\row*(1, 0)$) -- ($\row*(1,0)+(0.5,{0.5*sqrt(3)})$);        \draw (${\row+1}*(1, 0)$) -- (${\row+1}*(1,0)+(-0.5,{0.5*sqrt(3)})$);
    }
    \foreach \row in {0, 1} {
        \draw ($\row*(0.5, {0.5*sqrt(3)})$) -- ($({\rowsc+1},0)+\row*(-0.5, {0.5*sqrt(3)})$);
        }

\node[color=red,fill=red] at ($(0,0)+(0.5,{0.5*sqrt(3)})$) {};

\node[draw=none,label=right:$v_{18}$] at (1seven) {};

\node[draw=none,label=135:$v_{19}$] at (1eight) {};

\begin{scope}[shift={($\rowsc*(1,0)+(0.5,{0.5*sqrt(3)})$)},rotate=-60]

    \foreach \row in {0, 1, ...,\rowsd} {

    \node[] (\row nine) at ($\row*(1, 0)+(1,0)$) {};
    \node[] (\row ten) at ($\row*(1,0)+(0.5,{0.5*sqrt(3)})$) {};
    
        \draw ($\row*(1, 0)$) -- ($\row*(1,0)+(0.5,{0.5*sqrt(3)})$);        \draw (${\row+1}*(1, 0)$) -- (${\row+1}*(1,0)+(-0.5,{0.5*sqrt(3)})$);
    }
    \foreach \row in {0, 1} {
        \draw ($\row*(0.5, {0.5*sqrt(3)})$) -- ($({\rowsd+1},0)+\row*(-0.5, {0.5*sqrt(3)})$);
        }
        
\begin{scope}[shift={($(0,{-1*sqrt(3)})$)}]
        
  \node[label=below:$v_{28}$] (final) at ($\rowsd*(1,0)+(0.5,{0.5*sqrt(3)})$) {};

\end{scope}

\draw (2nine) -- (final) -- (3nine);

\node[draw=none,label=below:$v_{20}$] at (0nine) {};
\node[draw=none,label=below:$v_{23}$] at (1nine) {};
\node[draw=none,label=below:$v_{25}$] at (2nine) {};
\node[draw=none,label=below:$v_{27}$] at (3nine) {};

\node[draw=none,label=above:$v_{21}$] at (0ten) {};
\node[draw=none,label=above:$v_{22}$] at (1ten) {};
\node[draw=none,label=above:$v_{24}$] at (2ten) {};
\node[draw=none,label=above:$v_{26}$] at (3ten) {};


        \node[color=red,fill=red] at (1,0) {};
                
\end{scope}

\end{scope}

\end{scope}

\end{scope}

\usetikzlibrary{decorations.markings}

\tikzset{->-/.style={decoration={
  markings,
  mark=at position #1 with {\arrow[scale=1.5]{>}}},postaction={decorate}}}

        \node[color=red,fill=red] at (2nine) {};

\draw[very thick,color=red,->-=.5] (3one) -- (1four);
\draw[very thick,color=red,->-=.5] (3one) -- (2three);
\draw[very thick,color=red,->-=.5] (2four) -- (1six);
\draw[very thick,color=red,->-=.5] (1six) -- (1seven);
\draw[very thick,color=red,->-=.5] (1six) -- (1eight);
\draw[very thick,color=red,->-=.5] (0nine) -- (1nine);
\draw[very thick,color=red,->-=.5] (2nine) -- (final);

\node[draw=none,fill=none,label=$F$] at (0,-3) {};

\draw[<-] (3.5,-4) -- (3.5,-5);

\node[draw=none,fill=none,label=$g$] at (3.4,-4.7) {};
\node[draw=none,fill=none] at (3.5,-5.2) {};

\end{tikzpicture}

\begin{tikzpicture}[every node/.style={draw,circle,fill=black,inner sep=0.5mm},scale=1.5]

\begin{scope}
    \foreach \row in {0, 1, ...,\rowsa} {

    \node[] (\row one) at ($\row*(1, 0)$) {};    
    \node[] (\row two) at ($\row*(1,0)+(0.5,{0.5*sqrt(3)})$) {};
      }

\end{scope}

\begin{scope}[shift={($\rowsa*(1,0)+(0.5,{0.5*sqrt(3)})$)},rotate=240]

    \foreach \row in {0, 1, ...,\rowsb} {

    \node[] (\row three) at ($\row*(1, 0)$) {};
    \node[] (\row four) at ($\row*(1,0)+(0.5,{0.5*sqrt(3)})$) {};
}

\begin{scope}[shift={(\rowsb+1,0)},rotate=60]

    \foreach \row in {0, 1, ...,\rowsb} {

    \node[] (\row five) at ($\row*(1, 0)$) {};
    \node[] (\row six) at ($\row*(1,0)+(0.5,{0.5*sqrt(3)})$) {};
}

\begin{scope}[shift={($\rowsb*(1,0)+(0.5,{0.5*sqrt(3)})$)},rotate=120]

    \foreach \row in {0, 1, ...,\rowsc} {
    
        \node[] (\row seven) at ($\row*(1, 0)$) {};
    \node[] (\row eight) at ($\row*(1,0)+(0.5,{0.5*sqrt(3)})$) {};
 }

\begin{scope}[shift={($\rowsc*(1,0)+(0.5,{0.5*sqrt(3)})$)},rotate=-60]

    \foreach \row in {0, 1, ...,\rowsd} {

    \node[] (\row nine) at ($\row*(1, 0)+(1,0)$) {};
    \node[] (\row ten) at ($\row*(1,0)+(0.5,{0.5*sqrt(3)})$) {};

  }

  \begin{scope}[shift={($(0,{-1*sqrt(3)})$)}]
        
  \node[] (final) at ($\rowsd*(1,0)+(0.5,{0.5*sqrt(3)})$) {};

\end{scope}

              
\end{scope}

\end{scope}

\end{scope}

\end{scope}

\node[] (3one1) at ($(3one)+(0.1,-0.1)$) {};
\node[] (2four1) at ($(2four)+(0.1,-0.1)$) {};
\node[] (1five1) at ($(1five)+(0.1,-0.1)$) {};
\node[] (1six1) at ($(1six)+(0.1,-0.1)$) {};
\node[] (1six2) at ($(1six)+(0.1,0.1)$) {};
\node[] (1seven1) at ($(1seven)+(0.1,0.1)$) {};
\node[] (1eight1) at ($(1eight)+(0.1,0.1)$) {};
\node[] (0nine1) at ($(0nine)+(0.1,0.1)$) {};
\node[] (2nine1) at ($(2nine)+(0.1,-0.1)$) {};
\node[] (3nine1) at ($(3nine)+(0.1,-0.1)$) {};

 \definecolor{r1}{RGB}{100,200,100}
 \definecolor{r2}{RGB}{120,150,184}
 
\draw[fill=r1,opacity=0.5] ($(0one)$) -- ($(0two)$) -- ($(1one)$) -- ($(0one)$);
\draw[fill=r1,opacity=0.5] ($(1one)$) -- ($(1two)$) -- ($(0two)$) -- ($(1one)$);
\draw[fill=r1,opacity=0.5] ($(1one)$) -- ($(1two)$) -- ($(2one)$) -- ($(1one)$);
\draw[fill=r1,opacity=0.5] ($(2one)$) -- ($(2two)$) -- ($(1two)$) -- ($(2one)$);
\draw[fill=r1,opacity=0.5] ($(2one)$) -- ($(2two)$) -- ($(3one)$) -- ($(2one)$);
\draw[fill=r1,opacity=0.5] ($(3one)$) -- ($(3two)$) -- ($(2two)$) -- ($(3one)$);
\draw[fill=r1,opacity=0.5] ($(3one)$) -- ($(3two)$) -- ($(0four)$) -- ($(3one)$);
\draw[fill=r2,opacity=0.5] ($(3two)$) -- ($(3one1)$) -- ($(0four)$) -- ($(3two)$);
\draw[fill=r2,opacity=0.5] ($(3one1)$) -- ($(0four)$) -- ($(1four)$) -- ($(3one1)$);
\draw[fill=r2,opacity=0.5] ($(3one1)$) -- ($(1four)$) -- ($(2three)$) -- ($(3one1)$);
\draw[fill=r2,opacity=0.5] ($(1four)$) -- ($(2three)$) -- ($(2four)$) -- ($(1four)$);
\draw[fill=r2,opacity=0.5] ($(2three)$) -- ($(2four)$) -- ($(0five)$) -- ($(2three)$);
\draw[fill=r2,opacity=0.5] ($(2four)$) -- ($(0five)$) -- ($(1five)$) -- ($(2four)$);
\draw[fill=r1,opacity=0.5] ($(0five)$) -- ($(1five)$) -- ($(2four1)$) -- ($(0five)$);
\draw[fill=r1,opacity=0.5] ($(1five)$) -- ($(2four1)$) -- ($(1six)$) -- ($(1five)$);
\draw[fill=r2,opacity=0.5] ($(2four1)$) -- ($(1six)$) -- ($(1five1)$) -- ($(2four1)$);
\draw[fill=r2,opacity=0.5] ($(1six)$) -- ($(1five1)$) -- ($(2five)$) -- ($(1six)$);
\draw[fill=r1,opacity=0.5] ($(1five1)$) -- ($(2five)$) -- ($(1six1)$) -- ($(1five1)$);
\draw[fill=r1,opacity=0.5] ($(2five)$) -- ($(1six1)$) --($(2six)$) -- ($(2five)$);
\draw[fill=r1,opacity=0.5] ($(1six1)$) -- ($(2six)$) -- ($(1seven)$) -- ($(1six1)$);
\draw[fill=r2,opacity=0.5] ($(2six)$) -- ($(1seven)$) -- ($(1six2)$) -- ($(2six)$);
\draw[fill=r2,opacity=0.5] ($(1seven)$) -- ($(1six2)$) -- ($(1eight)$) -- ($(1seven)$);
\draw[fill=r1,opacity=0.5] ($(1six2)$) -- ($(1eight)$) -- ($(1seven1)$) -- ($(1six2)$);
\draw[fill=r1,opacity=0.5] ($(1eight)$) -- ($(1seven1)$) -- ($(0nine)$) -- ($(1eight)$);
\draw[fill=r2,opacity=0.5] ($(1seven1)$) -- ($(0nine)$) -- ($(1eight1)$) -- ($(1seven1)$);
\draw[fill=r2,opacity=0.5] ($(0nine)$) -- ($(1eight1)$) -- ($(0ten)$) -- ($(0nine)$);
\draw[fill=r1,opacity=0.5] ($(1eight1)$) -- ($(0ten)$) -- ($(0nine1)$) -- ($(1eight1)$);
\draw[fill=r1,opacity=0.5] ($(0ten)$) -- ($(0nine1)$) -- ($(1ten)$) -- ($(0ten)$);
\draw[fill=r1,opacity=0.5] ($(0nine1)$) -- ($(1ten)$) -- ($(1nine)$) -- ($(0nine1)$);
\draw[fill=r1,opacity=0.5] ($(1ten)$) -- ($(1nine)$) -- ($(2ten)$) -- ($(1ten)$);
\draw[fill=r1,opacity=0.5] ($(1nine)$) -- ($(2ten)$) -- ($(2nine)$) -- ($(1nine)$);
\draw[fill=r1,opacity=0.5] ($(2ten)$) -- ($(2nine)$) -- ($(3ten)$) -- ($(2ten)$);
\draw[fill=r1,opacity=0.5] ($(2nine)$) -- ($(3ten)$) -- ($(3nine)$) -- ($(2nine)$);
\draw[fill=r2,opacity=0.5] ($(3ten)$) -- ($(3nine)$) -- ($(2nine1)$) -- ($(3ten)$);

\draw[fill=r2,opacity=0.5] ($(2nine1)$) -- ($(3nine)$) -- ($(final)$) -- ($(2nine1)$);
\draw[fill=r1,opacity=0.5] ($(2nine1)$) -- ($(3nine1)$) -- ($(final)$) -- ($(2nine1)$);

\begin{scope}
    \foreach \row in {0, 1, ...,\rowsa} {

    \node[] (\row one) at ($\row*(1, 0)$) {};    
    \node[] (\row two) at ($\row*(1,0)+(0.5,{0.5*sqrt(3)})$) {};
      } 
        
\node[draw=none,label=below:$x_1$] at (0one) {};        
\node[draw=none,label=below:$x_3$] at (1one) {};
\node[draw=none,label=below:$x_5$] at (2one) {};
\node[draw=none,label=200:$x_7$] at (3one) {};

\node[draw=none,label=above:$x_2$] at (0two) {};        
\node[draw=none,label=above:$x_4$] at (1two) {};
\node[draw=none,label=above:$x_6$] at (2two) {};
\node[draw=none,label=above:$x_8$] at (3two) {};

\end{scope}

\begin{scope}[shift={($\rowsa*(1,0)+(0.5,{0.5*sqrt(3)})$)},rotate=240]
    \foreach \row in {0, 1, ...,\rowsb} {
    \node[] (\row three) at ($\row*(1, 0)$) {};
    \node[] (\row four) at ($\row*(1,0)+(0.5,{0.5*sqrt(3)})$) {};
}
\node[draw=none,label=left:$x_{12}$] at (2three) {};
\node[draw=none,label=right:$x_{9}$] at (0four) {};        
\node[draw=none,label=right:$x_{11}$] at (1four) {};
\node[draw=none,label=above:$x_{13}$] at (2four) {};
\begin{scope}[shift={(\rowsb+1,0)},rotate=60]
    \foreach \row in {0, 1, ...,\rowsb} {
    \node[] (\row five) at ($\row*(1, 0)$) {};
    \node[] (\row six) at ($\row*(1,0)+(0.5,{0.5*sqrt(3)})$) {};
}
\node[draw=none,label=170:$x_{17}$] at (1six) {};
\node[draw=none,label=below:$x_{21}$] at (2six) {};
\node[draw=none,label=left:$x_{14}$] at (0five) {};        
\node[draw=none,label=left:$x_{15}$] at (1five) {};
\node[draw=none,label=below:$x_{19}$] at (2five) {};
\begin{scope}[shift={($\rowsb*(1,0)+(0.5,{0.5*sqrt(3)})$)},rotate=120]
    \foreach \row in {0, 1, ...,\rowsc} {    
        \node[] (\row seven) at ($\row*(1, 0)$) {};
    \node[] (\row eight) at ($\row*(1,0)+(0.5,{0.5*sqrt(3)})$) {};
 }   
\node[draw=none,label=-60:$x_{22}$] at (1seven) {};
\node[draw=none,label=180:$x_{24}$] at (1eight) {};
\begin{scope}[shift={($\rowsc*(1,0)+(0.5,{0.5*sqrt(3)})$)},rotate=-60]
    \foreach \row in {0, 1, ...,\rowsd} {
    \node[] (\row nine) at ($\row*(1, 0)+(1,0)$) {};
    \node[] (\row ten) at ($\row*(1,0)+(0.5,{0.5*sqrt(3)})$) {};
  }         
\node[draw=none,label=-20:$x_{26}$] at (0nine) {};
\node[draw=none,label=below:$x_{31}$] at (1nine) {};
\node[draw=none,label=225:$x_{33}$] at (2nine) {};
\node[draw=none,label=-25:$x_{36}$] at (2nine1) {};
\node[draw=none,label=below:$x_{37}$] at (final) {};
\node[draw=none,label=right:$x_{38}$] at (3nine1) {};
\node[draw=none,label=above:$x_{35}$] at (3nine) {};
\node[draw=none,label=above:$x_{28}$] at (0ten) {};
\node[draw=none,label=above:$x_{30}$] at (1ten) {};
\node[draw=none,label=above:$x_{32}$] at (2ten) {};
\node[draw=none,label=above:$x_{34}$] at (3ten) {};
\node[draw=none,label=-5:$x_{10}$] at (3one1) {};
\node[draw=none,label=right:$x_{16}$] at (2four1) {};
\node[draw=none,label=200:$x_{18}$] at (1five1) {}; 
\node[draw=none,label=below:$x_{20}$] at (1six1) {};
\node[draw=none,label=15:$x_{23}$] at (1six2) {}; 
\node[draw=none,label=right:$x_{25}$] at (1seven1) {};
\node[draw=none,label=100:$x_{27}$] at (1eight1) {};
\node[draw=none,label=above:$x_{29}$] at (0nine1) {};
                   
\end{scope}

\end{scope}

\end{scope}

\end{scope}

\node[draw=none,fill=none,label=$P$] at (0,-3) {};

\end{tikzpicture}
\caption{A folded path $F$ with $|F|=28$, and a square path $P$ with $|P|=38$, drawn to show a homomorphism $g : [|P|] \rightarrow [|F|]$ (so, for example, $g(\lbrace 17,20,23\rbrace)=\lbrace 15 \rbrace$), as described in Lemma~\ref{foldedpath}. The final two vertices map to each other: $g(37) = 28$ and $g(38)=27$. The values $k_i$ for $F$ are represented by arrows: an arrow is drawn from $v_{k_i}$ to $v_i$ if and only if $k_i = k_{i-1}$ (so if a vertex $j$ has no incoming arrow, $k_j=j-2$).}
\label{foldedpathfig}
\end{figure}
\end{center}


\begin{definition} \emph{(Folded paths)}
We say a graph $F$ is a \emph{folded path} if there exists an ordered sequence $v_1, \ldots, v_n$ of distinct vertices and integers $k_3, \ldots, k_n$ such that
\begin{itemize}
\item $V(F) = \lbrace v_1, \ldots, v_n \rbrace$;
\item $k_3 := 1$ and $k_i \in \lbrace i-2, k_{i-1} \rbrace$ for $4 \leq i \leq n$;
\item $E(F) = \lbrace v_1v_2 \rbrace \cup \lbrace v_iv_{k_i}, v_iv_{i-1} : 3 \leq i \leq n \rbrace$.
\end{itemize}

\end{definition}

We implicitly assume that a folded path is equipped with ordered sequences $v_1, \ldots, v_n$ and $k_3, \ldots, k_n$.
We will sometimes write $F = v_1 \ldots v_n$, and say that
$k_3, \ldots, k_n$ is the \emph{ordering} of $F$.
Observe that
$k_3, \ldots, k_n$ is a non-decreasing sequence and $k_i \leq i-2$ for all $i \geq 3$.

A folded path $F$ is a generalisation of a square path.
Indeed, the special case when the ordering of $F$ is $1, \ldots, n-2$ (i.e. $k_i = i-2$ for all $i \geq 3$) corresponds to the square path on $n$ vertices.
When $k_i \neq i-2$, one can think of $v_{k_i}$ as a `pivot', at which the triangles that form the structure `change direction'.
The top of Figure~\ref{foldedpathfig} shows a folded path, with arrows drawn from a pivot $v_{k_i}$ to $v_i$.

Another way to view a folded path is as a sequence of square paths which are disjoint apart from initial and final triangles, which are shared by consecutive paths.
 Figure~\ref{foldedpathfig} gives an example of a homomorphism from a square path to a folded path.
One can visualise folding a square path so that triangles map onto triangles, using the pivots as directions for where to fold.
In Lemma~\ref{foldedpath}, we show that given any folded path $F$, there is a homomorphism from some square path $P$ to $F$ where $P$ `stretches along the length' of $F$ and where $|P|$ is not significantly greater than $|F|$.

In the next proposition, we prove that, in a folded path, every edge lies in a triangle.

\begin{proposition}\label{nbrpath}
Let $F := v_1 \ldots v_n$ be a folded path with ordering $k_3, \ldots, k_n$.
Then, for all $xy \in E(F)$, we have $N^2_F(x,y) \neq \emptyset$.
\end{proposition} 

\begin{proof}
Write $x=:v_j$ and $y=:v_{\ell}$ where $j < \ell$.
Then $j \in \lbrace \ell-1,k_{\ell} \rbrace$.
Recall that $k_\ell \in \lbrace \ell-2,k_{\ell-1} \rbrace$.
For each of the four possible values of $(j,k_\ell)$, we will exhibit a vertex $z \in N^2_F(v_j,v_\ell)$.
Suppose first that $k_\ell = \ell-2$.
If $j=\ell-1$, then we set $z := v_{\ell-2}$.
If $j=k_\ell$, then we set $z := v_{\ell-1}$.
Suppose instead that $k_\ell = k_{\ell-1}$.
If $j=\ell-1$, then we set $z := v_{k_{\ell-1}}$.
If $j=k_\ell$, then we set $z := v_{\ell-1}$.
\end{proof}

\subsection{Embedding a square path into a folded path}

The next lemma guarantees a homomorphism from a square path $P$ into a given folded path $F$, with some special properties.
Later (in the proof of Lemma~\ref{babyhorror}), we will use this lemma in combination with Lemma~\ref{partialembed} to embed a square path $P$ into a graph $G$ whose reduced graph contains a copy of some folded path $F$.
Figure~\ref{foldedpathfig} illustrates the idea of the proof.

\begin{lemma}\label{foldedpath}
Let $n \geq 3$ be an integer and let $F = v_{1} \ldots v_{n}$ be a folded path.
Then there exists $p \in \mathbb{N}$ and a mapping $g:[p] \rightarrow [n]$ such that $n \leq p \leq 2n+1$ and if  $P := x_1 \ldots x_p$ is a square path, we have that
\begin{itemize}
\item[(i)] $v_{g(i)}v_{g(j)} \in E(F)$ whenever $x_{i}x_{j} \in E(P)$;
\item[(ii)] $g(1)=1$, $g(2)=2$, $g(3)=3$, $g(\lbrace p-1,p \rbrace)= \lbrace n-1, n \rbrace$.
\end{itemize} 
\end{lemma}

\begin{proof}
Let $k_3, \ldots, k_n$ be the ordering of $F$.
We will begin by finding a square path $P' := x_1 \ldots x_{p'}$ with $n \leq p' \leq 2n$ and a function $g : [p'] \rightarrow [n]$ such that $v_{g(i)}v_{g(j)} \in E(F)$ whenever $x_ix_j \in E(P')$; $g(1)=1$, $g(2)=2$, $g(3)=3$ and $g(p')=n$.
We prove this iteratively.
Suppose that for some $3 \leq i \leq n-1$ there exists $p_i \in \mathbb{N}$ such that $i \leq p_i \leq 2i$ and a function $g : [p_i] \rightarrow [n]$ such that
\begin{itemize}
\item[($A_i$)] $g(1)=1$, $g(2)=2$, $g(3)=3$, $g(p_i) = i$, $g([p_i]) = [i]$, and $g(j) < i$ for all $j < p_i$;
\item[($B_i$)] $v_{g(j)}v_{g(j')} \in E(F)$ whenever $1 \leq j < j' \leq p_i$ and $j'-j \in \lbrace 1,2 \rbrace$.
\end{itemize}
Observe that $(B_i)$ is equivalent to the statement that $v_{g(j)}v_{g(j')} \in E(F)$ whenever $x_jx_{j'} \in E(P_i)$, where $P_i := x_1 \ldots x_{p_i}$ is a square path.
So our aim is to find $p' := p_n \in \mathbb{N}$ and $g : [p'] \rightarrow [n]$ such that $(A_n)$ and $(B_n)$ hold.
Certainly $(A_3)$ and $(B_3)$ hold. Assume that $(A_i)$ and $(B_i)$ hold for some fixed $3 \leq i \leq n-1$.
We will extend the domain of $g$ by defining $p_{i+1} \geq p_i+1$ and $g(p_i+1), \ldots, g(p_{i+1})$ so that ($A_{i+1}$) and ($B_{i+1}$) hold.

We first give some motivation. Ideally,
we would like to set $p_{i+1} := p_i+1$, so $g(p_i+1)=i+1$.
But we can only do so if $v_{g(p_i-1)},v_{g(p_i)}$ are both neighbours of $v_{i+1}$ in $F$ such that $g(p_i-1),g(p_i) \leq i$, and by definition of $F$, this only occurs if
  $g(\lbrace p_i-1,p_i \rbrace) = \lbrace k_{i+1},i \rbrace$.
By $(A_i)$, $g(p_i)=i$, so we need that $g(p_i-1)=k_{i+1}$.
If this is not the case, we cannot set $p_{i+1}=p_i+1$, but it turns out that we can take $p_{i+1}=p_i+2$ (by taking a single intermediate step via $k_{i+1}$).

Indeed, set
\begin{alignat}{2}
\label{def1} p_{i+1} &:= p_i+1\ \ \text{ and }\ \ g(p_i+1) := i+1\ \ &\text{ if }\ \ k_{i+1} = g(p_i-1);\\
\label{def2} p_{i+1} &:= p_i+2\ \ \text{ and }\ \ g(p_i+1) := k_{i+1},\ g(p_i+2) := i+1\ \ &\text{ if }\ \ k_{i+1} \neq g(p_i-1).
\end{alignat}
First we check that $(A_{i+1})$ holds.
We have $g(1)=1$, $g(2)=2$, $g(3)=3$ and $p_{i+1} \geq p_i+1 \geq i+1$.
Moreover, $p_{i+1} \leq p_i+2 \leq 2(i+1)$ and $g(p_{i+1}) = i+1$.
Furthermore,
$$
g([p_{i+1}]) = g([p_i]) \cup g([p_i+1,p_{i+1}]) \in \lbrace [i] \cup \lbrace i+1 \rbrace , [i] \cup \lbrace k_{i+1}, i+1 \rbrace \rbrace = [i+1],
$$
and $g(j) < i+1$ for all $j < p_{i+1}$.
(Here we used that $k_{i+1} \leq i-1$.)
So ($A_{i+1}$) holds.

Now we show that ($B_{i+1}$) holds.
Let $E'(F)$ be the set of all triples $\lbrace j,j',j'' \rbrace$ such that $v_jv_{j'}v_{j''}$ is a triangle in $F$.
It suffices to prove that $g(\lbrace j-2,j-1,j \rbrace) \in E'(F)$ for all $p_i+1 \leq j \leq p_{i+1}$.
We claim that
\begin{equation}\label{ki}
\lbrace k_t,t-1,t \rbrace \in E'(F)
\end{equation}
for all $3 \leq t \leq n$.
To see this, note that $\lbrace v_{k_{t-1}}v_{t-1}, v_{t-2}v_{t-1}, v_{k_t}v_t, v_{t-1}v_t \rbrace \subseteq E(F)$ by definition and $k_t \in \lbrace k_{t-1},t-2 \rbrace$, proving the claim.

By the definition of a folded path,
\begin{equation}\label{NFvi}
N_F(v_i) \cap \lbrace v_{g(1)}, \ldots, v_{g(p_i)} \rbrace \stackrel{(A_i)}{=} N_F(v_i) \cap \lbrace v_1, \ldots, v_{i} \rbrace = \lbrace v_{k_i}, v_{i-1} \rbrace.
\end{equation}
 ($B_i$) implies that $v_{g(p_i-1)}v_{g(p_i)} \in E(F)$, i.e. that
$v_{g(p_i-1)} \in N_F(v_{i})$. 
So (\ref{NFvi}) implies that
\begin{equation}\label{inset}
g(p_i-1) \in \lbrace k_{i}, i-1 \rbrace,\ \ \text{ and also }\ \ k_{i+1} \in \lbrace k_{i}, i-1 \rbrace.
\end{equation}
There are now two cases to consider depending on the value of $g(p_i-1)$.

Suppose first that
$k_{i+1} = g(p_i-1)$.
Then $p_{i+1} = p_i+1$ by (\ref{def1}).
We have that
$$
g(\lbrace p_i-1,p_i,p_i+1 \rbrace) \stackrel{(\ref{def1})}{=} \lbrace k_{i+1}, i, i+1 \rbrace \stackrel{(\ref{ki})}{\subseteq} E'(F),
$$
proving ($B_{i+1}$) in this case.

Suppose instead that
$k_{i+1} \neq g(p_i-1)$.
Then $p_{i+1} = p_i+2$ by (\ref{def2}).
Now (\ref{inset}) gives that
\begin{equation}\label{ki+1}
\lbrace k_{i+1},g(p_i-1) \rbrace = \lbrace k_i,i-1 \rbrace.
\end{equation}
Therefore
\begin{align*}
g(\lbrace \lbrace p_i-1,p_i,p_i+1\rbrace, \lbrace p_i,p_i+1,p_i+2 \rbrace \rbrace)
&\stackrel{(\ref{def2})}{=}
\lbrace \lbrace g(p_i-1), i, k_{i+1} \rbrace,
\lbrace i, k_{i+1}, i+1 \rbrace
\rbrace\\
&\stackrel{(\ref{ki+1})}{=} \lbrace \lbrace k_i, i-1, i \rbrace,
\lbrace k_{i+1}, i, i+1 \rbrace
\rbrace\\
&\stackrel{(\ref{ki})}{\subseteq} E'(F),
\end{align*}
proving ($B_{i+1}$) here.

We have proved that we can obtain $p' := p_n$ where $n \leq p' \leq 2n$ and $g:[p'] \rightarrow [n]$ such that ($A_{n}$) and ($B_{n}$) hold.
Now $g(p')=n$ by $(A_n)$.
Therefore if $g(p'-1)=n-1$, we are done.
So suppose not (see e.g. Figure~\ref{foldedpathfig} with $p'=37$).
The definition of $F$ implies that $N_F(v_n) = \lbrace v_{k_n}, v_{n-1} \rbrace$.
So $v_n$ has exactly one neighbour $v_{k_n}$ in $F$ which is not $v_{n-1}$.
Then $(B_n)$ implies that $g(p'-1)=k_n$ and $g(p'-2)=n-1$.
Now (\ref{ki}) implies that $\lbrace k_n,n-1,n \rbrace \in E'(F)$,
i.e. $v_{k_n}v_{n-1}v_n$ is a triangle in $F$.
We are done by setting $p := p'+1$ and $g(p):=n-1$.
\end{proof}

\subsection{Finding a folded path in an $\eta$-good graph}

Recall that we can use Theorem~\ref{trianglepack} to find a triangle packing $(T_i)_i$ in the reduced graph $R$ of $G$, and then apply the Blow-up lemma (Theorem~\ref{blowup})
so that for each $i$ we find a 
 square path $P_i$  in $G$ that almost spans the vertex set of $G$ corresponding to $T_i$.
So $(P_i)_i$ covers almost all the vertices of $G$.
However, to prove Lemma~\ref{almostpath}, we require that each $P_i$ is both head- and tail-heavy.
We will extend each $P_i$ both forwards and backwards by finding square paths $R_i,R_i'$ such that $R_iP_iR_i'$ is a head- and tail-heavy square path, and $|R_i|,|R_i'|$ are small.
 To do so, we will find folded paths $F_i$ and $F_i'$ in $R$ which will form the `framework' for $R_i$ and $R_i'$ respectively.

This is achieved in Lemma~\ref{findfpath}, whose proof is the aim of this subsection.
Given a triangle $T_i$, in order to find two `types' of paths $R_i,R_i'$,  Lemma~\ref{findfpath} `produces' two folded paths such that the first three clusters in both of these folded paths are the clusters from $T_i$, but the order of these clusters differs. Further, in both folded paths the last two vertices  correspond to clusters containing many high-degree vertices in $G$.

We  use standard cycle notation for permutations, so, for example, $(12)$ maps $1$ to $2$ and $2$ to $1$ (and fixes everything else).

\begin{lemma}\label{findfpath}
Let $\gamma,\eta > 0$ and $n \in \mathbb{N}$ where $0 < 1/n \ll \gamma \ll \eta \ll 1$.
Suppose that $G$ is a graph on $n$ vertices.
Let $d_G' : V(G) \rightarrow \mathbb{R}$ be an $(\eta,n)$-good function such that
\begin{equation}\label{d'G}
d_G(x) \geq (1-\gamma)d_G'(x)
\end{equation}
for all $x \in V(G)$.
Let $T$ be the vertex set of a triangle in $G$.
Then there exists $8 \leq t \leq 5/\eta$ and an ordering $v_1,v_2,v_3$ of $T$ such that $G$ contains a folded path $F = v_{1}v_2 v_3 \ldots v_{t}$
such that $d'_G(v_{t-1}),d'_G(v_t) \geq (2/3 + \eta)n$.
Moreover, there exists $\sigma \in \lbrace (12),(23) \rbrace$ and $8 \leq t' \leq 5/\eta$  such that $G$ contains a folded path $H = v_{\sigma(1)} v_{\sigma(2)} v_{\sigma (3)} v'_4 \ldots v'_{t'}$
such that $d'_G(v'_{t'-1}),d'_G(v'_{t'}) \geq (2/3 + \eta)n$.
\end{lemma}

\begin{proof}
We split the proof into three steps.
In the first step, we find a short folded path $F'$ whose final vertex $v_s$ has $d'_{G}(v_s) \geq (2/3+\eta)n$.
In the second step, we extend $F'$ into $F''$ so that the final vertex $v_r$ of $F''$ is a neighbour of $v_s$ in $F''$, and $d'_G(v_r) \geq (2/3+\eta)n$.
Finally, in the third step, we extend $F''$ into $F$ by appending four additional vertices. 
Simultaneously we will construct $H$ (using the same process used to construct $F$).

\medskip
\noindent
\emph{Step 1.}
Obtaining an ordering $v_1,v_2,v_3$ of $T$ such that there exists a folded path $F' = v_1v_2v_3 \ldots v_{s}$ where $d'_G(v_{s}) \geq (2/3+\eta)n$ and $3 \leq s \leq 4/\eta$.
Obtaining $\sigma \in \lbrace (12),(23) \rbrace$ such that there exists a folded path $H' = v_{\sigma(1)}v_{\sigma(2)}v_{\sigma(3)}v'_4 \ldots v'_{s'}$ where $3 \leq s' \leq 4/\eta$ and $d'_G(v) \geq (2/3+\eta)n$ for the final vertex $v$ on $H'$.

\medskip
\noindent
Consider any $S\subseteq V(G)$ with $|S|=3$ and write $S=\{x_1,x_2,x_3\}$ where $d'_G (x_1)\leq d'_G (x_2) \leq d'_G (x_3)$. 
We define $\mathcal{C}(S) := (\alpha_1, \alpha_2,\alpha_3)$ where  $d'_G(x_i) = (1/3+\alpha_i\eta)n$ for all $1 \leq i\leq 3$.
So $1 \leq \alpha_1 \leq \alpha_2 \leq \alpha_3 \leq 2/(3\eta)$.
We also write $c(S) := \alpha_1 + \alpha_2 + \alpha_3$.

Suppose that there exists $z \in T$ with $d'_G(z) \geq (2/3+\eta)n$.
Then, writing $T := \lbrace v_1,v_2,v_3 := z \rbrace$, we are done by setting $F' := v_1v_2v_3$ and $H':=v_2v_1v_3$ (so $\sigma = (12))$.

Therefore we may assume that $d'_G(z) < (2/3+\eta)n$ for all $z \in T$.
Note that $c(T) \geq 3$ since $d'_G$ is $(\eta,n)$-good.
Further, the minimum degree of $G$ implies that there are two vertices $v_2,v_3$ in $T$ with common neighbour $v_4 \in N^2_G(T) \setminus T$.
Let $v_1$ be the vertex in $T \setminus \lbrace v_2,v_3 \rbrace$.
Set $F_1' := v_1v_2v_3$ and $F'_2 := v_1v_2v_3v_4$.
Then $F'_2$ is a folded path with $\lbrace v_1,v_2,v_3 \rbrace = T$; and  $k_3, k_4$ is the ordering of  $F'_2$ where $k_3:=1$ and $k_4:=2$.
Since $d'_G$ is $(\eta,n)$-good, we have $c(\lbrace v_2,v_3,v_4 \rbrace) \geq 3$.
Note that the ordering $v_2,v_3$ was arbitrary.

To achieve Step~1, we will now concentrate on achieving Step~1$'$.

\medskip
\noindent
\emph{Step 1$'$.}
Obtaining a folded path $F' = v_1 v_2v_3 \ldots v_{s}$ where  $d'_G(v_{s}) \geq (2/3+\eta)n$ and $4 \leq s \leq 4/\eta$.

\medskip
\noindent
Suppose that for some $2 \leq i \leq  3/\eta $, we have defined a folded path $F'_i$ such that the following hold. 
\begin{itemize}
\item[($A_i$)] $F'_i := v_1 \ldots v_{m_i}$ for some $4 \leq m_i \leq i+2$;
\item[($B_i$)] $T_i := \lbrace v_{m_i}, v_{m_i-1}, v_{k_{m_i}} \rbrace$ is such that $c(T_i) \geq i/2$.
\end{itemize}
Note that $T_i = \lbrace v_{m_i} \rbrace \cup N_{F_i'}(v_{m_i})$. We have shown that ($A_2$) and ($B_2$) hold.

If $d'_G(v_{m_i}) \geq (2/3+\eta)n$,
set $s := m_i$ and $F' := F_i'$.
Otherwise,
we will obtain $F'_{i+1}$ from $F'_i$ so that $F'_{i+1}$ satisfies ($A_{i+1}$) and ($B_{i+1}$).
Write $m := m_i$ and let $k_3, \ldots, k_m$ be the ordering of $F'_i$.
Note that if there exist
\begin{equation}\label{need}
k_{m+1} \in \lbrace m-1,k_m \rbrace\ \ \text{ and }\ \ v_{m+1} \in N^2_{G}(v_m,v_{k_{m+1}}) \setminus V(F_i'),
\end{equation}
then $v_1 \ldots v_{m+1}$ is a folded path in $G$ with ordering $k_3, \ldots, k_{m+1}$.

Proposition~\ref{2ndnbrhd}(ii) and (\ref{d'G}) imply that
\begin{eqnarray}\label{N2N3}
|N^2_G(T_i)|+|N^3_G(T_i)| &\geq& (1-\gamma)\sum_{x \in T_i}d'_G(x)-n = (1-\gamma)c(T_i)\eta n - \gamma n.
\end{eqnarray}
Write 
\begin{equation}\label{CT0}
\mathcal{C}(T_i) =: (\alpha_1,\alpha_2,\alpha_3)\ \ \text{ so }\ \ \alpha_1+\alpha_2+\alpha_3 \stackrel{(B_i)}{\geq} i/2.
\end{equation}
There are two cases to consider, depending on the sizes of $N^2_G(T_i)$ and $N^3_G(T_i)$.

\medskip
\noindent
\textbf{Case 1.}
\emph{$|N^3_G(T_i)| \geq (1-\gamma)\alpha_1\eta n - \gamma n/2$.}

\medskip
\noindent
Suppose that there is no vertex $v_{m+1} \in N^3_G(T_i) \setminus V(F_i')$ with $d'_G(v_{m+1}) \geq (2/3+\eta)n$.
Then
$$
(1-\gamma)\alpha_1\eta n - \gamma n/2 - |F_i'| \leq |N_G^3(T_i) \setminus V(F_i')| \leq n/3
$$
by Proposition~\ref{vertexdeg}(i).
So
Proposition~\ref{vertexdeg}(ii) implies that we can choose $v_{m+1} \in N^3_G(T_i) \setminus V(F_i')$ with
\begin{eqnarray}\label{degbi+1}
d'_G(v_{m+1}) &\geq& \left(\frac{1}{3} + \eta\right)n + (1-\gamma)\alpha_1 \eta n  - \gamma n/2 - |F_i'|\\
\nonumber &\stackrel{(A_i)}{\geq}& \left(\frac{1}{3} + \left(\alpha_1 + 1\right)\eta\right)n -  2\gamma n - \left( \frac{3}{\eta} +2 \right)\\
\nonumber &\geq& \left(\frac{1}{3} + \left(\alpha_1 + 1/2\right)\eta\right)n.
\end{eqnarray}
Let $v_j,v_\ell \in T_i$ be such that $d'_G(v_j)= (1/3 + \alpha_2\eta) n$ and
$d'_G(v_\ell)= (1/3 + \alpha_3\eta) n$.
Observe that $v_{m+1} \in N^2_G(v_j,v_\ell)$.
Set $T_{i+1} := \lbrace v_{m+1},v_j,v_\ell \rbrace$.
Then, by (\ref{degbi+1}), we have $c(T_{i+1}) \geq (\alpha_1+1/2)+\alpha_2 +\alpha_3$.
Thus we have shown
\begin{equation}\label{CT1}
c(T_{i+1}) \geq (\alpha_1+1/2)+\alpha_2 +\alpha_3\ \ \text{ or }\ \ d'_G(v_{m+1}) \geq (2/3+\eta)n.
\end{equation}

\medskip
\noindent
\textbf{Case 2.}
\emph{$|N_G^2(T_i)| \geq (1-\gamma)(\alpha_2+\alpha_3)\eta n - \gamma n/2$.}

\medskip
\noindent
Note that this is indeed the only other case by (\ref{N2N3}) and (\ref{CT0}).
Then, similarly as above, Proposition~\ref{vertexdeg}(i) and (ii) imply that we can choose $v_{m+1} \in N^2_G(T_i) \setminus V(F_i')$ with
\begin{align}\label{degbi+1'}
d'_G(v_{m+1}) &\geq \min \left\lbrace \left( \frac{2}{3} + \eta \right)n, \left(\frac{1}{3} + \left(\alpha_2 + \alpha_3 + 1/2\right)\eta\right)n \right\rbrace.
\end{align}
Let $v_j,v_\ell$ be two neighbours of $v_{m+1}$ in $T_i$, where $d'_G(v_j) \leq d'_G(v_\ell)$.
So $d'_G(v_j) \geq (1/3 + \alpha_1\eta) n$ and
$d'_G(v_\ell) \geq (1/3 + \alpha_2\eta) n$.
In this case we set $T_{i+1} := \lbrace v_{m+1},v_j,v_\ell \rbrace$.
Then
\begin{equation}\label{CT2}
c(T_{i+1}) \geq \alpha_1+\alpha_2+(\alpha_2+\alpha_3+1/2)\ \ \text{ or }\ \ d'_G(v_{m+1}) \geq (2/3+\eta)n.
\end{equation}

\bigskip

\noindent
We now consider both Cases 1 and 2 together. In both cases, $v_{m+1} \notin V(F'_i)$ and $\lbrace j,\ell \rbrace \subseteq \lbrace k_m,m-1,m \rbrace$.
If $\lbrace j,\ell \rbrace = \lbrace k_m,m-1 \rbrace$ then we obtain $F'_{i+1}$ from $F_i'$ by replacing $v_m$ with $v_{m+1}$ (and $k_m$ is unchanged).
Then $F'_{i+1}$ is certainly a folded path in $G$.
Otherwise, one of $j,\ell$ equals $m$,
and we choose $k_{m+1}$ so that $\lbrace m,k_{m+1} \rbrace = \lbrace j,\ell \rbrace$.
Note that $k_{m+1} \in \lbrace k_m,m-1 \rbrace$.
So $F_{i+1}':=v_1\dots v_m v_{m+1}$ is a folded path in $G$ by (\ref{need}).
In both cases, $F'_{i+1}$ is a folded path with
$$
4 \leq m \leq |F'_{i+1}| \leq m+1 \leq i+3.
$$
Moreover, since $m = m_i \geq 4$, the first three vertices of $F'_{i+1}$ are $v_1,v_2,v_3$, as required.
So ($A_{i+1}$) holds.

If $d'_G(v_{m+1}) \geq (2/3+\eta)n$ then we set $F' := F'_{i+1}$ and Step~1$'$  is complete.
Otherwise, (\ref{CT0}), (\ref{CT1}), (\ref{CT2}) and ($B_i$) imply that  $c(T_{i+1}) \geq (i+1)/2$.
So ($B_{i+1}$) holds.
We have thus defined $F_{i+1}'$ so that $(A_{i+1})$ and ($B_{i+1}$) both hold.

Therefore after repeating this process at most $S:= 3/\eta $ times  either we obtain a folded path $F'$ as desired in Step~1$'$ or we obtain a folded path $F'_S=v_1 \dots v_{s}$ (where $s := m_S$) that satisfies $(A_S)$ and $(B_S)$ and so $c(T_S) \geq 3/(2\eta)$. In the latter case, we may assume that $d'_G (v_j) < (2/3+\eta)n$ for all $4\leq j \leq s-1$ (otherwise setting $F':=v_1\dots v_j$ yields our desired folded path). 
Note that
\begin{equation*}
\frac{1}{3}\sum_{x \in T_S}d'_G(x) = \frac{1}{3}(1 + c(T_{S})\eta)n \stackrel{(B_S)}{\geq} \frac{1}{3}(1+3/2)n =5n/6 \geq (2/3+\eta)n
\end{equation*}
and so 
\begin{equation}\label{vs6}
d'_G (v_{s})\geq (2/3+\eta)n.
\end{equation}
Observe that
\begin{equation}\label{VF'}
s \stackrel{(A_S)}{\leq} S+2 \leq  3/\eta+2 \leq 4/\eta.
\end{equation}
Set $F' := F'_S$.
This completes the proof of Step~1$'$. Since the choice of the ordering $v_2,v_3$ was arbitrary, we can argue precisely as in Step~1$'$, now with $F'_2$ replaced with $H'_2:=v_1v_3v_2v_4$ to obtain a folded path $H'$ as desired in Step~1 (so $\sigma = (23)$).
This completes Step~1.
From now on we only extend $F'$ to $F$ since the process of extending $H'$ to $H$ is identical.

\medskip
\noindent
\emph{Step 2.}
Obtaining a folded path $F'' := v_1v_2v_3 \ldots v_r$ where $4 \leq s+1 \leq r \leq 49/(10\eta)$, where $k_r = s$ and $d'_G(v_s), d'_G(v_r) \geq (2/3+\eta)n$. 
\medskip
\noindent

Let $F_0 := F' = v_1 \ldots v_s$.
Suppose that for some $0 \leq i \leq  1/3\eta $, we have  defined folded paths $F_0, \ldots, F_i$ such that
$F_i := v_1 \ldots v_{s+i}$ where for all $1 \leq j \leq i$ we have $v_{s+j} \in N_{F_i}(v_s)$
and $d'_G(v_{s+i}) \geq (1/3+i\eta)n$.
So $k_{s+j} = s$ for all $2 \leq j \leq i$.

By choosing an arbitrary $v_{s+1} \in N_G ^2 (v_{s-1}, v_s) \setminus V(F')$ we can find our desired folded path $F_1=v_1 \dots v_s v_{s+1}$. (Such a vertex $v_{s+1}$ exists by (\ref{d'G}) and (\ref{vs6}) and since $d'_G$ is $(\eta,n)$-good.)
Thus, we may assume that $i \geq 1$.

If there exists $1 \leq j \leq i$ with $d'_G(v_{s+j}) \geq (2/3+\eta)n$, we are done by setting $r := s+j \geq 1$ and $F'' := F_j$.
Otherwise, Proposition~\ref{2ndnbrhd}(i) implies that
\begin{eqnarray*}
|N^2_G(v_s,v_{s+i}) \setminus V(F_i)| &\stackrel{(\ref{VF'})}{\geq}& (d_G(v_s)+d_G(v_{s+i})) - n - (4/\eta+i)\\
&\stackrel{(\ref{d'G})}{\geq}& (1-\gamma)(d'_G(v_s) + d'_G(v_{s+i}))-n - 5/\eta\\
&\stackrel{(\ref{vs6})}{\geq}& (1-\gamma)(i+1)\eta n -\gamma n -5/\eta \geq i\eta n.
\end{eqnarray*}
Since $i \leq 1/3\eta$, Proposition~\ref{vertexdeg}(ii) implies that $N^2_G(v_s,v_{s+i}) \setminus V(F_i)$ contains a vertex $v_{s+i+1}$ with
$$
d'_G(v_{s+i+1}) \geq (1/3 + (i+1)\eta)n.
$$
Therefore (\ref{need}) implies that $F_{i+1} := v_1 \ldots v_{s+i+1}$ is a folded path with $k_{s+i+1}=s$.

After  $r-s \leq  1/3\eta+1$ steps we obtain $F'' := F_{r-s} = v_1 \ldots v_{r}$, where $v_{r} \in N_{F''}(v_s)$ and
\begin{equation}\label{vt}
d'_G(v_{r}) \geq (1/3 + (1/3\eta+1)\eta)n=(2/3+\eta)n.
\end{equation}
Now (\ref{VF'}) implies that
$
4 \leq r \leq 4/\eta +  1/3\eta +1 \leq 49/(10\eta)
$.
This completes the proof of Step~2.

\medskip
\noindent
\emph{Step 3.}
Obtaining $F$.

\medskip
\noindent
Let $a,b \in V(G)_{\eta/2}$ be arbitrary.
By Proposition~\ref{2ndnbrhd}(i), we have that
$$
|N^2_G(a,b) \setminus V(F'')| \geq \left(\frac{1}{3} + \eta\right)n - \left(\frac{4}{\eta} + \frac{1}{3\eta} + 1 \right) \geq \left(\frac{1}{3} + \frac{\eta}{2}\right) n.
$$
Proposition~\ref{vertexdeg}(i) implies that there exists a set $K(a,b) \subseteq N^2_G(a,b) \setminus V(F'')$ with $|K(a,b)| \geq \eta n/2$, such that for each $x \in K(a,b)$ we have $d'_G(x) \geq (2/3+\eta)n$. 
Furthermore, (\ref{d'G}) implies that $K(a,b) \subseteq V(G)_{\eta/2}$.
Observe that $\lbrace v_s,v_r \rbrace \subseteq V(G)_{\eta/2}$.
So, for each $1 \leq j \leq 4$, we can find a distinct vertex $v_{r+j}$ so that
$v_{r+1} \in K(v_s,v_r)$,
and $v_{r+j} \in K(v_{r+j-2},v_{r+j-1})$
for all $2 \leq j \leq 4$.
Let $k_1, \ldots, k_r$ be the ordering of $F''$.
Then $F := v_1 \ldots v_{r+4}$ is a folded path with ordering $k_1, \ldots, k_{r+4}$, where 
$$
k_{r+1} := s,\ \  k_{r+2} := r,\ \  k_{r+3} := r+1,\ \ k_{r+4} := r+2.
$$ 
To see this, observe that $k_i\in \lbrace i-2,k_{i-1} \rbrace$ for all $r+1 \leq i \leq r+4$ since $k_r=s < r$.

Let $t := r+4$.
Then Step~2 implies that $8 \leq t \leq 49/(10\eta) + 4 \leq 5/\eta$, as required.
Finally, $d'_G(v_{t-1}), d'_G(v_t) \geq (2/3+\eta)n$, as required.
\end{proof}

Note that in Step~3 of the proof of Lemma~\ref{findfpath} we add $4$ vertices only to ensure that the folded path $F$ has length at least $8$ (this property will be useful later on). In particular, we could have guaranteed that the last two vertices of $F$ have `large' degree by only adding a single vertex in this final step.

\subsection{The proof of Lemma~\ref{almostpath}}

The next lemma shows that, given a suitable framework in the reduced graph $R$ of $G$, we can find a tail-heavy square path $P$ such that, in two (or three) given clusters, there are many pairs (or triples) of vertices that can be added to the start of $P$ to extend the square path.
The necessary framework is a folded path whose first three vertices correspond to these given clusters, and whose final two vertices have large core degree. (Recall that the $\alpha$-core degree $d^\alpha_{R,G}$ is defined in Section~\ref{coredegree}.)
The proof is essentially just an application of Lemmas~\ref{partialembed} and~\ref{foldedpath}; its length is due to technical issues.

\begin{lemma}\label{babyhorror}
Let $n,L \in \mathbb{N}$ and suppose that $0 < 1/n \ll 1/L \ll \eps \ll c \ll d \ll \eta \ll 1$.
Let $R$ be a graph with $V(R) = [L]$.
Let $G$ be a graph on $n$ vertices with vertex partition $V_0,V_1, \ldots, V_{L}$ such that $|V_0| \leq \eps n$ and so that there exists $m \in \mathbb{N}$ with $|V_i|=(1 \pm \eps)m$ for all $1 \leq i \leq L$. Further, suppose that  $G[V_i,V_j]$ is $(\eps,d)$-regular whenever $ij \in E(R)$.
Let $F = i_1 \ldots i_t$ be a folded path in $R$ with $8 \leq t \leq 5/\eta$ such that $d^{2c}_{R,G}(i_{t-1}), d^{2c}_{R,G}(i_t) \geq (2/3+\eta)L$.
Then 
\begin{itemize}
\item[(1)] $G$ contains an $\eta$-tail-heavy square path $Q$ with $|Q| \leq 11/\eta$, and sets $A_k \subseteq V_{i_k} \setminus V(Q)$ for $k=1,2$  with $|A_k| \geq cm$, such that for any
$z_k \in A_k$ where $z_1z_2 \in E(G)$, we have that $z_1z_2Q$ is a square path;
\item[(2)] $G$ contains an $\eta$-tail-heavy square path $P$ with $|P| \leq 11/\eta$, and sets $B_k \subseteq V_{i_k} \setminus V(P)$ for $k=1,2,3$  with $|B_k| \geq cm$, such that for any
$z_k \in B_k$ where $z_1z_2z_3$ is a triangle in $G$, we have that $z_1z_2z_3P$ is a square path. Further, for any $z_2 \in B_2$, $z_3 \in B_3$ such that $z_2z_3 \in E(G)$ we have that $z_2z_3P$ is a square path.
\end{itemize}
Moreover, neither $P$ nor $Q$ contain any vertices from $V_0$.
\end{lemma}

\begin{proof}
We will only prove (2) since the proof of (1) is very similar.
Apply Lemma~\ref{foldedpath} with $t$ playing the role of $n$ to obtain a square path $P' := x_1 \ldots x_p$ where $p$ satisfies
\begin{equation}\label{pwhat}
8 \leq t \leq p \leq 2t+1 \leq 10/\eta+1
\end{equation}
and a mapping $g : [p] \rightarrow [t]$ such that $i_{g(j)}i_{g(k)} \in E(F)$ whenever $x_jx_k \in E(P')$; and $g(1)=1$, $g(2)=2$, $g(3)=3$ and $g(\lbrace p-1,p \rbrace ) = \lbrace t-1, t \rbrace$.
Let $f : V(P') \rightarrow V(F)$ be such that $f(x_j) = i_{g(j)}$ for all $1 \leq j \leq p$.
So $f(x)f(y) \in E(F)$ whenever $xy \in E(P')$.
Moreover,
\begin{equation}\label{gwhere}
f(x_1) = i_1, \ \ f(x_2) = i_2, \ \ f(x_3) = i_3 \ \ \text{ and }\ \  f(\lbrace x_{p-1},x_p \rbrace ) = \lbrace i_{t-1}, i_t \rbrace.
\end{equation}

Let
\begin{equation}\label{YX}
Y := (P')^-_3 \cup (P')^+_2 = \lbrace x_1,x_2,x_3, x_{p-1},x_p \rbrace\ \ \text{ and let }\ \ X := V(P') \setminus Y.
\end{equation}
Observe that $X \neq \emptyset$ by (\ref{pwhat}).
Then Lemma~\ref{partialembed} with $G\setminus V_0, R,P',X,Y,2c,f$ playing the roles of $G,R,H,X,Y,c,f$
 implies that there exists an injective mapping $\tau : X \rightarrow V(G)$ with $\tau(x_j) \in V_{f(x_j)}$ for all $4 \leq j \leq p-2$, such that there exist sets
\begin{equation}\label{Ck}
C_k \subseteq V_{f(x_k)} \setminus \tau(X)\ \ \text{ for all }\ \ k \in \lbrace 1,2,3,p-1,p \rbrace
\end{equation}
such that
\begin{itemize}
\item[(i)] if $x,x' \in X$ and $xx' \in E(P')$ then $\tau(x)\tau(x') \in E(G)$;
\item[(ii)] for all $k \in \lbrace 1,2,3,p-1,p \rbrace$ we have that $C_k \subseteq N_G(\tau(x))$ for all $x \in N_{P'}(x_k) \cap X$;
\item[(iii)] $|C_k| \geq 2c(1-\eps)m$ for all $k \in \lbrace 1,2,3,p-1,p \rbrace$.
\end{itemize}
Property (i) implies that $\tau(X)$ spans a square path $P_1 := \tau(x_4) \ldots \tau(x_{p-2})$ in $G$ (which contains at least three vertices by (\ref{pwhat})).
We would like to use $P_1$ to find an $\eta$-tail-heavy path.
To do this, we will find a short square path whose first two vertices lie in $C_{p-1},C_p$ respectively, and whose last two vertices both lie in $V(G)_\eta$.
Since $d^{2c}_{R,G}(i_t), d^{2c}_{R,G}(i_{t-1}) \geq (2/3+\eta)L$, there exist sets $B^{t-1} \subseteq V_{i_{t-1}}$ and $B^t \subseteq V_{i_t}$ such that $|B^{t-1}|,|B^t| \geq 2c(1-\eps)m$ and $B^{t-1} \cup B^t \subseteq V(G)_\eta$. 
Let $C^{t-1} := B^{t-1} \setminus \tau(X)$ and define $C^t$ similarly.
Then
$$
|C^{t-1}|,|C^t| \geq 2c(1-\eps)m - |\tau(X)| \geq 2c(1-\eps)m - p + 5 \stackrel{(\ref{pwhat})}{\geq} 2c(1-\eps)m - 10/\eta + 4 \geq cm.
$$
Recall that $i_{t-1}i_t \in E(F)$ by the definition of a folded path.
Proposition~\ref{nbrpath} implies that there exists $i_s \in N^2_F(i_{t-1},i_t)$, i.e. $i_si_{t-1}i_t$ is a triangle in $F$.

We will assume that $f(x_{p-1})=i_{t-1}$ and so $f(x_p) = i_t$ by (\ref{gwhere}) (the other case, when $f(x_{p-1}) = i_t$ and $f(x_{p})=i_{t-1}$, is almost identical).
Then (\ref{Ck}) implies that $C_{p-1} \subseteq V_{f(x_{p-1})} \setminus \tau(X) = V_{i_{t-1}} \setminus \tau(X)$.
Similarly $C_p \subseteq V_{i_t} \setminus \tau(X)$.
Proposition~\ref{triangleembed} applied with $V_{i_{t-1}},V_{i_t},V_{i_s},C_{p-1},C_p,C^{t-1},C^t, \tau(X)$ playing the roles of $X_1,X_2,X_3,A_1,A_2,B_1,B_2,W$ implies that $G$ contains
a square path $P_2 \in C_{p-1} \times C_p \times V_{i_s} \times C^{t-1} \times C^t$.
Write $P_2 := y_1y_2y_3y_4y_5$.
Observe that, by construction, $V(P_2) \cap \tau(X) = \emptyset$, and $P_2$ is $\eta$-tail-heavy.
We claim that
$$
P := P_1P_2 = \tau(x_4) \ldots \tau(x_{p-2}) y_1y_2y_3y_4y_5
$$
is an $\eta$-tail-heavy square path.
Since $P_1$ and $P_2$ are vertex-disjoint square paths each containing at least two vertices (by (\ref{pwhat})), it suffices to show that the necessary edges between $P_1$ and $P_2$ are present, i.e. that the necessary edges between $\tau(x_{p-3}),\tau(x_{p-2})$ and $y_1,y_2$ are present.
Observe that $N_{P'}(x_{p-1}) \cap X = \lbrace x_{p-3},x_{p-2} \rbrace$.
Then (ii) implies that $y_1 \in C_{p-1} \subseteq N_G(\tau(x_{p-3})) \cap N_G((\tau(x_{p-2}))$, as required.
Similarly $y_2 \in C_{p} \subseteq N_G((\tau(x_{p-2}))$, as required.
So $P$ is a square path. Further, by construction, $P$ is disjoint from $V_0$.

Note further that
$$
|P| = |P_1| + |P_2| = p \stackrel{(\ref{pwhat})}{\leq} 11/\eta.
$$
Let $B_k := C_k \setminus V(P_2)$ for $k=1,2,3$.
Then (\ref{Ck}) implies that
$$
B_k \subseteq V_{f(x_k)} \setminus (\tau(X) \cup V(P_2)) \stackrel{(\ref{gwhere})}{=} V_{i_k} \setminus V(P).
$$
Moreover, (iii) implies that, for $k=1,2,3$, we have $|B_k| \geq 2c(1-\eps)m - |P_2| =2c(1-\eps)m - 5 \geq cm$.
Let $z_k \in B_k$ for $k=1,2,3$ such that $z_1z_2z_3$ is a triangle in $G$.
We must show that $z_1z_2z_3P$ is a square path.
That is, we need to show that $z_2 \in N_G(\tau(x_4))$ and $z_3 \in N_G(\tau(x_4)) \cap N_G(\tau(x_5))$.
But, since $z_k \in C_k$ for all $k=1,2,3$, this is implied by (ii). Similarly, given any $z_2 \in B_2$, $z_3 \in B_3$ such that $z_2z_3 \in E(G)$, (ii) implies that $z_2z_3P$ is a square path.
\end{proof}

In the next lemma, given a small collection of folded paths, we obtain a small collection of short square paths, with certain useful properties.
We find a pair of square paths in $G$ corresponding to each of the $\ell$ triangles $T_i = (i,1)(i,2)(i,3)$ in the reduced graph $R$ of $G$.
The first, $P_i$, is tail-heavy, and there are many pairs of vertices in $(i,1) \times (i,2)$ which can precede $P_i$.
The second, $P_{\ell+i}$, is head-heavy, and there are many pairs of vertices in $(i,3) \times (i,1)$ which can succeed $P_{\ell+i}$.
The proof is by repeated application of Lemma~\ref{babyhorror}.

In the proof of Lemma~\ref{almostpath}, we will find a square path $Q_i$ containing most of the vertices in $T_i$ which will be sandwiched between $P_{\ell+i}$ and $P_i$.
In order to connect $Q_i$ with $P_i$ and $P_{\ell+i}$ we need many pairs of possible start- and endpoints.

\begin{lemma}\label{horror}
Suppose that  $0 < 1/n \ll 1/\ell \ll \eps \ll c \ll d \ll \eta \ll 1$.
Let $R$ be a graph with $V(R) = [\ell] \times [3]$.
Suppose that $G$ is a graph on $n$ vertices with vertex partition $\lbrace V_0 \rbrace \cup \lbrace V_{i,j} : (i,j) \in [\ell] \times [3] \rbrace$ such that $|V_0| \leq \eps n$ and $|V_{i,j}| =: m$ for all $(i,j) \in [\ell] \times [3]$, and $G[V_{i,j},V_{i',j'}]$ is $(\eps,d)$-regular whenever $(i,j)(i',j') \in E(R)$.
Define $\mathcal{V} := \lbrace (i,j) : d^{2c}_{R,G}(V_{i,j}) \geq (2/3+\eta)3\ell \rbrace$.
Let $F_1, \ldots, F_{\ell}, F'_1, \ldots, F'_{\ell}$ be a collection of folded paths in $R$ such that, for all $1 \leq i \leq \ell$ we have
\begin{itemize}
\item[(F1)] $F_i := v_{i,1} \ldots v_{i,t_i}$ and  $F'_i := u_{i,1} \ldots u_{i,s_i}$ where $8 \leq s_i, t_i \leq 5/\eta$;
\item[(F2)]   $\lbrace v_{i,t_{i}-1},v_{i,t_i}, u_{i,s_{i}-1}, u_{i,s_i} \rbrace \subseteq \mathcal{V}$;
\item[(F3)] for all $1 \leq j \leq 3$ we have $v_{i,j} = (i,j)$ and 
there exists $\sigma_i \in \lbrace (12), (23) \rbrace$ such that $u_{i,\sigma_i(1)}=v_{i,1}$,  $u_{i,\sigma_i(2)}=v_{i,2}$ and $u_{i,\sigma_i(3)}=v_{i,3}$.
\end{itemize}
Then $G$ contains a collection $\mathcal{P} := \lbrace P_1, \ldots, P_{2\ell} \rbrace$ of vertex-disjoint square paths such that, for all $1 \leq s \leq \ell$, the following hold.
\begin{itemize}
\item[(P1)] $|P_s| , |P_{\ell+s}| \leq 11/\eta$;
\item[(P2)] $P_s$ is $\eta$-tail-heavy and $P_{\ell+s}$ is $\eta$-head-heavy;
\item[(P3)] for  $k =1,2$, there are sets $A^s_k \subseteq V_{s,k} $ such that $|A^s_k| \geq cm/2$, with the following property:
for any $x_k \in A^s_k$ where $x_1x_2 \in E(G)$ we have that $x_1x_2P_s$ is a square path;
\item[(P4)] for  $j =3,1$, there are sets $B^s_j \subseteq V_{s,j} $ such that $|B^s_j| \geq cm/2$, with the following property:
for any $y_j \in B^s_j$ where $y_3y_1 \in E(G)$ we have that $P_{\ell+s}y_3y_1$ is a square path.
\end{itemize}
\end{lemma}

\begin{proof}
Suppose, for some $1 \leq r \leq 2\ell$, we have obtained a collection $\mathcal{P}' = \lbrace P_1, \ldots, P_{r-1} \rbrace$ of vertex-disjoint square paths, such that each $P_i$ with $1 \leq i \leq r-1$ satisfies the required properties.
We will find a suitable embedding of $P_r$ into $G$.

Observe that
\begin{equation}\label{nmL2}
3m\ell \leq n = 3m\ell + |V_0| \leq 3m\ell + \eps n \leq 4m\ell.
\end{equation}
For $(i,j) \in [\ell] \times [3]$, let $V_{i,j}' := V_{i,j} \setminus \bigcup_{P \in \mathcal{P}'}V(P)$.
Then
\begin{equation}\label{Vj'size}
|V_{i,j} \setminus V_{i,j}'| \leq \frac{11}{\eta}(r-1) \leq \frac{22}{\eta}\ell \leq \frac{22\eps^2}{\eta}\frac{n}{\ell} \stackrel{(\ref{nmL2})}{\leq} \frac{88\eps^2}{\eta}m \leq \frac{\eps m}{3}.
\end{equation}
Proposition~\ref{superslice2}(i) implies that $G[V_{i,j}',V_{i',j'}']$ is $(2{\eps},d/2)$-regular whenever $(i,j)(i',j') \in E(R)$. 

Define $V'_0$ so that $\lbrace V'_0 \rbrace \cup \lbrace V'_{i,j} : (i,j) \in [\ell] \times [3] \rbrace$ is a partition of $V(G)$. Thus, (\ref{Vj'size}) implies that 
$|V'_0|\leq \eps n +\eps m \ell \leq 2 \eps n$. We can view the vertices $(i,j)$ in $R$ as corresponding to the clusters $V'_{i,j}$. In particular, if $d^{2c}_{R,G}(V_{i,j}) \geq (2/3+\eta)3\ell$ then (\ref{Vj'size}) implies
that $d^{c}_{R,G}(V'_{i,j}) \geq (2/3+\eta)3\ell$. 

We will consider three cases  depending on the properties required of $P_r$.

\medskip
\noindent
\textbf{Case 1.}
\emph{$1 \leq r \leq \ell$.}

\medskip
\noindent
Apply Lemma~\ref{babyhorror} to $G$ with $V'_0,V_{i,j}',2{\eps},c/2,d/2,\eta,F_r$ playing the roles of $V_0,V_{i},\eps,c,d,\eta,F$.
Thus Lemma~\ref{babyhorror}(1) implies that $G$ contains an $\eta$-tail-heavy square path $P_r$ with $|P_r| \leq 11/\eta$ and sets $A^s_k \subseteq V_{s,k}' \setminus V(P_r)$ for $k=1,2$ with $|A^s_k|\geq cm/2$ such that for any $x_k \in A^s_k$ where $x_1x_2 \in E(G)$, we have that $x_1x_2P_r$ is a square path.
Note that $P_r$ shares no vertex with any square path we have previously embedded (since it is disjoint from $V'_0$).
Therefore $P_r$ has the required properties.

\medskip
\noindent
\textbf{Case 2.}
\emph{$\ell+1 \leq r \leq 2\ell$ and $\sigma_{r-\ell} = (23)$.}

\medskip
\noindent
Let $s := r-\ell$.
The first three vertices in $F'_s$ are $(s,1),(s,3),(s,2)$ in that order.
Apply Lemma~\ref{babyhorror} to $G$ with $V'_0,V_{i,j}',2{\eps},c/2,d/2,\eta,F'_s$ playing the roles of $V_0,V_{i},\eps,c,d,\eta,F$.
Thus Lemma~\ref{babyhorror}(1) implies that $G$ contains an $\eta$-tail-heavy square path $Q_s$ with $|Q_s| \leq 11/\eta$ and sets $B^s_{j} \subseteq V_{s,j} \setminus V(Q_s)$ for $j=1,3$ with $|B^s_j|\geq cm/2$ such that for any $y_j \in B^s_{j}$ where $y_1y_3 \in E(G)$, we have that $y_1y_3Q_s$ is a square path.
Note that $Q_s$ shares no vertex with any square path we have previously embedded (since it is disjoint from $V'_0$).
Finally, observe that $P_r := Q_s^*$ is precisely the required square path.

\medskip
\noindent
\textbf{Case 3.}
\emph{$\ell+1 \leq r \leq 2\ell$ and $\sigma_{r-\ell} = (12)$.}

\medskip
\noindent
Let $s := r - \ell$.
The first three vertices of $F'_s$ are $(s,2),(s,1),(s,3)$ in that order.
Apply Lemma~\ref{babyhorror} to $G$ with $V'_0,V_{i,j}',2{\eps},c/2,d/2,\eta,F'_s$ playing the roles of $V_0,V_{i},\eps,c,d,\eta,F$.
Thus Lemma~\ref{babyhorror}(2) implies that $G$ contains an $\eta$-tail-heavy square path $Q_s$ with $|Q_s| \leq 11/\eta$ and sets $B^s_{j} \subseteq V_{s,j} \setminus V(Q_s)$ for $j=1,3$ such that for any $y_j \in B^s_{j}$ where $y_1y_3\in E(G)$, we have that $y_1y_3Q_s$ is a square path.
Note that $Q_s$ shares no vertex with any square path we have previously embedded.
Finally, observe that $P_r := Q_s^*$ is precisely the required square path (and $B^s_1,B^s_3$ the required sets).
\end{proof}

The final step in this section is to combine Theorem~\ref{trianglepack} and Lemmas~\ref{findfpath} and~\ref{horror} to prove Lemma~\ref{almostpath}.

\medskip
\noindent
\emph{Proof of Lemma~\ref{almostpath}.}
Without loss of generality, we may suppose that $0 < \eps \ll \eta \ll 1$ since proving the lemma for $\eps' \leq \eps$ implies the lemma for $\eps$.
Choose further constants $d, \alpha$ with $\eps \ll d \ll \alpha \ll \eta$.
Apply Theorem~\ref{trianglepack} to obtain $L_0 \in \mathbb{N}$ such that every $(\eta/2)$-good graph on $L \geq L_0$ vertices contains a perfect $K_3$-packing.
Without loss of generality, we may assume that $1/L_0 \ll \eps$, and that the conclusion of Lemma~\ref{findfpath} holds with $L_0/2,7d,\eta/4$ playing the roles of $n,\gamma,\eta$.
Apply Lemma~\ref{reg} with parameters $\eps' := \eps^5,L_0$ to obtain $M,n_0$.
Without loss of generality, assume that $1/n_0 \ll 1/M \ll 1/L_0$.
We therefore have the hierarchy
$$
0 < 1/n_0 \ll 1/M \ll 1/L_0 \ll \eps \ll d \ll \alpha \ll \eta \ll 1.
$$

Let $G$ be a graph of order $n \geq n_0$ such that $G$ is $\eta$-good.
Apply the Regularity lemma (Lemma~\ref{reg}) to $G$ with parameters $\eps',d,L_0$ to obtain clusters $V_1, \ldots, V_{L}$ of size $m$, an exceptional set $V_0$, a pure graph $G'$ and a reduced graph $R$.
So $|R|=L$ and $L_0 \leq L \leq M$ and $|V_0| \leq \eps' n$;
and $G'[V_j,V_{j'}]$ is $(\eps',d)$-regular whenever $jj' \in E(R)$.
By Lemma~\ref{reduced}(ii), $d^\alpha_{R,G}$ and $R$ are both $(\eta/2,L)$-good.
Moreover, Lemma~\ref{reduced}(i) implies that, for all $j \in V(R)$,
\begin{equation}\label{dalphaR}
d_R(j) \geq (1-6d)d^\alpha_{R,G}(j).
\end{equation}
Theorem~\ref{trianglepack} implies that $R$ contains a perfect $K_3$-packing $\mathcal{T}$.
So there exists an integer $\ell$ with
\begin{equation}\label{Lell}
0 \leq L-3\ell \leq 2
\end{equation}
so that $\mathcal{T} := \lbrace T_1, \ldots, T_\ell \rbrace$ contains exactly $\ell$ triangles.
Let $R' := R[V(\mathcal{T})]$.
Then $\mathcal{T}$ is a $2$-regular spanning subgraph of $R'$.
We have that
\begin{equation}\label{nval}
n = mL + |V_0| \leq mL + \eps' n \stackrel{(\ref{Lell})}{\leq} m(3\ell+2) + \eps' n\ \ \text{ and so }\ \ n \leq 4m\ell.
\end{equation}
Relabel the vertices in $R'$ so that the $i$th triangle of $\mathcal{T}$ has vertex set $T_i := \lbrace(i,1),(i,2),(i,3)\rbrace$.
So $V(R') = [\ell] \times [3]$.
Relabel those clusters of $G$ which correspond to vertices of $R'$ by writing $X_{i,j}$ for the cluster corresponding to $(i,j)$.
Choose $X_0$ so that $\lbrace X_0 \rbrace \cup \lbrace X_{i,j} : (i,j) \in [\ell] \times [3] \rbrace$ is a partition of $V(G)$.
Note that
$|X_0| \leq |V_0|+2m \leq 2\eps'n$.

Notice that since $G'$ is the pure graph of $G$, the definition of core degree implies that for all $X = (i,j) \in V(R')$,
\begin{align}\label{newlab}d^\alpha _{R',G} (X) \geq d^\alpha _{R',G'} (X) \geq d^{\alpha} _{R',G} (X)-(d+\eps)|R'|
\end{align} and
 $(d^\alpha_{R',G}(X))/|R'| = (d^\alpha_{R,G}(X))/|R|$.
Thus, Proposition~\ref{stoneage} implies that $d^\alpha_{R',G'}$ and $R'$ are both $(\eta/4,L)$-good.
Then (\ref{dalphaR}) implies that, for all $X \in V(R')$, we have
$$
d_{R'}(X) \geq d_R(X)-2 \stackrel{(\ref{dalphaR})}{\geq} (1-6d)d_{R,G}^\alpha(X)-2 \geq (1-7d)d_{R',G}^\alpha(X) \stackrel{(\ref{newlab})}{\geq} (1-7d)d_{R',G'}^\alpha(X).
$$
Let
\begin{equation}\label{Xdef}
\mathcal{X} := \lbrace (i,j) \in V(R'): d^\alpha_{R',G'}((i,j)) \geq (2/3+\eta/4)L \rbrace.
\end{equation}
For each $1 \leq i \leq \ell$, apply Lemma~\ref{findfpath} with $R',3\ell,T_i, \eta/4, 7d, d^\alpha_{R',G'}$ playing the roles of $G,n,T,\eta, \gamma, d'_G$, to show that $R'$ contains
a folded path $F_i := v^i_1 \ldots v^i_{t_i}$ where $8 \leq t_i \leq 20/\eta$ and $\lbrace v^i_1, v^i_2, v^i_3 \rbrace = T_i$;
and $\lbrace v^i_{t_i-1}, v^i_{t_i} \rbrace \subseteq \mathcal{X}$.
Without loss of generality, we may assume that
$$
v_j^i = (i,j)\ \ \text{ for }\ \ (i,j) \in [\ell] \times [3].
$$
Moreover, 
for each $1 \leq i \leq \ell$, $R'$ contains
a folded path $F'_i := u^i_1 \ldots u^i_{s_i}$ where $8 \leq s_i \leq 20/\eta$, $\lbrace u^i_1, u^i_2, u^i_3 \rbrace = T_i$;
and $\lbrace u^i_{s_i-1}, u^i_{s_i} \rbrace \subseteq \mathcal{X}$. Further, there exists $\sigma_i \in \{(12), (23)\}$ such that
$$
u_j^i = (i,\sigma_i(j))\ \ \text{ for }\ \ (i,j) \in [\ell] \times [3].
$$
Therefore the properties (F1)--(F3) as stated in Lemma~\ref{horror} hold with $\eta/4$ playing the role of $\eta$.

Therefore Lemma~\ref{horror} applied with $R',G',X_0,X_{i,j},\eps',\alpha/2,d,\eta/4,\mathcal{X},F_i,F_i'$ playing the roles of $R,G,$ $V_0,V_{i,j},\eps,c,d,\eta,\mathcal{V},F_i,F_i'$, implies that $G'$ contains a collection $\mathcal{P} := \lbrace P_1, \ldots, P_{2\ell} \rbrace$ of vertex-disjoint square paths which satisfy (P1)--(P4) with $\eta/4,\alpha/2$ playing the roles of $\eta,c$ respectively.
In particular, (P1) implies that $|P| \leq 44/\eta$ for all $P \in \mathcal{P}$.

For each $1 \leq i \leq \ell$, write $[P_i]^+_2 =: u_iv_i$ and $[P_{\ell+i}]^-_2 =: v_i'u_i'$.
So (P2) implies that
$$
\lbrace u_i,v_i,u_i',v_i' \rbrace \subseteq V(G')_{\eta/4} \subseteq V(G)_{\eta/4}.
$$
Now Lemma~\ref{horror} implies that
\begin{equation}\label{shortpaths}
\sum_{P \in \mathcal{P}}|P|  \leq \frac{88\ell}{\eta} \leq \frac{\eps' m}{2}.
\end{equation}
Let $a,b \in V(G)_{\eta/4}$ be arbitrary.
By Propositions~\ref{2ndnbrhd}(i) and~\ref{largeset}(i),
$$
|N^2_G(a,b)_\eta| \geq \left(\frac{1}{3} + \frac{\eta}{2}\right)n - \frac{n}{3} = \frac{\eta n}{2} \stackrel{(\ref{shortpaths})}{>} \sum_{P \in \mathcal{P}}|P| + 4\ell.
$$
So we can find a collection $\lbrace w_i,x_i,w_i',x_i' : i \in [\ell] \rbrace$ of distinct vertices disjoint from $\mathcal P$ such that $u_iv_iw_ix_i$ is an $\eta$-tail-heavy square path in $G$, and $x_i'w_i'v_i'u_i'$ is an $\eta$-head-heavy square path in $G$.
Therefore, for $1 \leq i \leq \ell$, setting $Q_i := P_iw_ix_i$ and $Q_{\ell+i} := w_i'x_i'P_{\ell+i}$, we have that $\mathcal{Q} := \lbrace Q_1, \ldots, Q_{2\ell} \rbrace$ is a collection of vertex-disjoint square paths in $G$ such that $|Q| \leq 44/\eta+4 \leq 45/\eta$ for all $Q \in \mathcal{Q}$; for all $1 \leq i \leq \ell$ we have that $Q_i$ is $\eta$-tail-heavy and $Q_{\ell+i}$ is $\eta$-head-heavy; and $\mathcal{Q}$ satisfies (P3) and (P4) with $\alpha/2$ playing the role of $c$.
For each $1 \leq i \leq \ell$ and $k=1,2$, let $A_k^i\subseteq X_{i,k}$ be the sets guaranteed by (P3), and for each $j=3,1$, let $B_j^i\subseteq X_{i,j}$ be the sets guaranteed by (P4).
So $|A^i_k|,|B^i_j| \geq \alpha m/4$.

For $(i,j) \in [\ell] \times [3]$, let $X_{i,j}' := X_{i,j} \setminus \bigcup_{Q \in \mathcal{Q}}V(Q)$.
So
$$
(1-\eps')m \leq (1-\eps'/2)m - 4\ell \stackrel{(\ref{shortpaths})}{\leq} |X_{i,j}'| \leq m.
$$
Lemma~\ref{superslice2}(i) implies that, whenever $(i,j)(i',j') \in E(R')$, $G'[X_{i,j}',X_{i',j'}']$ is $(2{\eps'},d/2)$-regular.

Recall that $E(\mathcal{T}) = \lbrace (i,j)(i,j') : 1 \leq i \leq \ell, 1 \leq j < j' \leq 3 \rbrace$.
Apply Lemma~\ref{superslice} with $R',G',X_{i,j}',3\ell,\mathcal{T},2,2{\eps'},d/2$ playing the roles of $R,G,V_j,L,H,\Delta,\eps,d$ to obtain a collection $\lbrace Y_{i,j}: (i,j) \in [\ell] \times [3] \rbrace$ of disjoint subsets of $V(G)$ so that, for all $(i,j) \in [\ell] \times [3]$, $Y_{i,j} \subseteq X_{i,j}'$ (so $Y_{i,j} \cap \bigcup_{Q \in \mathcal{Q}}V(Q) = \emptyset$); $G'[Y_{i,j},Y_{i,j'}]$ is $(\eps'^{1/3},d/4)$-superregular for all $1 \leq i \leq \ell$ and $1 \leq j<j' \leq 3$; and
\begin{equation}\label{Xjsize}
|Y_{i,j}| =: m' \geq (1-\eps'^{1/3})m\ \ \text{ for all }\ \ (i,j) \in [\ell] \times [3].
\end{equation}
Lemma~\ref{horror}(P3) implies that, for $k \in \lbrace 1,2 \rbrace$,
$$
|A^{i}_k \cap Y_{i,k}| \geq (\alpha/4-\eps'^{1/3})m \geq \alpha m'/5,
$$
and similarly (P4) implies that, for $j \in \lbrace3,1\rbrace$, $|B^i_j \cap Y_{i,j}| \geq \alpha m'/5$.

Write $P^2_{3m'} = z_1 \ldots z_{3m'}$ for the square path on $3m'$ vertices.
Let $\phi_i: V(P^2_{3m'}) \rightarrow T_i$ be defined as follows.
For all integers $0 \leq j < m'$, we set $\phi_i(z_{3j+1})=(i,3)$, $\phi_i(z_{3j+2})=(i,1)$, and $\phi_i(z_{3j+3})=(i,2)$.
It is easy to check that, for all $1 \leq i \leq \ell$, $\phi_i$ is a graph homomorphism; and $|\phi^{-1}_i(x)| = m'$ for all $x \in T_i$.

For each $1 \leq i \leq \ell$ we will (independently) do the following.
Apply Theorem~\ref{blowup} to the subgraph of $G'$ spanned by $Y_{i,1} \cup Y_{i,2} \cup Y_{i,3}$, with $P^2_{3m'}$ playing the role of $H$ (so $\Delta := 4$) and $\phi_i$ playing the role of $\phi$.
(So the remaining parameters are given by $d/4,\alpha/5,\eps'^{1/3}$ playing the roles of $d,c,\eps$.)
Identify special vertices $z_1,z_2,z_{3m'-1},z_{3m'}$ to the corresponding special sets $B^i_3\cap Y_{i,3},B^i_1\cap Y_{i,1},A^i_1\cap Y_{i,1},A^i_2\cap Y_{i,2}$.

Thus obtain a square path
$$
S_i := x^i_{1,3}x^i_{1,1}x^i_{1,2}x^i_{2,3} \ldots x^i_{m',3}x^i_{m',1}x^i_{m',2}
$$
in $G'$
with $V(S_i) = Y_{i,1} \cup Y_{i,2} \cup Y_{i,3}$ such that
\begin{align*}
x^i_{1,3} \in B^i_3 \cap Y_{i,3};\ \ x^i_{1,1} \in B^i_1 \cap Y_{i,1};\ \ x^i_{m',1} \in A^i_1 \cap Y_{i,1}\ \ \text{ and }\ \ x^i_{m',2} \in A^i_2 \cap Y_{i,2}.
\end{align*}
Lemma~\ref{horror}(P3) implies that $x^i_{m',1}x^i_{m',2}Q_i$ is a square path and (P4) implies that $Q_{\ell+i}x^i_{1,3}x^i_{1,1}$ is  a square path.

Let $\mathcal{P}' := \lbrace Q_{\ell+i}S_iQ_i : 1 \leq i \leq \ell \rbrace$.
Observe that $\mathcal{P}'$ is a collection of vertex-disjoint square paths.
We saw earlier that
$Q_i$ is $\eta$-tail-heavy and $Q_{\ell+i}$ is $\eta$-head-heavy.
Therefore each path in $\mathcal{P}'$ is $\eta$-heavy.
Finally,
\begin{eqnarray*}
\sum_{P \in \mathcal{P}'}|P| &\geq& \sum_{1 \leq i \leq \ell}|S_i| = 3m'\ell \stackrel{(\ref{Xjsize})}{\geq} 3(1-\eps'^{1/3})m\ell \stackrel{(\ref{nval})}{\geq}  (1-\eps)n.
\end{eqnarray*}
This completes the proof of Lemma~\ref{almostpath}.
\hfill$\square$
\medskip


\section{Connecting heavy square paths into an almost spanning  square cycle}\label{connectingsec}

Lemma~\ref{almostpath} implies that we can obtain a collection $(P_i)_i$ of vertex-disjoint $\eta$-heavy square paths, which together cover almost all of the vertices of our $\eta$-good graph $G$.
The next stage -- the goal of this section -- is to connect these paths together into a square cycle, which necessarily covers almost all of the vertices of $G$.
Roughly speaking, we will show that one can connect square paths $P$ and $Q$ into a new square path whose initial segment is $P$ and whose final segment is $Q$, provided that $P$ is $\eta$-tail-heavy and $Q$ is $\eta$-head-heavy. 
This new square path will only contain a small number of vertices which do not lie in $P$ or $Q$.
Then, provided that the additional vertices lie outside of $(P_i)_i$, we can repeat this process to obtain an almost spanning square cycle.

Given a graph $G$ with $ab,cd \in E(G)$, we define an \emph{$(ab,cd)$-path} to be a square path $P$ in $G$ such that $[P]^-_2=ab$ and $[P]^+_2=cd$.
 Note that an $(ab,cd)$-path is not, for example, a $(ba,cd)$-path. Given a set of vertices $W$, we say that a square path $P$ \emph{avoids $W$} if $V(P) \cap W = \emptyset$.

\begin{definition}\emph{($\eta$-flexibility)}
Given $\eta > 0$, we say that a square path $P$ in a graph $G$ is \emph{$\eta$-head-flexible} if $P$ is $\eta$-head-heavy and $G[(P)^-_4] \cong K_4$.
We say that $P$ is \emph{$\eta$-tail-flexible} if $P$ is $\eta$-tail-heavy and $G[(P)^+_4] \cong K_4$.
If $P$ is both $\eta$-head- and $\eta$-tail-flexible, we say that it is \emph{$\eta$-flexible}.
We drop the prefix $\eta$- if it is clear from the context.
\end{definition}

This concept is useful for the following reason.
Suppose that $P = x_1 \ldots x_\ell$ is a tail-heavy square path and $\ell \geq 4$.
If $P$ is tail-flexible, then $P' := x_1 \ldots x_{\ell-2}x_\ell x_{\ell-1}$ is also a tail-heavy square path.
So we have more flexibility (in the literal sense) in connecting $P$ (or rather a square path containing the vertices of $P$) to another square path.

Our first aim will be to extend a tail-heavy square path to a tail-flexible square path.

\subsection{Finding flexible square paths}

Our aim in this subsection is to prove the following lemma, which implies that, given a tail-heavy square path $P$ and a head-heavy square path $P'$, either $P$ and $P'$ can be `connected' or $P$ and $P'$ can be extended to tail- and head-flexible square paths respectively. Recall that in an $\eta$-good graph $G$ on $n$ vertices, $V(G)_\eta$ is the set of all vertices $x \in V(G)$ with $d_G (x)\geq (2/3+\eta)n$.

\begin{lemma}\label{flexlemma}
Let $n \in \mathbb{N}$ and $\eta > 0$ such that $0 < 1/n \ll \eta < 1$.
Suppose that $G$ is an $\eta$-good graph on $n$ vertices.
Let $a,b,c,d$ be distinct vertices in $V(G)_{\eta}$, and suppose that $ab,cd \in E(G)$. 
Let $W \subseteq V(G) \setminus \lbrace a,b,c,d \rbrace$ with $|W| \leq \eta n/8$.
Suppose that $G$ contains no $(ab,cd)$-path $P$ such that $|P| \leq 17$ and $P$ avoids $W$.
Then there exist square paths $S,S'$ such that all of the following hold.
\begin{itemize}
\item[(i)] $[S]^-_2=ab$, $[S']^+_2=cd$ and $S,S'$ avoid $W$;
\item[(ii)] $|S|,|S'| \leq 10$ and $V(S) \cap V(S') = \emptyset$;
\item[(iii)] $S$ is $\eta$-tail-flexible and $S'$ is $\eta$-head-flexible.
\end{itemize}
\end{lemma}

\begin{proof}
Throughout the proof, we will write tail-flexible (head-flexible) for $\eta$-tail-flexible ($\eta$-head-flexible) and will similarly write tail-heavy and head-heavy. We
say that a square path $S$ is \emph{$ab$-good} if $|S| \leq 10$; $[S]_2^-=ab$; $S$ avoids $W\cup \{c,d\}$; and $S$ is tail-flexible.
Analogously, we say that a path $S'$ is \emph{$cd$-good} if $|S'| \leq 10$; $[S']_2^+=cd$; $S'$ avoids $W\cup\{a,b\}$; and $S'$ is head-flexible.
Suppose that 
$G$ contains no pair $S,S'$ of vertex-disjoint square paths such that $S$ is $ab$-good and $S'$ is $cd$-good. We must show that this implies that $G$ contains an $(ab,cd)$-path which has at most $17$ vertices and avoids $W$.

Suppose that there is a square path $S$ in $G$ that is $ab$-good. By our assumption any $cd$-good square path $S'$ in $G$ is such that $V(S) \cap V(S') \not = \emptyset$. By adding the vertices in $V(S) \setminus \lbrace a,b \rbrace$ to $W$ we now have that $|W|\leq \eta n/7$ and there is no $cd$-good square path $S'$ in $G$. Otherwise, we have that there is no  square path $S$ in $G$ that is $ab$-good (and $|W| \leq \eta n/8$).  Without loss of generality, assume that $|W|\leq \eta n/7$ and there is no $cd$-good square path $S'$ in $G$. (The proof in the other case is essentially identical.)

At every step of the proof, we will have two vertex-disjoint square paths $P,P'$ such that $[P]^-_2=ab$ and $[P']^+_2=cd$, and $|P|,|P'| \leq 8$; and a set $U := V(G) \setminus (W \cup V(P) \cup V(P'))$ which we call the \emph{surround} of $P,P'$.
Initially, we take $P := ab$ and $P' := cd$.
In each step, we modify $P,P'$ so that any new additional vertices were taken from $U$, and $P,P'$ still satisfy the specified properties.
Then we update the surround $U$ of the new $P,P'$.
Note that $P'$ is not head-flexible at any stage (otherwise it is $cd$-good).
Further,  in every step we have  $|U| \geq (1-\eta/4)n$.
Proposition~\ref{stoneage} implies that the graph with vertex set $V(G)$ containing every edge of $G$ with at least one endpoint in $U$ is $(\eta/2,n)$-good.
Moreover, for all $x \in V(G)$,
\begin{equation}\label{n'deg}
d_{G}(x,U) \geq d_G(x)-\eta n/4.
\end{equation}
Assume, for a contradiction, that there is no $(ab,cd)$-path in $G$ which has at most $17$ vertices and avoids $W$.

\medskip
\noindent
\textbf{Claim 1.}
\emph{
Suppose that $P,P'$ are vertex-disjoint square paths that avoid $W$ with $|P|,|P'| \leq 8$ and $[P]^-_2=ab, [P']^+_2=cd$.
Let $U$ be the surround of $P,P'$.
Then the following hold:
\begin{itemize}
\item[(A)] for any $4$-segment $x_1x_2y_1y_2$ of $P'$ with $x_1,x_2 \in V(G)_{\eta}$, we have $N_U^2(x_1,x_2) \cap N_U^2(y_1,y_2) = \emptyset$;
\item[(B)] for any $2$-segment $x_1x_2$ of $P'$ we have that $N_U^2(x_1,x_2)_{\eta}$ is an independent set in $G$;
\item[(C)] for any $2$-segments $x_1x_2,y_2y_1$ of $P,P'$ respectively, where $x_2y_2 \in E(G)$, we have that $N_U^2(x_1,x_2) \cap N_U^2(y_1,y_2) = \emptyset$. 
\end{itemize}
}

\medskip
\noindent
We now prove Claim~1.
If (A) does not hold, there is some $u \in N_U^2(x_1,x_2) \cap N_U^2(y_1,y_2)$ and then $G$ contains a $cd$-good path $Q$ (with $|Q| \leq |P'|+1 \leq 9$ and $[Q]^-_5=x_1x_2uy_1y_2$),  a contradiction.
If (B) does not hold, there is an edge $uv \in E(G[N_U^2(x_1,x_2)_{\eta}])$ and then $G$ contains a $cd$-good path $Q$ (with $|Q|\leq|P'|+2 \leq 10$ and $[Q]^-_4=uvx_1x_2$),  a contradiction.
If (C) does not hold, there is some $z \in N_U^2(x_1,x_2) \cap N_U^2(y_1,y_2)$ and then $G$ contains an $(ab,cd)$-path $Q$ with $|Q| \leq |P|+|P'|+1 \leq 17$ which avoids $W$,  a contradiction.
This completes the proof of the claim.

\medskip
\noindent
Observe that, by Propositions~\ref{2ndnbrhd}(i) and~\ref{largeset}(i), for all distinct $u,v \in V(G)_{\eta}$,
\begin{equation}\label{Nab}
|N^2_{U}(u,v)| \stackrel{(\ref{n'deg})}{\geq} 2(2/3+3\eta/4)n-n \geq (1/3+\eta)n\ \ \text{ and }\ \ |N^2_{U}(u,v)_{\eta}| \geq \eta n.
\end{equation}

\medskip
\noindent
\textbf{Claim 2.}
\emph{
There exist vertex-disjoint square paths $T,T'$ in $G$ such that
$|T|,|T'| \leq 5$; $T$ is tail-heavy and $T'$ is head-heavy; $[T]^-_2=ab$, $[T']^+_2=cd$; $T,T'$ avoid $W$; and the final vertex of $T$ is adjacent to the initial vertex of $T'$.
}

\medskip
\noindent
We now prove Claim~2.
Let $U$ be the surround of $ab,cd$.
By (\ref{Nab}), there exist $d' \in N^2_{U}(c,d)_{\eta}$ and $c' \in N^2_{U}(d',c)_{\eta}$ (which are necessarily distinct).
Then $ab$ and $c'd'cd$ are vertex-disjoint square paths avoiding $W$.
Remove $c',d'$ from $U$.
So $U$ is the surround of $ab,c'd'cd$.
Since $c',d' \in V(G)_{\eta}$, Claim~1(A) applied to $c'd'cd$ implies that $N^2_U(c',d') \cap N^2_U(c,d)=\emptyset$.
Let
$$
N := N^2_{U}(a,b)\ \ \text{ and } \ \ N' := N^2_{U}(c',d') \cup N^2_{U}(c,d).
$$
Then
$$
|N'|= |N^2_{U}(c',d')| + |N^2_{U}(c,d)| \stackrel{(\ref{Nab})}{\geq} (2/3+2\eta)n.
$$
Proposition~\ref{largeset}(i) implies that
\begin{equation}\label{N'L}
|N'_{\eta}| \geq (1/3+2\eta)n.
\end{equation}

Let $y \in N_{\eta}$ be arbitrary ($N_{\eta} \neq \emptyset$ by (\ref{Nab})).
Then 
$$
d_{G}(y,N'_\eta) \geq d_G(y)  - n + |N_{\eta}'| \stackrel{(\ref{N'L})}{\geq} (2/3+\eta)n - (2/3-2\eta)n = 3\eta n.
$$
So there is some $z \in N'_{\eta} \cap N_G(y)$.
Set
$T := aby$ and take $T' := zcd$ if $z \in N^2_{U}(c,d)$, or $T' := zc'd'cd$ if $z \in N^2_{U}(c',d')$.
This completes the proof of Claim~2.

\medskip
\noindent
Let $T$ and $T'$ be as in Claim~2.
Write $[T]^+_2:=wx$ and $[T']^-_2 := x'w'$, where $xx' \in E(G)$ (see Figure~\ref{connectfig}).
Let $t := |T| \leq 5$ and $t' := |T'| \leq 5$.
Let $U$ be the surround of $T,T'$ and let $Y := N^2_{U}(w,x)$ and $Y' := N^2_{U}(x',w')$.
Claim~1(C) applied with $T,T',wx,x'w'$ playing the roles of $P,P',x_1x_2,y_2y_1$ implies that $Y \cap Y' = \emptyset$. 
Therefore, by Proposition~\ref{largeset}(i),
$$
|(Y \cup Y')_{\eta}| \geq |Y \cup Y'| - n/3 = |Y|+|Y'|-n/3 \stackrel{(\ref{Nab})}{\geq} (1/3+2\eta)n.
$$ 
Now $G[(Y\cup Y')_{\eta}]$ contains no isolated vertices by Proposition~\ref{largeset}(ii).
Observe that $Y_{\eta}' \neq \emptyset$ by (\ref{Nab}).
Moreover, Claim~1(B) implies that $Y'_{\eta}$ is an independent set in $G$.
Therefore every vertex of $Y'_{\eta}$ has a neighbour in $Y_{\eta}$.
Choose $y' \in Y'_{\eta}$ and $y \in Y_{\eta}$ with $y'y \in E(G)$.

We have obtained vertex-disjoint square paths
\begin{equation}\label{Ty}
Ty = [T]^-_{t-2}wxy\ \ \text{ and }\ \ y'T' = y'x'w'[T']^+_{t'-2}\ \ \text{ such that }\ \ xx',yy' \in E(G),
\end{equation}
and $Ty$ is tail-heavy and $y'T'$ is head-heavy.
Remove $y,y'$ from $U$.
So $U$ is the surround of $Ty,y'T'$.
Let $Z := N^2_{U}(x,y)$ and $Z' := N^2_{U}(y',x')$.
Claim~1(C) applied with $Ty,y'T',xy,y'x'$ playing the roles of $P,P',x_1x_2,y_2y_1$ implies that
\begin{equation}\label{ZZ'}
Z \cap Z' = \emptyset.
\end{equation}
Let $A^{xy} := Z \cap N^1_{U}(y',x')$ be the set of vertices in $U$ adjacent to both $x,y$ and at least one of $y',x'$.
Define $A^{y'x'} := Z' \cap N^1_{U}(x,y)$ similarly.
So certainly
\begin{equation}\label{A'L}
A^{y'x'}_{\eta} \subseteq Z'_{\eta}.
\end{equation}
\medskip
\noindent
\textbf{Claim 3.}
\emph{$E(G[A^{xy}_{\eta},Z_{\eta}']) \neq \emptyset$.}

\medskip
\noindent
Now (\ref{ZZ'}) and (\ref{disjoint}) imply that
\begin{equation}\label{intersections}
\emptyset = N^4_{U}(x,y,y',x') = A^{xy} \cap A^{y'x'} = Z' \cap A^{xy} = Z \cap A^{y'x'}.
\end{equation}
Let $A := N^3_{U}(x,y,y',x')$.
Observe that $A = A^{xy} \cup A^{y'x'}$ and $U \cap \lbrace x,y,y',x' \rbrace = \emptyset$.
Then
$$
(8/3+3\eta)n \stackrel{(\ref{n'deg})}{\leq} \sum_{v \in \lbrace x,y,y',x' \rbrace}d_G(v,U) = \sum_{u \in U}d_G(u,{\lbrace x,y,y',x' \rbrace}) \stackrel{(\ref{intersections})}{\leq} 3|A| + 2(n-|A|) = |A| + 2n,
$$
and hence $|A| \geq (2/3+3\eta)n$. By Proposition~\ref{largeset}(i),
\begin{equation}\label{AL}
|A_{\eta}| \geq (1/3+3\eta)n.
\end{equation}
(\ref{Nab}) implies that
$|Z'_{\eta}| \geq \eta n$.
Claim~1(B) applied with $Ty,y'T',y'x'$ playing the roles of $P,P',x_1x_2$ implies that $Z_{\eta}'$ is an independent set in $G$.
Suppose that $E(G[A^{xy}_{\eta},Z'_{\eta}]) = \emptyset$.
Then no vertex in $Z'_{\eta}$ has a neighbour in $A^{xy}_{\eta}$.
Therefore, for all $z \in Z_{\eta}'$, $N_{U}(z) \cap ((A^{xy} \cup Z')_{\eta}) = \emptyset$.
So
\begin{eqnarray*}
|A_{\eta}| &\stackrel{(\ref{intersections})}{=}& |A_{\eta}^{xy}| + |A_{\eta}^{y'x'}| \stackrel{(\ref{A'L})}{\leq} |A^{xy}_{\eta}| + |Z'_{\eta}| \stackrel{(\ref{intersections})}{=} |(A^{xy} \cup Z')_{\eta}| \leq |U \setminus N_{U}(Z'_{\eta})| \leq n - \max_{z \in Z_{\eta}'}d_{G}(z,U)\\
&\stackrel{(\ref{n'deg})}{\leq}& (1/3-3\eta/4)n,
\end{eqnarray*}
a contradiction to (\ref{AL}).
This proves Claim~3.

\medskip
\noindent
By Claim 3, we may choose $z \in A^{xy}_{\eta}$ and $z' \in Z'_{\eta}$ such that $zz' \in E(G)$.
We have shown that $G$ contains vertex-disjoint $W$-avoiding square paths
\begin{align}\label{Tyz}
&Tyz=[T]^-_{t-2}wxyz\ \ \text{ and}\ \ z'y'T'=z'y'x'w'[T']^+_{t'-2};
\end{align}
such that $xx',yy',zz' \in E(G)$ and one of $zy',zx' \in E(G)$;
where $\lbrace w,x,y,z,z',y',x',w' \rbrace \subseteq V(G)_{\eta}$ (see Figure~\ref{connectfig}).
Claim~2 implies that $|Tyz|,|z'y'T'| \leq 7$.
Remove $z,z'$ from $U$.
So $U$ is the surround of $Tyz,z'y'T'$.

\tikzstyle{every node}=[fill=black,draw,circle,minimum width=1pt,outer sep = 0, inner sep = 1]

\begin{center}
\begin{figure}
\begin{tikzpicture}
\clip (-3.5,-1.5) rectangle (11.5,1);

\begin{scope}

\node[label=below:$w$] (w) at (0,0) {};
\node[label=below:$x$] (x) at (1,0) {};
\node[label=below:$y$] (y) at (2,0) {};
\node[label=below:$z$] (z) at (3,0) {};

\node[label=below:$z'$] (z') at (5,0) {};
\node[label=below:$y'$] (y') at (6,0) {};
\node[label=below:$x'$] (x') at (7,0) {};
\node[label=below:$w'$] (w') at (8,0) {};

\node[label=below:$b$] (b) at (-2,0) {};
\node[label=below:$a$] (a) at (-3,0) {};
\node[label=below:$c$] (c) at (10,0) {};
\node[label=below:$d$] (d) at (11,0) {};

\begin{scope}
\clip (9.75,-1) rectangle (13,3);
\draw[thick] (d) .. controls (10,0.5) .. (9,0);
\draw[thick] (c) .. controls (9,0.5) .. (8,0);
\end{scope}

\begin{scope}
\clip (9.5,-3) rectangle (9.75,1);
\draw[thick,dotted] (d) .. controls (10,0.5) .. (9,0);
\draw[thick,dotted] (9.75,0) -- (9.5,0);
\draw[thick,dotted] (c) .. controls (9,0.5) .. (8,0);
\end{scope}

\begin{scope}
\clip (-4,6) rectangle (-1.75,-1);
\draw[thick] (a) .. controls (-2,0.5) .. (-1,0);
\draw[thick] (b) .. controls (-1,0.5) .. (w);
\end{scope}

\begin{scope}
\clip (-1.75,-1) rectangle (-1.5,1);
\draw[thick,dotted] (a) .. controls (-2,0.5) .. (-1,0);
\draw[thick,dotted] (-1.75,0) -- (-1.5,0);
\draw[thick,dotted] (b) .. controls (-1,0.5) .. (w);
\end{scope}

\draw[thick] (a) -- (-1.75,0);
\draw[thick] (d) -- (9.75,0);

\draw[thick,dotted] (-1.25,0) -- (-0.75,0);
\draw[thick,dotted] (8.75,0) -- (9.25,0);

\draw[thick] (-0.25,0) -- (z);
\draw[thick] (z') -- (8.25,0);

\draw[color=blue,thick] (x) -- (z);
\draw[color=blue,thick] (z') -- (x');

\draw[color=red] (z') -- (z);

\draw[color=blue,thick] (y) .. controls (1,0.5) .. (w);
\draw[color=blue,thick] (z) .. controls (2,0.5) .. (x);

\draw[color=blue,thick] (z') .. controls (6,0.5) .. (x');
\draw[color=blue,thick] (y') .. controls (7,0.5) .. (w');

\draw[color=black] (x) .. controls (4,-1.5) .. (x');
\draw[color=red] (y) .. controls (4,-1) .. (y');
\draw[color=red] (z) .. controls (4,-0.5) .. (z');

\draw[color=red,dashed] (z) .. controls (4.5,0.5) .. (y');
\draw[color=red,dashed] (z) .. controls (5,1) .. (x');

\begin{scope}
\clip (-0.25,-1) rectangle (3,3);
\draw[thick] (x) .. controls (0,0.5) .. (-1,0);
\end{scope}

\begin{scope}
\clip (6,6) rectangle (8.25,-1);
\draw[thick] (x') .. controls (8,0.5) .. (9,0);
\end{scope}

\begin{scope}
\clip (-0.5,-3) rectangle (-0.25,1);
\draw[thick,dotted] (x) .. controls (0,0.5) .. (-1,0);
\draw[thick,dotted] (-0.25,0) -- (-0.5,0);
\end{scope}

\begin{scope}
\clip (8.25,-1) rectangle (8.5,1);
\draw[thick,dotted] (x) .. controls (8,0.5) .. (9,0);
\draw[thick,dotted] (8.25,0) -- (8.5,0);
\end{scope}

\draw[] (-3,-1) -- (-1.3,-1);
\draw[<-] (1,-1) -- (-0.7,-1);
\node[draw=none,fill=none,label=$T$] at  (-1,-1.25) {}; 

\draw[] (7,-1) -- (8.7,-1);
\draw[<-] (11,-1) -- (9.3,-1);
\node[draw=none,fill=none,label=$T'$] at  (9,-1.25) {}; 

\end{scope}

\end{tikzpicture}

\caption{The structure obtained at (\ref{Tyz}). We first obtain the black edges (Claim~2), then the blue edges, then the red edges, where at least one of the dashed red edges is present after (\ref{Tyz}).}\label{connectfig}
\end{figure}
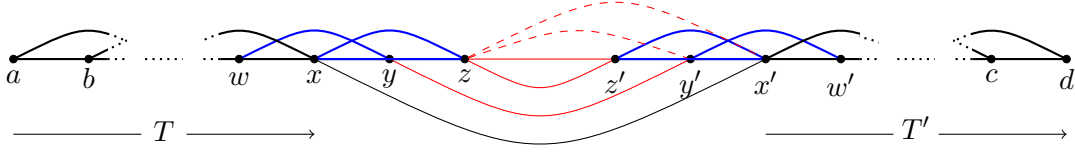
\end{center}

We consider two cases, depending on whether $zx' \in E(G)$ or $zy' \in E(G)$.

\medskip
\noindent
\textbf{Case 1.}
\emph{$zx' \in E(G)$.}

\medskip
\noindent
We will apply Claim~1(A) and~(C) with $Tyz,z'y'T'$ playing the roles of $P,P'$.
Claim~1(A) applied with $z'y'x'w'$ playing the role of $x_1x_2y_1y_2$ implies that $N^2_U(z',y') \cap N^2_U(x',w') = \emptyset$.
Claim~1(C) applied with $yz,z'y'$ playing the roles of $x_1x_2,y_2y_1$ implies that $N^2_U(y,z) \cap N^2_U(z',y') = \emptyset$.
Claim~1(C) applied with $yz,x'w'$ playing the roles of $x_1x_2,y_2y_1$ implies that $N^2_U(y,z) \cap N^2_U(x',w') = \emptyset$.
Therefore $N^2_U(y,z),N^2_U(z',y'),N^2_U(x',w')$ are pairwise vertex-disjoint subsets of $U$.
But (\ref{Nab}) implies that each set has size at least $(1/3+\eta)n$, a contradiction.
So we are done in Case~1.

\medskip
\noindent
\textbf{Case 2.}
\emph{$zy' \in E(G)$.}

\medskip
\noindent
This case is similar to Case~1.
Observe that now $Ty,zz'y'T'$ are vertex-disjoint square paths each containing at most eight vertices, and $U$ is the surround of $Ty,zz'y'T'$.
We will apply Claim~1(A) and~(C) with $Ty,zz'y'T'$ playing the roles of $P,P'$.
Claim~1(A) applied with $zz'y'x'$ playing the role of $x_1x_2y_1y_2$ implies that $N^2_U(z,z') \cap N^2_U(y',x') = \emptyset$.
Claim~1(C) applied with $xy,zz'$ playing the roles of $x_1x_2,y_2y_1$ implies that $N^2_U(x,y) \cap N^2_U(z,z') = \emptyset$.
Claim~1(C) applied with $xy,y'x'$ playing the roles of $x_1x_2,y_2y_1$ implies that $N^2_U(x,y) \cap N^2_U(y',x') = \emptyset$.
Therefore $N^2_U(x,y),N^2_U(z,z'),N^2_U(y',x')$ are pairwise vertex-disjoint subsets of $U$.
But (\ref{Nab}) implies that each set has size at least $(1/3+\eta)n$, a contradiction.
So we are done in Case~2.

\medskip
\noindent
In both cases we obtain a contradiction to our assumption that there is no $(ab,cd)$-path in $G$ which has at most $17$ vertices and avoids $W$.
This completes the proof of the lemma.
\end{proof}

\subsection{Connecting flexible square paths}

The proof of the next result is similar to that of Lemma~21 in~\cite{cdk} (although there the graph $G$ has minimum degree not much less than $2n/3$ and is `non-extremal').

\begin{lemma}[Connecting lemma]\label{abcd}
Let $n \in \mathbb{N}$ and $\delta,\eta > 0$ such that $0 < 1/n \ll \delta \ll \eta < 1$.
Suppose that $G$ is an $\eta$-good graph on $n$ vertices.
Let $a',b',c',d'$ be distinct vertices in $V(G)_{\eta}$ where $a'b',c'd' \in E(G)$.
Let $W \subseteq V(G) \setminus \lbrace a',b',c',d' \rbrace$ with $|W| \leq \delta n$.
Then $G$ contains an $(a'b',c'd')$-path on at most $22$ vertices which avoids $W$.
\end{lemma}

\begin{proof}
Suppose that $G$ contains no $(a'b',c'd')$-path on at most $17$ vertices which avoids $W$.
Apply Lemma~\ref{flexlemma} to obtain vertex-disjoint square paths $Q,Q'$ such that (i)--(iii) hold. (Where $a'b',c'd', Q,Q'$ play the roles of $ab,cd,S,S'$ respectively.)
Let $q := |Q| \leq 10$ and $q' := |Q'| \leq 10$.
Write $[Q]^+_2 =: ab$ and $[Q']^-_2 =: cd$ and set $X := \lbrace a,b,c,d \rbrace \subseteq V(G)_{\eta}$.
Let $U := V(G) \setminus (W \cup V(Q) \cup V(Q'))$.
Observe that $X \cap U = \emptyset$, and $|U| \geq (1-2\delta)n$.
Proposition~\ref{stoneage} implies that $G[U]$ is $(\eta/2,n)$-good.
Moreover, for all $x \in V(G)$,
\begin{equation}\label{Udeg2}
d_{G}(x,U) \geq d_G(x)-\eta n/2.
\end{equation}

\medskip
\noindent
\textbf{Claim.}
\emph{It suffices to find a path $P$ with $(P)^-_2=\lbrace a,b \rbrace$, $(P)^+_2=\lbrace c,d \rbrace$; $V(P) \setminus X \subseteq U$ and $|P| \leq 6$.}

\medskip
\noindent
To prove the claim, suppose we have such a path $P$.
Note that
$$
[Q]^-_{q-2}ab,\ \ [Q]^-_{q-2}ba,\ \ cd[Q']^+_{q'-2},\ \ dc[Q']^+_{q'-2}
$$
are square paths by Lemma~\ref{flexlemma}(iii).
Then $P' := [Q]^-_{q-2}P[Q']^+_{q'-2}$ is an $(a'b',c'd')$-path which avoids $W$ by Lemma~\ref{flexlemma}(i).
Finally, Lemma~\ref{flexlemma}(ii) implies that $|P'| \leq |Q|+|Q'|+|P|-4 \leq 22$.
This completes the proof of the claim.

\medskip
\noindent
For all $1 \leq i \leq 4$, let $S_i := \lbrace v \in U: d_G(v,X)=i \rbrace$.
Then
$$
(8/3+2\eta)n \stackrel{(\ref{Udeg2})}{\leq} \sum_{x \in X}d_G(x,U) = \sum_{u \in U}d_G(u,X) = \sum_{1 \leq i \leq 4}i|S_i| \leq 4|S_4|+3|S_3|+2(n-|S_3|-|S_4|),
$$
and therefore
\begin{equation}\label{Sbound}
|S_3|+2|S_4| \geq (2/3+\eta)n.
\end{equation}
Suppose that there is some $xy \in E(G[S_4,S_3 \cup S_4])$.
Then $G$ contains a square path $P$ with $V(P) = \lbrace a,b,x,y,c,d \rbrace$ which satisfies the claim.
(Indeed, if for example $a \notin N_G(y)$, then we can take $P := abxycd$ or $P := abxydc$; or if $c \notin N_G(x)$, then we can take $P := abxydc$ or $P := baxydc$.
The other cases are similar.)
Therefore we may assume that
\begin{equation}\label{S3S4}
E(G[S_4,S_3 \cup S_4]) = \emptyset.
\end{equation}
Suppose that $S_4 \neq \emptyset$.
Proposition~\ref{largeset}(iv) applied with $G[U],\eta/2,S_4,S_3 \cup S_4$ playing the roles of $G,\eta,X,Y$ implies that $|S_4|+(|S_3|+|S_4|) \leq (2/3-\eta/2)n$, a contradiction to (\ref{Sbound}).
Therefore (\ref{Sbound}) implies that
\begin{equation}\label{S3}
S_4 = \emptyset\ \ \text{ and }\ \ |S_3| \geq (2/3+\eta)n.
\end{equation}
Let
$$
T_{ab} := N^2_{U}(a,b) \cap S_3\ \ \text{ and }\ \ T_{cd} := N^2_{U}(c,d) \cap S_3.
$$

Suppose that there exists $x \in T_{ab}$ and $y \in T_{cd}$ such that $xy \in E(G)$.
Then $G$ contains a square path $P$ with $V(P) = \lbrace a,b,x,y,c,d \rbrace$ which satisfies the claim.
(For example, if $x \in N^3_{U}(a,b,c)$ and $y \in N^3_{U}(a,c,d)$ then
we can take $P := baxycd$, as in Figure~\ref{pathfig}. Observe that in this case and the other three cases, there is exactly one such $P$.)
So we may assume that 
\begin{equation}\label{Ts}
T_{ab} \cap T_{cd} = \emptyset\ ;\ \ T_{ab} \cup T_{cd} = S_3\ \ \text{ and }\ \ E(G[T_{ab},T_{cd}]) = \emptyset.
\end{equation}
(The first two assertions follow from (\ref{S3}) and the definitions.)
Now we will always obtain a contradiction.
Proposition~\ref{largeset}(iv) applied with $G[U],\eta/2,T_{ab},T_{cd}$ playing the roles of $G,\eta,X,Y$ implies that, if $T_{ab},T_{cd}$ are both non-empty, then $|T_{ab}|+|T_{cd}| \leq (2/3-\eta/2)n$, a contradiction to (\ref{S3}) and (\ref{Ts}).
Without loss of generality, we may assume that $T_{ab}=\emptyset$.
Therefore $|T_{cd}| \geq (2/3+\eta)n$.
Now, by Proposition~\ref{2ndnbrhd}(i) and (\ref{Udeg2}), we have that $|N^2_{U}(a,b)| \geq 2(2/3+\eta/2)n-n \geq (1/3+\eta)n$.
But together with~(\ref{S3}), this implies that $T_{ab} \neq \emptyset$, a contradiction.
\end{proof}

\begin{center}
\begin{figure}

\tikzstyle{every node}=[fill=black,draw,circle,minimum width=1pt,outer sep = 0, inner sep = 1]

\begin{tikzpicture}
\clip (-3.5,-1.5) rectangle (11.5,1.4);

\begin{scope}

\node[] (w) at (0,0) {};
\node[] (x) at (1,0) {};
\node[label={[anchor=south]above:$a$}] (a) at (2,1) {};
\node[label=below:$b$] (b) at (2,-1) {};

\node[label=above:$c$] (c) at (6,1) {};
\node[label=below:$d$] (d) at (6,-1) {};
\node[] (x') at (7,0) {};
\node[] (w') at (8,0) {};

\node[anchor=base,label={below:$b'$}] (b') at (-2,0) {};
\node[anchor=base,label={below:$a'$}] (a') at (-3,0) {};
\node[label={[anchor=south]below:$c'$}] (c') at (10,0) {};
\node[label={[anchor=south]below:$d'$}] (d') at (11,0) {};

\begin{scope}
\clip (9.75,-1) rectangle (13,3);
\draw[thick] (d') .. controls (10,0.5) .. (9,0);
\draw[thick] (c') .. controls (9,0.5) .. (w');
\end{scope}

\begin{scope}
\clip (9.5,-3) rectangle (9.75,1);
\draw[thick,dotted] (d') .. controls (10,0.5) .. (9,0);
\draw[thick,dotted,color=red] (9.75,0) -- (9.5,0);
\draw[thick,dotted] (c') .. controls (9,0.5) .. (w');
\end{scope}

\begin{scope}
\clip (-4,6) rectangle (-1.75,-1);
\draw[thick] (a') .. controls (-2,0.5) .. (-1,0);
\draw[thick] (b') .. controls (-1,0.5) .. (w);
\end{scope}

\begin{scope}
\clip (-1.75,-1) rectangle (-1.5,1);
\draw[thick,dotted] (a') .. controls (-2,0.5) .. (-1,0);
\draw[thick,dotted,color=red] (-1.75,0) -- (-1.5,0);
\draw[thick,dotted] (b') .. controls (-1,0.5) .. (w);
\end{scope}

\draw[thick,color=red] (a') -- (-1.75,0);
\draw[thick,color=red] (d') -- (9.75,0);

\draw[thick,dotted,color=red] (-1.25,0) -- (-0.75,0);
\draw[thick,dotted,color=red] (8.75,0) -- (9.25,0);

\draw[thick,color=red] (-0.25,0) -- (x);
\draw[thick,color=red] (x') -- (8.25,0);

\draw[thick,color=red] (x) -- (b) -- (a);
\draw[thick] (w) -- (b);
\draw[thick] (w) -- (a) -- (x);

\draw[thick,color=red] (x') -- (d) -- (c);
\draw[thick] (w') -- (d);
\draw[thick] (w') -- (c) -- (x');

\begin{scope}
\clip (-0.25,-1) rectangle (3,3);
\draw[thick] (x) .. controls (0,0.5) .. (-1,0);
\end{scope}

\begin{scope}
\clip (6,6) rectangle (8.25,-1);
\draw[thick] (x') .. controls (8,0.5) .. (9,0);
\end{scope}

\begin{scope}
\clip (-0.5,-3) rectangle (-0.25,1);
\draw[thick,dotted] (x) .. controls (0,0.5) .. (-1,0);
\draw[thick,dotted,color=red] (-0.25,0) -- (-0.5,0);
\end{scope}

\begin{scope}
\clip (8.25,-1) rectangle (8.5,1);
\draw[thick,dotted] (x) .. controls (8,0.5) .. (9,0);
\draw[thick,dotted,color=red] (8.25,0) -- (8.5,0);
\end{scope}

\node[draw=none,fill=none,label=$Q$] at  (-1,-1.25) {};

\node[draw=none,fill=none,label=$Q'$] at  (9,-1.25) {}; 

\node[label=below:$x$] (x) at (3.5,0) {};
\node[label=below:$y$] (y) at (4.5,0) {};

\draw[thick,color=red] (x) -- (y);
\draw[thick,color=red] (a) -- (x);
\draw[thick] (x) -- (b);
\draw[thick] (x) -- (c);
\draw[thick,color=red] (c) -- (y);
\draw[thick] (y) -- (d);
\draw[thick] (y) -- (a);

\end{scope}

\end{tikzpicture}
\caption{A tail-flexible path $Q$ and head-flexible path $Q'$ with $[Q]^+_2=ab$ and $[Q']^-_2=cd$ and adjacent vertices $x \in N^3(a,b,c) \subseteq T_{ab}$ and $y \in N^3(a,c,d) \subseteq T_{cd}$. The red line represents the ordering of a square path with vertex set $V(Q) \cup V(Q') \cup \lbrace x,y \rbrace$.}\label{pathfig}
\end{figure}
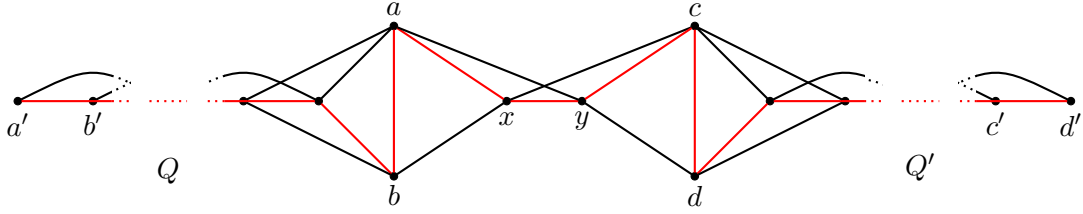
\end{center}

\subsection{An almost spanning square cycle}

The aim of this section is to prove Lemma~\ref{almostcycle}, which states that every sufficiently large $\eta$-good graph $G$ on $n$ vertices contains a square cycle that covers almost every vertex in $G$. The idea is to first apply Lemma~\ref{almostpath} to $G$ to find a collection $\mathcal P$ of heavy square paths that cover most of $G$. Then we repeatedly apply Lemma~\ref{abcd} to connect together these square paths into a single almost spanning square cycle in $G$. If we just apply Lemma~\ref{abcd} to $G$, then when connecting two square paths together we may be forced to use some vertices from other square paths from $\mathcal P$. To avoid this problem we in fact connect the square paths from $\mathcal P$ together using \emph{only} vertices from a small set $R \subseteq V(G)$ that is disjoint from $\mathcal P$. We refer to $R$ as a \emph{reservoir}. $R$ will be constructed in Lemma~\ref{reservoir} so that $G[R]$ `inherits' the degree sequence of $G$ (that is $G[R]$ is 
 $(\eta/2,|R|)$-good). This will allow us to apply Lemma~\ref{abcd} to $G[R]$ rather that $G$ itself. The idea of connecting paths through a reservoir has been used, for example, in~\cite{cdk, deore, lev, rodl}.

The hypergeometric random variable $X$ with parameters $(n,m,k)$ is
defined as follows. We let $N$ be a set of size $n$, fix $S \subseteq N$ of size
$|S|=m$, pick a uniformly random $T \subseteq N$ of size $|T|=k$,
then define $X=|T \cap S|$. Note that $\mathbb{E}X = km/n$.
To prove Lemma~\ref{reservoir} we will use the following standard Chernoff-type bound (see e.g.~Theorem 2.10 in~\cite{JLR}).

\begin{proposition}\label{chernoff}
Suppose $X$ has hypergeometric distribution and $0<a<3/2$. Then
$$\mathbb{P}(|X - \mathbb{E}X| \ge a\mathbb{E}X) \le 2 e^{-\frac{a^2}{3}\mathbb{E}X}.$$
\end{proposition}

\begin{lemma}\label{reservoir}\emph{(Reservoir lemma)}
Let $n \in \mathbb{N}$ and let $\delta,\eta > 0$ such that $0 < 1/n \ll \delta \ll \eta \ll 1$.
Suppose that $G$ is an $\eta$-good graph on $n$ vertices.
Then there exists $R \subseteq V(G)$ such that $|R|=\delta n$ and
\begin{itemize}
\item for all $v \in V(G)$ we have $d_G(v,R) \geq (d_G(v)/n-\eta/8)|R|$;
\item $G[R]$ is $(\eta/2,\delta n)$-good.
\end{itemize}
\end{lemma}

\begin{proof}
Choose $R \subseteq V(G)$ uniformly at random from all $\binom{n}{\delta n}$ subsets of $V(G)$ with size $\delta n$.
We first show that the probability that, for all $v \in V(G)$, we have
\begin{equation}\label{Rdeg}
d_G(v,R) \geq \left( 1 - \frac{\eta}{8} \right) \delta d_G(v)
\end{equation}
is more than $1/2$.
Indeed, for each $v \in V(G)$, we have
$$
\mathbb{E}(d_G(v,R)) = d_G(v)|R|/n = \delta d_G(v).
$$
Proposition~\ref{chernoff} implies that
$$
\mathbb{P}\left(d_G(v,R) < \left(1-\frac{\eta}{8}\right)\delta d_G(v)\right) \leq 2e^{-\eta^2\delta d_G(v)/192} < 2e^{-\eta^2\delta n/576} \leq e^{-\sqrt{n}}.
$$
(For the penultimate inequality, we used the fact that $G$ is $\eta$-good and so $\delta(G) > n/3$.)
So taking a union bound over all $v \in V(G)$, we see that the probability that some vertex fails to satisfy (\ref{Rdeg}) is at most $ne^{-\sqrt{n}} < 1/2$, as required.

Given $j,\lambda > 0$ and $H \subseteq H' \subseteq G$, let
$$
T_{j,\lambda}(H,H') := \lbrace x \in V(H): d_{H'}(x) \geq (1/3+\lambda)|H'|+j +1 \rbrace.
$$
Note that 
for all $\kappa \in \mathbb{R}$, whenever $j>\kappa|H'|$ and $\lambda+\kappa > 0$, we have
\begin{equation}\label{kappa}
T_{j,\lambda}(H,H')=T_{j-\kappa |H'|,\lambda+\kappa}(H,H').
\end{equation}
Observe that $H$ is $\lambda$-good if and only if, for all integers $1 \leq j \leq |H|/3$ we have $|T_{j,\lambda}(H,H)| \geq |H|- j+1$.
So, since $G$ is $\eta$-good, for all integers $1 \leq j \leq n/3$ we have $|T_{j,\eta}(G,G)| \geq n- j+1$.
Observe that
$$
\mathbb{E}(|T_{j,\eta}(G[R],G)|) = \delta |T_{j,\eta}(G,G)| \geq \delta(n-j).
$$
Proposition~\ref{chernoff} implies that, for fixed $1 \leq j \leq n/3$,
$$
\mathbb{P}\left(|T_{j,\eta}(G[R],G)| <\left(1-\frac{\eta}{8}\right)\delta(n- j)\right) \leq 2e^{-\eta^2\delta (n-j)/192} \leq 2e^{-\eta^2\delta n/288} \leq e^{-\sqrt{n}}.
$$
So the probability that $|T_{j,\eta}(G[R],G)|\leq (1-\eta/8)\delta (n-j)$ for some integer $1 \leq j \leq n/3$ is at most $ne^{-\sqrt{n}}/3 < 1/2$.

Thus there is some choice of $R$ such that, for all $v \in V(G)$, we have
\begin{equation}\label{prop1}
d_G(v,R) \geq \left(1-\frac{\eta}{8}\right)\delta d_G(v) = \frac{(1-\eta/8)|R|d_G(v)}{n}\geq \left( \frac{d_G(v)}{n} - \frac{\eta}{8} \right) |R|,
\end{equation}
and for all integers $1 \leq j \leq n/3$ we have
\begin{equation}\label{prop2}
|T_{j,\eta}(G[R],G)| \geq \left(1-\frac{\eta}{8}\right)\delta(n- j).
\end{equation}
To complete the proof, it remains to show that $G[R]$ is $(\eta/2,\delta n)$-good.
By an earlier observation, it suffices to show that
\begin{equation}\label{aim}
|T_{i,\eta/2}(G[R],G[R])| \geq \delta n - i+1 \ \ \text{ for all integers }\ 1 \leq i \leq \delta n/3.
\end{equation}
Let $x \in R$ be arbitrary.
Then (\ref{prop1}) and the fact that $G$ is $\eta$-good imply that 
$$
d_G(x,R) \geq \left( \frac{1}{3} + \frac{7\eta}{8} \right)\delta n \geq \left(\frac{1}{3}+\frac{5\eta}{6}\right)\delta n + 1.
$$
A simple rearrangement implies that $|T_{\delta\eta n/3,\eta/2}(G[R],G[R])| = |R|=\delta n$.
So, for all $1 \leq i \leq \delta\eta n/3$ we have that $|T_{i,\eta/2}(G[R],G[R])|= \delta n \geq \delta n - i +1$.
Thus, to show (\ref{prop2}), we may assume that $\delta\eta n/3 < i \leq \delta n/3$ for the remainder of the proof.

By definition, for all $x \in T_{i/\delta,2\eta/3}(G[R],G)$, we have that
$d_G(x) \geq (1/3+2\eta/3)n + i/\delta + 1$.
Therefore
(\ref{prop1}) implies that for such $x$,
$$
d_G(x,R) \geq \left(1-\frac{\eta}{8}\right)\delta \left( \left(\frac{1}{3} + \frac{2\eta}{3} \right) n + i/\delta + 1 \right) \geq \left(\frac{1}{3}+\frac{\eta}{2}\right)\delta n + i + 1.
$$
Thus
$$
T_{i,\eta/2}(G[R],G[R]) \supseteq T_{i/\delta,2\eta/3}(G[R],G).
$$
Therefore for all $\delta\eta n/3 < i \leq \delta n/3$,
\begin{eqnarray*}
|T_{i,\eta/2}(G[R],G[R])| &\geq& |T_{i/\delta,2\eta/3}(G[R],G)| \stackrel{(\ref{kappa})}{=} |T_{i/\delta-\eta n/3,\eta}(G[R],G)|\\
&\stackrel{(\ref{prop2})}{\geq}& \left(1-\frac{\eta}{8}\right)\delta\left(n - \frac{i}{\delta}+\frac{\eta n}{3}\right) = \delta n - i + \frac{\delta \eta n}{24}(5- \eta) + \frac{\eta i }{8} \geq \delta n-i+1.
\end{eqnarray*}
So (\ref{aim}) holds, as required.
\end{proof}

We will now combine Lemmas~\ref{almostpath},~\ref{abcd} and~\ref{reservoir} to prove the main result of this section.

\begin{lemma}\label{almostcycle}
Let $n \in \mathbb{N}$ and $0 < 1/n \ll \eps \ll \eta \ll 1$.
Then every $\eta$-good graph $G$ on $n$ vertices contains a square cycle $C$ with $|C| \geq (1-\eps)n$.
\end{lemma}

\begin{proof}
Apply Lemma~\ref{almostpath} with $\eta/2,\eps/2$ playing the roles of $\eta,\eps$ to obtain  $n_0, M \in \mathbb N$ such that every $(\eta/2)$-good graph $H$ on at least $ n_0$ vertices contains a collection of at most $M$ vertex-disjoint ($\eta/2$)-heavy square paths which together cover at least $(1-\eps/2)|H|$ vertices.
Note that we may assume that $1/n \ll 1/n_0 \ll 1/M \ll \eps$.
Further, choose $\delta$ so that $1/M \ll \delta \ll \eps$.

Apply Lemma~\ref{reservoir} to $G$ to obtain a set $R$
such that $|R|=\delta n$; for all $v \in V(G)$ we have
\begin{equation}\label{dRv}
d_G(v,R) \geq \left(\frac{d_{G}(v)}{n}-\frac{\eta}{8}\right)|R|;
\end{equation}
and $G[R]$ is $(\eta/2,\delta n)$-good.

Note that $|G \setminus R| = (1-\delta)n \geq (1-\eta/4)n$.
Proposition~\ref{stoneage} implies that $G \setminus R$ is $(\eta/2,n)$-good.
Lemma~\ref{almostpath} and the choice of $M$ above implies that $G \setminus R$ contains a collection $\mathcal{P}$ of $m \leq M$ vertex-disjoint $(\eta/2)$-heavy square paths such that \begin{equation}\label{pclass}
\sum_{P \in \mathcal{P}}|P| \geq (1-\eps/2)(1-\delta)n \geq (1-\eps)n.
\end{equation}

Write $\mathcal{P} := \lbrace P_1, \ldots, P_m \rbrace$. 
Let $P_0 = Q_0 := \emptyset$ and $P_{m+1} := P_1$.
For each $0\leq i \leq m$, we will find a square path $Q_i$ in $G[R]$ such that $P_iQ_iP_{i+1}$ is an $(\eta/2)$-heavy square path in $G$.
Suppose, for some $0 \leq i \leq m-1$, we have obtained vertex-disjoint square paths $Q_0,\ldots,Q_i$ in $G[R]$ such that, for all $0 \leq j \leq i$ we have that $P_{j}Q_{j}P_{j+1}$ is an $(\eta/2)$-heavy square path in $G$, and $|Q_j| \leq 19$.
Let $[P_{i+1}]^+_2=a'b'$ and $[P_{i+2}]^-_2=c'd'$.

Set $G' := G[R \cup \lbrace a',b',c',d' \rbrace]$ and $n' := |G'|=\delta n+4$. 
We claim that $G'$ is $(\eta/4,n')$-good.
First note that, for all $v \in V(G)_{\eta/2}$, (\ref{dRv}) implies that
$$
d_{G}(v,V(G')) \geq d_G(v,R) \geq \left(\frac{2}{3}+\frac{3\eta}{8}\right)|R| \geq \left(\frac{2}{3}+\frac{\eta}{4}\right)n'+1.
$$ 
So, since $P_{i+1}$ and $P_{i+2}$ are $(\eta/2)$-heavy square paths in $G$, we have
$\lbrace a',b',c',d' \rbrace \subseteq V(G')_{\eta/4}$.
Fix $1 \leq i \leq n'/3$ and let $X_i \subseteq V(G')$ be such that $|X_i|=i$.
We need to show that $\max_{x \in X_i}d_{G'}(x) \geq (1/3+\eta/4)n' + i +1$.
So we may assume that $\lbrace a',b',c',d' \rbrace \cap X_i = \emptyset$, i.e. $X_i \subseteq R$ (otherwise we are done).
Since $G[R]$ is $(\eta/2)$-good, we have that
$$
\max_{x \in X_i}d_{G'}(x) \geq \max_{x \in X_i}d_G(x,R) \geq \left(\frac{1}{3}+\frac{\eta}{2}\right)\delta n+i+1 \geq \left(\frac{1}{3}+\frac{\eta}{4}\right)n'+i+1,
$$
as required.
So $G'$ is $(\eta/4,n')$-good.

Let $W := \bigcup_{0 \leq j \leq i}V(Q_j)$.
Then $|W| \leq 19M \leq \delta n'$.
So we can apply Lemma~\ref{abcd} with $G',n',\delta,\eta/4$ playing the roles of $G,n,\delta,\eta$ to find in $G'$ an $(a'b',c'd')$-path $P'$ on at most $22$ vertices which avoids $W$.
We take $Q_{i+1}$ to be the square path such that $P' = a'b'Q_{i+1}c'd'$.
So $Q_{i+1} \subseteq G[R]$.
Then $Q_{i+1}$ is vertex-disjoint from $Q_0, \ldots, Q_i$; $|Q_{i+1}| \leq 19$ and  $P_{i+1}Q_{i+1}P_{i+2}$
is an $(\eta/2)$-heavy square path in $G$.

Follow this procedure until we have obtained $Q_0, \ldots, Q_m$ in $G[R]$ with the required properties.
It is easy to see that $C := P_1Q_1P_2\ldots P_mQ_m$ is a square cycle in $G$.
Finally, (\ref{pclass}) implies that $|C| \geq \sum_{P \in \mathcal{P}}|P| \geq (1-\eps)n$.
\end{proof}



\section{An almost spanning triangle cycle}\label{sectc}

In order to find the  square of a Hamilton cycle in $G$, we will first show that the reduced graph $R$ of $G$ contains an almost spanning subgraph $Z_\ell$ which itself contains a spanning square cycle, but with some specific additional edges.
We call this structure $Z_\ell$ an `$\ell$-triangle cycle'. The structure $Z_\ell$ in $R$ will act as a `framework' for embedding the square of a Hamilton cycle in $G$.
Given $c \in \mathbb{N}$, write $C^2_c$ for the square cycle on $c$ vertices. So $V(C^2 _c) = \lbrace x_1, \ldots, x_c \rbrace$, and $x_ix_j \in E(C^2 _c)$ whenever $|i-j| \in \lbrace 1,2 \rbrace$ modulo $c$. We will often write $C^2_c =: x_1 \ldots x_c$.

\begin{definition}\emph{($\ell$-triangle cycle $Z_\ell$)}\label{ltricycle}
Write $Z_\ell$ for the graph with vertex set $[\ell] \times [3]$ such that for all $1 \leq i \leq \ell$ and distinct $1 \leq j,j' \leq 3$, we have $(i,j)(i,j') \in E(Z_\ell)$, and $(i,j)(i+1,j') \in E(Z_\ell)$, where addition is modulo $\ell$. We call $Z_\ell$ an \emph{$\ell$-triangle cycle.}

Let $T_\ell$ be the spanning subgraph of $Z_\ell$ such that for all $1 \leq i \leq i' \leq \ell$ and $1\leq j < j' \leq 3$, $(i,j)(i',j') \in E(T_\ell)$ whenever $i=i'$.
So $T_\ell$ is a collection of $\ell$ vertex-disjoint triangles.
\end{definition}

So $Z_\ell$ consists of a cyclically ordered collection of $\ell$ vertex-disjoint triangles $T_\ell$, and between any pair of consecutive triangles, there is a complete bipartite graph minus a perfect matching.
We observe the following properties of $Z_\ell$:
\begin{itemize}
\item $|Z_\ell|=3\ell$ and $Z_\ell$ is $6$-regular;
\item $Z_\ell \supseteq C^2_{3\ell}$, i.e. $Z_\ell$ contains the square of a Hamilton cycle;
\item $Z_\ell$ is a $3$-partite graph (where the vertex $(i,j)$ belongs to the $j$th colour class);
\item $Z_\ell$ is invariant under permutation of the second index $j$.
\end{itemize}
This final property will be crucial when using a copy of $Z_{\ell}$ in $R$ to embed the square of a Hamilton cycle in $G$.
We explain this further in Section~\ref{sec:square}.

The following lemma states that a large $\eta$-good graph $G$ contains a copy of $Z_\ell$ which covers almost every vertex of $G$.
Its proof is a consequence of Theorem~\ref{blowup} and Lemma~\ref{almostcycle}.

\begin{lemma}\label{trianglecycle}
Let $n \in \mathbb{N}$ and $0 < 1/n \ll \eps \ll \eta \ll 1$.
Then, for every $\eta$-good graph $G$ on $n$ vertices, there exists an integer $\ell$ with $(1-\eps) n \leq 3\ell \leq n$ such that $G \supseteq Z_\ell$.
\end{lemma}

A  structure very similar to $Z_\ell$ was used in~\cite{bst} as a framework for embedding spanning subgraphs of small bandwidth and bounded maximum degree.
As such, we believe that Lemma~\ref{trianglecycle} could also be applied to embed such subgraphs into graphs satisfying the hypothesis of Theorem~\ref{mainthm} (see Section~\ref{concsec}).

\medskip
\noindent
\emph{Proof of Lemma~\ref{trianglecycle}.}
Let $M \in \mathbb{N}$ and let $d$ be a constant such that $1/n \ll 1/M \ll \eps \ll d \ll \eta$.
Apply Lemma~\ref{reg} (the Regularity lemma) to $G$ with parameters $\eps^4,d,M$ to obtain a reduced graph $R$ with $|R|=:L$ and pure graph $G'$.
So $G$ has a partition into $L$ clusters $V_1, \ldots, V_L$ each of size $m$, and an exceptional set $V_0$ of size at most $\eps^4 n$.
We may assume that $n$ is sufficiently large so that $1/n \ll 1/L \leq 1/M$.
Therefore we have the hierarchy
$$
0 < 1/n \ll 1/L \ll \eps \ll d \ll \eta \ll 1.
$$
Moreover,
\begin{equation}\label{oncemore}
L \geq \frac{(1-\eps^4)n}{m}.
\end{equation}
Lemma~\ref{reduced}(ii) implies that $R$ is $(\eta/2,L)$-good.
Lemma~\ref{almostcycle} applied with $\eta/2,\eps^4,L$ playing the roles of $\eta,\eps,n$ implies that $R$ contains a square cycle $C^2_c$ with
\begin{equation}\label{cdef}
|C^2_c| = c \geq (1-\eps^4)L.
\end{equation}
So each edge $ij \in E(C^2_c)$ corresponds to an $(\eps^4,d)$-regular pair $G'[V_{i},V_{j}]$ in $G'$.
Lemma~\ref{superslice} applied with $C^2_c,4,\eps^4,d$ playing the roles of $H,\Delta,\eps,d$ implies that each $V_i$ contains a set $V_i'$ with $|V_i'| = (1-\eps^2)m$ such that for every edge $ij$ of $C^2_c$, the graph $G'[V_i',V_j']$ is $(4\eps^2,d/2)$-superregular.
Now vertices in $R$ correspond naturally to the clusters $V_i'$.
Choose $\ell \in 3c\mathbb{N}+1$ such that
\begin{equation}\label{elldef}
\left(\frac{1}{3}-\frac{\eps}{3}\right)n < \ell < \left( \frac{1}{3} - \frac{\eps}{2} \right)n .
\end{equation}
(This is possible since $3c \leq 3L < \eps n/6$.)

Note that it suffices to find a graph homomorphism $\phi : V(Z_\ell) \rightarrow V(C^2_c)$ such that at most $(1-\eps^2)m$ vertices of $Z_\ell$ are mapped to the same vertex of $C^2_c$, i.e. that $|\phi^{-1}(w)| \leq (1-\eps^2)m$ for all $w \in V(C^2_c)$.
Then Theorem~\ref{blowup2} (the alternative Blow-up lemma) with $Z_\ell,V_i',(1-\eps^2)m,C^2_c$ playing the roles of $H,V_i,n/k,J$ implies that $G$ contains a copy of $Z_\ell$.

We will find $\phi$ in two stages.
We define graph homomorphisms $\phi_1 : V(C^2_{3c}) \rightarrow V(C^2_c)$ and $\phi_2 : V(Z_\ell) \rightarrow V(C^2_{3c})$.
Then $\phi := \phi_1 \circ \phi_2 : V(Z_\ell) \rightarrow V(C^2_c)$ is a graph homomorphism.

Write $C^2_c := w_1 \ldots w_c$ and $C^2_{3c} := x_1 \ldots x_{3c}$.
Given integers $k,N$, write $[k]_N$ for the unique integer in $[N]$ such that $k \equiv [k]_N \mod N$.
Let $\phi_1 : V(C^2_{3c}) \rightarrow V(C^2_{c})$ be defined by setting
$
\phi_1(x_i) = w_{[i]_c}
$
for all $1 \leq i \leq 3c$.
Then $\phi_1$ is a graph homomorphism, and
\begin{equation}\label{phi1}
|\phi_1^{-1}(w)| = 3\ \ \text{ for all }w \in V(C^2_c).
\end{equation} 
For each $1 \leq j \leq 3c$, relabel the vertex $x_j$ of $C^2_{3c}$ by the ordered pair $(\lceil j/3 \rceil, [j]_3)$.
(So the new vertex set is $[c] \times [3]$.)
For each $1 \leq j \leq 3c$, let $T_j$ be the triangle in $C^2_{3c}$ spanned by $x_j,x_{j+1},x_{j+2}$ (where $x_{3c+1}:=x_1$ and $x_{3c+2}:=x_2$).
So
$$
V(T_j) = \lbrace x_j,x_{j+1},x_{j+2} \rbrace = \left\lbrace \left(\left\lceil \frac{j}{3} \right\rceil, [j]_3 \right), \left(\left\lceil \frac{j+1}{3} \right\rceil, [j+1]_3 \right), \left(\left\lceil \frac{j+2}{3} \right\rceil, [j+2]_3 \right) \right\rbrace.
$$
So for any $j$, $T_j$ and $T_{j+1}$ have exactly two vertices in common.
Observe that  $\lbrace [j]_3, [j+1]_3, [j+2]_3 \rbrace = [3]$.
Let $\phi_2 : V(Z_\ell) \rightarrow V(C^2_{3c})$ be the map that takes a vertex $(i,j)$ to the unique vertex in $T_{[i]_{3c}}$ whose second index is $j$. This is illustrated in Figure~4.

To see why $\phi_2$ is a graph homomorphism, consider an edge $uv \in E(Z_\ell)$.
Let $S_i$ be the triangle in $Z_\ell$ spanned by $(i,1),(i,2),(i,3)$.
So $\phi_2$ maps each of the vertices of $S_i$ to a distinct vertex in $T_{[i]_{3c}}$.
Suppose first that there exists $1 \leq i \leq \ell$ such that $u$ and $v$ both lie in $S_i$.
Then $\phi_2$ maps both of $u$ and $v$ to different vertices of the same triangle $T_{[i]_{3c}}$ in $C^2_{3c}$.
So $\phi_2(u)\phi_2(v) \in E(C^2_{3c})$.
Suppose instead that $u$ and $v$ do not lie in the same triangle $S_i$.
Then, since $uv \in E(Z_\ell)$, $u$ and $v$ lie in consecutive triangles.
More precisely, there exist $1 \leq i \leq \ell$ and distinct $1 \leq j,j' \leq 3$ such that $u = (i,j)$ and $v = (i+1,j')$ (where $(\ell+1,j') := (1,j')$).

Suppose first that $i \leq \ell-1$.
Then by definition $\phi_2$ maps $u$ and $v$ to consecutive triangles $T_k$ and $T_{k+1}$ respectively.
It is not hard to see that every pair of the four vertices in $T_k \cup T_{k+1}$ is joined by an edge whenever their second index is different.
But the second indices of $\phi_2(u)$ and $\phi_2(v)$ are indeed different since $j \neq j'$.
So $\phi_2(u)\phi_2(v) \in E(C^2_{3c})$.

Suppose instead that $i = \ell$ (observe that we cannot have $i > \ell$).
So $u = (\ell,j)$ and $v = (1,j')$ for some distinct $1 \leq j,j' \leq 3$.
Since $\ell \in 3c\mathbb{N}+1$, we have $[\ell]_{3c} = 1 = [1]_{3c}$.
So, by the definition of $\phi_2$, $u$ is mapped to the unique vertex in $T_1$ with second index $j$ and $v$ is mapped to the unique vertex in $T_1$ with second index $j'$.
Since $j \neq j'$, we have $\phi_2(u)\phi_2(v)=(1,j)(1,j') \in E(C^2_{3c})$.

Therefore $\phi_2$, and hence $\phi$, is a graph homomorphism.
It remains to check that the preimage of each vertex of $C^2_c$ under $\phi$ is not too large. First note that
\begin{equation}\label{phi2}
\lfloor \ell/c \rfloor \leq |\phi_2^{-1}(x)| \leq \lceil \ell/c \rceil \ \ \text{ for all } x \in V(C^2_{3c}).
\end{equation}
Thus, for each $1 \leq j \leq c$ we have that
\begin{eqnarray*}
|\phi^{-1}(w_j)| &\leq& \left|\max_{x \in V(C^2_{3c})}\phi_2^{-1}(x)\right| \left|\max_{w \in V(C^2_c)}\phi_1^{-1}(w)\right| \stackrel{(\ref{phi1}),(\ref{phi2})}{\leq} 3\lceil \ell/c \rceil < \frac{3\ell}{c}+3\\
&\stackrel{(\ref{cdef}),(\ref{elldef})}{\leq}& \frac{(1-3\eps/2)n}{(1-\eps^4)L}+3 \stackrel{(\ref{oncemore})}{\leq} \frac{(1-\eps)m}{(1-\eps^4)^2} \leq (1-\eps^2)m,
\end{eqnarray*}
as required.
\hfill$\square$

\begin{center}
\begin{figure}

\includegraphics[scale=1.5]{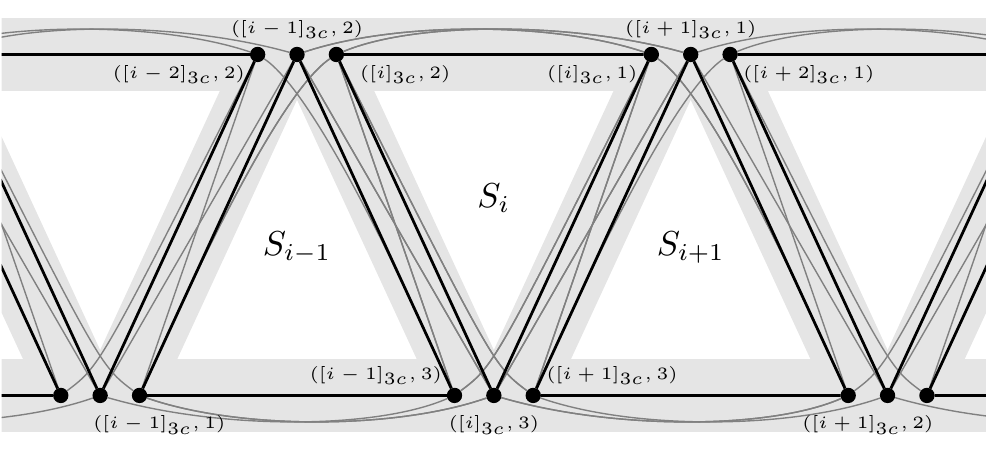}
\caption{The homomorphism $\phi_2$ maps triangles $S_i$ in $Z_\ell$ (drawn in black) to triangles in $C^2_{3c}$ (drawn in grey).}\label{homofig}
\end{figure}
\end{center}


\section{The square of a Hamilton cycle}\label{sec:square}

The final step in the proof of Theorem~\ref{mainthm} is to use the almost spanning triangle cycle guaranteed by Lemma~\ref{trianglecycle} to obtain the square of a Hamilton cycle.

Let $G$ be a large $\eta$-good graph on $n$ vertices.
Since the reduced graph $R$ almost inherits the degree sequence of $G$, we can find an almost spanning $\ell$-triangle cycle $Z_\ell$ in $R$ (whose vertices correspond to clusters and edges to ($\eps,d$)-regular pairs).
By removing a small number of vertices, we can ensure that the edges in the triangle packing $T_\ell \subseteq Z_\ell \subseteq R$ are superregular, and each of the $3\ell$ clusters has the same size.
We say that the collection of clusters now induces a \emph{cycle structure} $\mathcal{C}$ in $G$.
We colour the clusters and vertices in clusters of $\mathcal{C}$ according to the $3$-colouring of $Z_\ell$ (which is unique up to isomorphism), so that both $V_{i,j}$ and $x \in V_{i,j}$ have colour $j$.
It is now a fairly simple consequence of the Blow-up lemma that $G$ contains a square cycle whose vertex set contains precisely the vertices in the clusters of 
 $\mathcal{C}$.
In fact, this would still be true as long as the clusters in each triangle in $T_\ell$ each had the same size (in which case we say that $\mathcal{C}$ is $0$-\emph{balanced}).

However, there is a small set $V_0$ of vertices in $G$ which lie outside any of the clusters of $\mathcal{C}$.
We need to incorporate these into the clusters in an appropriate way, and also preserve the structure $\mathcal{C}$ (perhaps with slightly worse parameters).
So after any changes to the clusters we require that
\begin{itemize}
\item[(i)] regular pairs remain regular;
\item[(ii)] superregular pairs remain superregular;
\item[(iii)] $\mathcal{C}$ is $0$-balanced.
\end{itemize}
(i) is satisfied as long as no cluster gains or loses too many vertices.
For (ii), we need to ensure that, if we insert a vertex $v$ into a cluster $V$, then $v$ has many neighbours in the two other clusters which lie in the same triangle as $V$ in $T_{\ell}$.
(In this case we say that $v \rightarrow V$ is \emph{valid}.)
It turns out that since $\delta(G) \geq (1/3 + \eta)n$, for each vertex $v$ there are at least $\eta|R|$ clusters $V$ such that $v \rightarrow V$ is valid.
This appears promising, but recall that a necessary condition for (iii) is that the colour classes in $\mathcal{C}$ are all the same size.
However, we may not be able to assign the vertices of $v \in V_0$ so that this is even almost true. 
For example, every $v \in V_0$ might only have valid clusters in the first colour class.

Given any $v \in V(G)$, we can guarantee more valid clusters $V$ if $d_G(v)$ is larger.
In fact, if $v \in V(G)_\eta$, there are $\eta|R|$ triangles $T \in T_\ell$ such that $v \rightarrow V$ is valid for every $V \in T$ (see Proposition~\ref{replacelarge}). 
So, if it were true that $V_0 \subseteq V(G)_\eta$, then we could assign each $v \in V_0$ to a triangle in $T_\ell$ so that no triangle receives too many vertices, and then split the vertices in each triangle among its clusters as equally as possible.
Then $\mathcal{C}$ is very close to being $0$-balanced (the sizes of clusters in a triangle in $T_{\ell}$ differ by at most one).

In order to achieve that $V_0 \subseteq V(G)_\eta$ (see Lemma~\ref{swap}), we do the following.
Whenever there is $v \in V_0 \setminus V(G)_\eta$, we find many clusters $V$ such that $v \rightarrow V$ is valid, and $V$ contains many vertices $v'$ with $d_G(v') \geq d_G(v) + \eta n/4$.
Then we swap $v$ and $v'$ without destroying the cycle structure.
This process is repeated until no longer possible, in such a way that no cluster $X$ is the location of too many swaps.
 
Now we have achieved (i) and (ii), and $\mathcal{C}$ is almost $0$-balanced.
Note that a necessary condition for (iii) is that $3 | n$, so assume that this is true.
At this stage we appeal to those pairs in $\mathcal{C}$ which correspond to edges in $Z_\ell$ (not just those in $T_\ell$).
This is also the stage where having $Z_\ell \subseteq R$ (and not only $C^2_{3\ell} \subseteq R$) is useful.
Consider a cluster $V_{i,j}$.
Then the fact that $Z_\ell \subseteq R$ ensures that almost every vertex $v \in V_{i,j}$ is such that $v \rightarrow V_{i-1,j}$ and $v \rightarrow V_{i+1,j}$ are both valid.
Applying this repeatedly allows us to make a small number of arbitrary reallocations within a colour class $j$ (see Lemma~\ref{shuffle}).

However, unless the colour classes have equal size (that is, size $n/3$), this procedure can never ensure that $\mathcal{C}$ is $0$-balanced.
We currently have that the colour classes have close to equal size.
Suppose, for example, that colour class $3$ is larger than colour class $1$, and colour class $2$ has exactly the right size.
We identify a `feeder cluster' $X_3$ in $\mathcal{C}$, whose vertices are all coloured $3$, and which has large core degree.
Then $X_3$ contains many vertices of degree at least $(2/3+\eta)n$.
For each of these vertices $v$, there are many colour $1$ clusters $V$ such that $v \rightarrow V$ is valid.
So we can move a small number of these vertices $v$ to colour $1$ clusters so that all the colour classes have the same size (see Lemma~\ref{colourclass}).

\subsection{Cycle structures}\label{subsec:square}

We  begin by formally defining a cycle structure.

\begin{definition}\emph{(Cycle structure)}\label{cycstruc}
Given an $\ell \times 3$ integer matrix $M$, integers $n,\ell$, a graph $G$ on $n$ vertices, and constants $\eps,d$,
we say that $G$ has an \emph{$(R,\ell,M,\eps,d)$-cycle structure $\mathcal{C}$} if the following hold:
\begin{itemize}
\item[(C1)] $G$ has vertex partition $\lbrace V_0 \rbrace \cup \lbrace V_{i,j} : (i,j) \in [\ell]\times [3] \rbrace$ where the $(i,j)$th entry of $M$ is $|V_{i,j}|$ and $|V_0| \leq \eps n$. The sets $V_{i,j}$ are called the \emph{clusters} of $\mathcal{C}$, $V_0$ is called the \emph{exceptional set} of $\mathcal{C}$, and $M$ is called the \emph{size matrix} of $\mathcal{C}$;
\item[(C2)] $R$ has vertex set $[\ell] \times [3]$ and $R \supseteq Z_\ell$ and $G[V_{i,j},V_{i',j'}]$ is $(\eps,d)$-regular whenever $(i,j)(i',j') \in E(R)$;
\item[(C3)] $G[V_{i,j},V_{i,j'}]$ is $(\eps,d)$-superregular whenever $1 \leq i \leq \ell$ and $1 \leq j<j' \leq 3$.
\end{itemize}
We say that $\lbrace V_{i,j} : (i,j) \in [\ell]\times [3] \rbrace$ \emph{induces} $\mathcal{C}$.
If $V_0 = \emptyset$ we say that $\mathcal{C}$ is \emph{spanning}.
\end{definition}

Let $\mathcal{C}'$ be the cycle structure obtained from $\mathcal{C}$ by relabelling $V_{i,j}$ by $V_{i,\sigma(j)}$ for all $(i,j) \in [\ell] \times [3]$ and some permutation $\sigma$ of $[3]$.
Since $Z_\ell$ is invariant under permutation of the second index (as observed immediately after Definition~\ref{ltricycle}), $\mathcal{C}'$ is an $(R,\ell,M',\eps,d)$-cycle structure where the $(i,j)$th entry of $M'$ is $|V_{i,\sigma(j)}|$.

Often we will consider two different cycle structures, say an $(R, \ell, M, \eps, d)$-cycle structure $\mathcal C$ and an $(R, \ell, M', \eps', d')$-cycle structure $\mathcal C'$. Since 
the vertex set of $R$ corresponds to both the clusters of $\mathcal C$ and $\mathcal C'$, it is ambiguous in this case to talk about the core degree $d^{\alpha} _{R,G}$. Indeed, even though the graph $R$ is the same
for both cycle structures $\mathcal C$ and $\mathcal C'$, the clusters of   $\mathcal C$ and $\mathcal C'$ may be different. We therefore say that $d^{\alpha} _{R,G}$ is \emph{$(\eta, 3\ell)$-good with respect to $\mathcal C$} to mean that $d^{\alpha} _{R,G}$ is $(\eta, 3\ell)$-good when considering the vertices of $R$ as corresponding to the clusters of $\mathcal C$.

\begin{definition}\emph{(Size matrices)}
Given an $n_1 \times n_2$ integer matrix $M$, we write $M = (m_{i,j})$ if the $(i,j)$th entry of $M$ is $m_{i,j}$ for all $(i,j) \in [n_1] \times [n_2]$.

Given integers $k_1 \leq k_2$, we say that $M$ is $(k_1,k_2)$-bounded if $k_1 \leq m_{i,j} \leq k_2$ for all $(i,j) \in [n_1] \times [n_2]$.

For a non-negative integer $k$, whenever $|m_{i,j}-m_{i,j'}| \leq k$ for all $1 \leq i \leq n_1$ and $1 \leq j < j' \leq n_2$, we say that $M$ is $k$-\emph{balanced}.

If $\sum_{1 \leq i \leq n_1}m_{i,j} = \sum_{1 \leq i \leq n_1}m_{i,j'}$ for all $1 \leq j,j' \leq n_2$, we say that $M$ has \emph{equal columns}.
\end{definition}

So if $\mathcal{C}$ is an $(R,\ell,M,\eps,d)$-cycle stucture in which $M$ is $(k_1,k_2)$-bounded, then
\begin{equation}\label{3mln}
(1-\eps)n \leq 3\ell k_2 \ \ \text{ and } \ \ 3\ell k_1 \leq n.
\end{equation}
Observe that, if $\mathcal{C}$ is spanning and $M$ has equal columns, then $3 | n$.
The columns of $M$ correspond to the colour classes of $Z_\ell$.

The purpose of this section is to prove the following lemma, which states that any large $\eta$-good graph contains a spanning $0$-balanced cycle structure.

\begin{lemma}\label{0balanced}
Let $n \in 3\mathbb{N}$, $L_0,L' \in \mathbb{N}$, and let $0 < 1/n \ll 1/L_0 \ll 1/L' \ll \eps   \ll d \ll \eta \ll 1$.
Suppose that $G$ is an $\eta$-good graph on $n$ vertices.
Then there exists a spanning subgraph $G' \subseteq G$ and $\ell \in \mathbb{N}$ with $L' \leq \ell \leq L_0$ such that
 $G'$ has a spanning $(R,\ell,M,\eps,d)$-cycle structure where $M$ is $((1-\eps)m,(1+\eps)m)$-bounded and $0$-balanced.
\end{lemma}

The next proposition will be used several times to show that cycle structures are robust in the following sense.
If a small number of vertices  in a cycle structure are reallocated, so that each of them has many neighbours in appropriate clusters, we still have a cycle structure (with slightly worse parameters).
Its proof is a consequence of Proposition~\ref{newsuperslice}.

\begin{proposition}\label{dull}
Let $n,\ell,m,r \in \mathbb{N}$ and $0 < 1/n \ll 1/\ell \ll \eps \leq \gamma \ll d < 1$.
Suppose that $G$ is a graph on $n$ vertices with an $(R,\ell,M,\eps,d)$-cycle structure, where $M = (m_{i,j})$ is $(m,(1+\eps)m)$-bounded.
Let $\lbrace V_{i,j} : (i,j) \in [\ell] \times [3] \rbrace$ be the set of clusters of $\mathcal{C}$, where $m_{i,j} := |V_{i,j}|$.
Suppose that there exists a collection $\mathcal{X} := \lbrace X_{i,j} : (i,j) \in [\ell] \times [3] \rbrace$ of vertex-disjoint subsets of $V(G)$ such that for all $(i,j) \in [\ell] \times [3]$,
\begin{itemize}
\item $|X_{i,j} \triangle V_{i,j}| \leq \gamma m/2$;
\item for all $x \in X_{i,j} \setminus V_{i,j}$ we have that $d_G(x,V_{i,j'}) \geq (d-\eps)m$ for all $j' \in [3] \setminus \lbrace j \rbrace$.
\end{itemize}
Let $N := (n_{i,j})$ where $n_{i,j} := |X_{i,j}|$.
Then, for any $\eps' \geq \eps +6 \sqrt{\gamma}$, we have that $\mathcal{X}$ induces an $(R,\ell,N,\eps',d/2)$-cycle structure $\mathcal{C}'$.
\end{proposition}

\begin{proof}
It is clear that, for all $(i,j) \in [\ell] \times [3]$,
\begin{equation}\label{dull1}
(1-\gamma)m \leq |X_{i,j}| \leq (1+2\gamma)m .
\end{equation}
We need to check that $\mathcal{C}'$ satisfies (C1)--(C3).
For (C1), it suffices to check that the exceptional set $X_0 := V(G) \setminus \bigcup_{X \in \mathcal{X}}X$ of $\mathcal{C}'$ is such that $|X_0| \leq \eps'n$.
Let $V_0$ be the exceptional set of $\mathcal{C}$.
Then $|X_0| \leq |V_0| + \sum_{(i,j) \in [\ell] \times [3]}|V_{i,j} \triangle X_{i,j}| \leq \eps n + 3\ell \gamma m \leq \eps 'n$ by (\ref{3mln}).
So (C1) holds.

Note that, since $M$ is $(m,(1+\eps)m)$-bounded, $|X_{i,j} \triangle V_{i,j}| \leq \gamma|X_{i,j}|$.
For (C2), let $(i,j)(i',j') \in E(R)$.
Then Proposition~\ref{newsuperslice} implies that $G[X_{i,j},X_{i',j'}]$ is $(\eps',d/2)$-regular, as required.

For (C3), let $1 \leq i \leq \ell$ and $1 \leq j<j' \leq 3$.
Then, since $(i,j)(i,j') \in E(R)$, Proposition~\ref{newsuperslice} implies that it suffices to show that, for all $x \in X_{i,j}$, we have $d_G(x,X_{i,j'}) \geq d|X_{i,j'}|/2$, and for all $y \in X_{i,j'}$, we have $d_G(y,X_{i,j}) \geq d|X_{i,j}|/2$.
Let $x \in X_{i,j}$.
Suppose first that $x \in V_{i,j}$.
Then, since $G[V_{i,j},V_{i,j'}]$ is $(\eps,d)$-superregular by (C3) for $\mathcal{C}$, we have that $d_G(x,V_{i,j'}) \geq d|V_{i,j'}| \geq dm$.
Suppose instead that $x \in X_{i,j} \setminus V_{i,j}$.
Then, by hypothesis, $d_G(x,V_{i,j'}) \geq (d-\eps)m$.
So for all $x \in X_{i,j}$ we have $d_G(x,V_{i,j'}) \geq (d-\eps)m$.
Therefore
$$
d_G(x,X_{i,j'}) \geq d_G(x,V_{i,j'}) - |X_{i,j'} \triangle V_{i,j'}| \geq (d-\eps)m - \gamma m  \stackrel{(\ref{dull1})}{\geq} \frac{d|X_{i,j'}|}{2},  
$$
as required.
The second assertion follows similarly.
This proves (C3).
\end{proof}

Our initial goal is to incorporate each vertex in the exceptional set into a suitable cluster.
However, we are only able to do this successfully for vertices with large degree.
The following proposition will be used to swap an exceptional vertex with a vertex in a cluster that has larger degree.
The cycle structure which remains has the same size matrix $M$ and the exceptional set has the same size.
The proposition will be applied repeatedly until every exceptional vertex has degree at least $(2/3+\eta)n$ (see Lemma~\ref{swap}).

\begin{proposition}\label{replacesmall}
Let $n,\ell,m \in \mathbb{N}$ and $0 < 1/n \ll 1/\ell \ll \eps \ll c \ll d \ll \eta < 1 \leq \alpha \leq 1/3\eta+3/4$.
Let $G$ be an $\eta$-good graph on $n$ vertices with an $(R,\ell,M,\eps,d)$-cycle structure where $M$ is $(m,m)$-bounded.
Let $\lbrace V_{i,j} : (i,j) \in [\ell] \times [3] \rbrace$ be the set of clusters of $\mathcal{C}$.
Suppose further that $d^c_{R,G}$ is $(\eta/2,3\ell)$-good with respect to $\mathcal C$.
Let $v \in V(G)$ with $d_G(v)\geq (1/3+\alpha\eta)n$.
Then there exists $I \subseteq V(R)$ with $|I| \geq \eta\ell/10$ such that, for all $(i,j) \in I$, the following hold:
\begin{itemize}
\item[(i)] for all $j' \neq j$, we have $d_{G}(v,V_{i,j'}) \geq (d-\eps)m$;
\item[(ii)] there are at least $cm$ vertices $x$ in $V_{i,j}$ such that $d_G(x) \geq (1/3+(\alpha+1/4)\eta)n$.
\end{itemize}
\end{proposition}

\begin{proof}
We begin by proving the following claim.

\medskip
\noindent
\textbf{Claim.}
\emph{Let $I' := \lbrace (i,j) \in V(R) : d_{G}(v,V_{i,j'}) \geq (d-\eps)m\ \text{ for all }\ j' \neq j \rbrace$. Then $|I'| \geq (3\alpha-1/10)\eta\ell$.}

\medskip
\noindent
To prove the claim, define
$\overline{d}_{G}(v) := n-d_{G}(v)$. 
For integers $0 \leq p \leq 3$, let 
$$
K_p := \lbrace 1 \leq i \leq \ell : d_{G}(v,V_{i,j}) \geq (d-\eps)m\ \text{ for exactly $p$ values }\ j \in [3] \rbrace
$$
and $k_p := |K_p|$.
Observe that $K_p \cap K_{p'} = \emptyset$ whenever $p \neq p'$.
So
\begin{equation}\label{ksum}
k_0+k_1+k_2+k_3=\ell.
\end{equation}
For each $i \in K_2$ there is exactly one $1 \leq j \leq 3$ such that $(i,j) \in I'$, and for each $i \in K_3$ we have $(i,j) \in I'$ for all $1 \leq j \leq 3$.
Therefore it suffices to show that $k_2+3k_3 \geq (3\alpha-1/10)\eta\ell$.
We have that
\begin{eqnarray*}
\overline{d}_G(v) &\geq& \sum_{0 \leq p \leq 3}\sum_{i \in K_p} \sum_{1 \leq j \leq 3} (m - d_{G}(v,V_{i,j})) \geq \sum_{0 \leq p \leq 3}\sum_{i \in K_p}  (3-p)(1-d-\eps)m\\
&=& (3k_0+2k_1+k_2)(1-d-\eps)m \geq (3k_0+2k_1+k_2)\left(1-2d\right)m\\
&\stackrel{(\ref{ksum})}{\geq}& (3\ell-(k_1+k_2)-(k_2+3k_3))\left(1-2d\right)m \stackrel{(\ref{ksum})}{\geq} (2\ell-(k_2+3k_3))\left(1-2d\right)m.
\end{eqnarray*}
Suppose that $k_2+3k_3 < (3\alpha-1/10)\eta\ell$.
Then
\begin{align*}
\overline{d}_G(v) &\geq (1-2d)\left(2-3\alpha\eta+\frac{\eta}{10}\right)m\ell \geq \left(2-3\alpha\eta+\frac{\eta}{11}\right)m\ell \stackrel{(\ref{3mln})}{\geq} (1-\eps)\left(\frac{2}{3}-\alpha\eta + \frac{\eta}{33} \right)n\\
&> (2/3-\alpha\eta)n,
\end{align*}
a contradiction.
This completes the proof of the claim.

\medskip
\noindent
Recall that $d^c_{R,G}$ is $(\eta/2,3\ell)$-good.
Proposition~\ref{vertexdeg}(ii) with $R,d^c_{R,G},I',\eta\ell/10$ playing the roles of $G,d'_G,X,k$ implies that there exists $I \subseteq I'$ with $|I| \geq \eta\ell/10$ such that for every $(i,j) \in I$, we have
$$
d^c_{R,G}((i,j)) \geq 3\left(\frac{1}{3}+\frac{\eta}{2}\right)\ell + 3\alpha\eta\ell - \frac{\eta\ell}{5}+2 \geq 3\left(\frac{1}{3}+\left(\alpha+\frac{1}{4}\right)\eta\right)\ell.
$$
The claim together with the fact that $I \subseteq I'$ imply that $I$ satisfies (i).
By the definition of core degree, for all $(i,j) \in I$, there are at least $c|V_{i,j}| = cm$ vertices $x \in V_{i,j}$ such that
$$
d_G(x) \geq \frac{d^c_{R,G}((i,j))n}{3\ell} \geq \left(\frac{1}{3}+\left(\alpha+\frac{1}{4}\right)\eta\right)n,
$$
so $I$ also satisfies (ii).
\end{proof}

The previous proposition will be used to modify our cycle structure slightly so that every exceptional vertex has large degree.
The next proposition will be used for incorporating these large degree exceptional vertices into the cycle structure $\mathcal{C}$.
It shows that, for each such vertex $v$, there are many triangles $T \in T_\ell$ such that $v$ can be added to any of the three clusters in $T$.

\begin{proposition}\label{replacelarge}
Let $n,\ell,m \in \mathbb{N}$  and $0<1/n \ll 1/\ell \ll \eps \ll d \ll \eta < 1$.
Suppose that $G$ is a graph on $n$ vertices with an $(R,\ell,M,\eps,d)$-cycle structure $\mathcal{C}$, where $M$ is $(m,(1+\eps)m)$-bounded.
Let $\lbrace V_{i,j} : (i,j) \in [\ell] \times [3] \rbrace$ be the set of clusters of $\mathcal{C}$.
Let $v \in V(G)$ with $d_G(v) \geq (2/3+\eta/2)n$.
Then there exists $I \subseteq [\ell]$ with $|I| \geq \eta\ell$ such that, for all $i \in I$ and all $j \in [3]$ we have $d_{G}(v,V_{i,j}) \geq (d-\eps)m$.
\end{proposition}

\begin{proof}
Let
$$
K := \lbrace 1 \leq i \leq \ell : \text{ there exists }j \in [3] \text{ such that }d_{G}(v,V_{i,j}) < (d-\eps)m \rbrace.
$$
It suffices to show that $|K| < (1-\eta)\ell$.
For all $1 \leq i \leq \ell$, let $U_i := \bigcup_{1 \leq j \leq 3}V_{i,j}$.
Then
\begin{align*}
d_G(v) &= \sum_{i \in K}d_{G}(v,U_i) + \sum_{i \notin K}d_{G}(v,U_i) + d_{G}(v,V_0) \leq |K|(2+\eps+d)m + 3(\ell-|K|)(1+\eps)m + \eps n\\
&= 3\ell m - (1-d+2\eps)|K|m + 3\eps \ell m + \eps n \stackrel{(\ref{3mln})}{\leq} 3\ell m - (1-\eta/3)|K|m + 2\eps n.
\end{align*}
Suppose, for a contradiction, that $|K| \geq (1-\eta)\ell$.
Then
\begin{align*}
d_G(v) &\leq \left(3-\left(1-\frac{\eta}{3}\right)(1-\eta)\right)\ell m + 2\eps n = 3\left( \frac{2}{3} + \frac{4\eta}{9}\left(1 - \frac{\eta}{4} \right) \right) \ell m + 2\eps n \stackrel{(\ref{3mln})}{<} \left(\frac{2}{3}+\frac{\eta}{2}\right)n,
\end{align*}
a contradiction.

\end{proof}

The following lemma is used to turn a cycle structure $\mathcal{C}$ which has a constant size matrix and non-empty exceptional set into a spanning $1$-balanced cycle structure $\mathcal{C}'$.
To prove it, we repeatedly apply Proposition~\ref{replacesmall} to swap vertices in and out of the exceptional set until every exceptional vertex has large degree.
We then apply Proposition~\ref{replacelarge} to allocate each of these vertices $v$ to a suitable triangle in $T_{\ell}$, such that $v$ can be placed in any of the three clusters in this triangle.
For each triangle, the allocated vertices are then split equally among the clusters so that they have size as equal as possible.

\begin{lemma}\label{swap}
Let $n,\ell,m \in \mathbb{N}$  and $0<1/n \ll 1/\ell \ll \eps \ll c \ll d \ll \eta < 1$.
Suppose that $G$ is an $\eta$-good graph on $n$ vertices with an $(R,\ell,M,\eps,d)$-cycle structure $\mathcal{C}$, where $M$ is $(m,m)$-bounded.
Suppose further that $d^c_{R,G}$ is $(\eta/2,3\ell)$-good with respect to $\mathcal C$.
Then $G$ has a spanning $(R,\ell,N,\eps^{1/3},d/2)$-cycle structure $\mathcal{C}'$, where $N$ is $(m, (1+\sqrt{\eps})m)$-bounded and $1$-balanced. Further,  $d^{c/2} _{R,G}$ is $(\eta/2,3\ell)$-good with respect to $\mathcal C'$.
\end{lemma}

\begin{proof}
Write $V_{i,j}$ for the cluster corresponding to $(i,j) \in V(R)$.
Given a vertex $v \in V(G)$ and $(i,j) \in [\ell] \times [3]$, we say that $v \rightarrow V_{i,j}$ is \emph{valid} if $d_{G}(v,V_{i,j'}) \geq (d-\eps)m$ for all $j' \in [3] \setminus \lbrace j \rbrace$.
As an initial step, we will prove 
the following claim.

\medskip
\noindent
\textbf{Claim.}
\emph{There exist subsets $X_0, X_{i,j}$ of $V(G)$ (for $(i,j) \in [\ell] \times [3]$) so that the following hold: 
\begin{itemize}
\item[(i)] $\lbrace X_0 \rbrace \cup \lbrace X_{i,j} : (i,j) \in [\ell]\times[3] \rbrace$ is a partition of $V(G)$;
\item[(ii)] $|X_0|=|V_0|$ and $|X_{i,j}|=m$ for all $(i,j) \in [\ell] \times [3]$;
\item[(iii)] $|V_{i,j} \triangle X_{i,j}| \leq 81\eps m/\eta^2$;
\item[(iv)] for all $v \in X_{i,j} \setminus V_{i,j}$ we have that $v \rightarrow V_{i,j}$ is valid;
\item[(v)] $X_0 \subseteq V(G)_{\eta/2}$.
\end{itemize}
}

\medskip
\noindent
To prove the claim, let
$K := 4\eps n/3\eta$.
Suppose that, for some $0 \leq k \leq K$, we have obtained vertex sets $V_0^k, V_{i,j}^k$ for $(i,j)\in [\ell]\times[3]$ such that the following properties hold:
\begin{itemize}
\item[($\alpha_k$)] $\lbrace V_0^k \rbrace \cup \lbrace V_{i,j}^k : (i,j) \in [\ell]\times[3] \rbrace$ is a partition of $V(G)$;
\item[($\beta_k$)] $|V_0^k|=|V_0|$ and $|V_{i,j}^k|=m$ for all $(i,j) \in [\ell] \times [3]$;
\item[($\gamma_k$)] $|V_{i,j} \triangle V_{i,j}^k| \leq 81\eps m/\eta^2$ and $\sum_{(i,j) \in V(R)}|V_{i,j} \triangle V_{i,j}^k| \leq 2k$;
\item[($\delta_k$)] for all $v \in V_{i,j}^k \setminus V_{i,j}$ we have that $v \rightarrow V_{i,j}$ is valid;
\item[($\eps_k$)] $S_k := \sum_{v \in V_0^k}d'_G(v)/|V_0| \geq (1/3+\eta)n + k\eta/4\eps$,
where $d'_G(v) := \min \lbrace d_G(v), (2/3+\eta)n \rbrace$.
\end{itemize}

\noindent
Observe that setting $V_0^0 := V_0$ and $V_{i,j}^0 := V_{i,j}$ for all $(i,j) \in [\ell] \times [3]$ satisfies ($\alpha_0$)--($\eps_0$).
Indeed,
properties ($\alpha_0$)--($\delta_0$) are clear; ($\eps_0$) follows from the fact that $G$ is $\eta$-good and therefore $\delta(G) \geq (1/3+\eta)n$.
So $S_0 \geq (1/3+\eta)n$.

We will show that there is some $k \leq K$ for which we can set $X_0 := V_0^k$ and $X_{i,j} := V_{i,j}^k$ for all $(i,j) \in [\ell] \times [3]$.
Observe that ($\alpha_k$)--($\eps_k$) imply that we can do this as long as $V_0^k \subseteq V(G)_{\eta/2}$.

So suppose that $V_0^k \not\subseteq V(G)_{\eta/2}$.
In particular, $V_0^k \neq \emptyset$.
Let $v_0 \in V_0^k$ be such that $\min_{v \in V_0^k} \lbrace d_G(v) \rbrace = d_G(v_0)$.
Then $d_G(v_0) < (2/3+\eta/2)n$, so there is some $1 \leq \alpha < 1/3\eta+1/2$ such that $d_G(v_0)=(1/3+\alpha\eta)n$.
Proposition~\ref{replacesmall} implies that there exists $I \subseteq V(R)$ with $|I| \geq \eta\ell/10$ such that, for all $(i,j) \in I$, the following hold:
\begin{itemize}
\item $v_0 \rightarrow V_{i,j}$ is valid;
\item there are at least $cm$ vertices $x$ in $V_{i,j}$ such that $d_G(x) \geq (1/3+(\alpha+1/4)\eta)n$.
\end{itemize}
We claim that $\min_{(i,j) \in I}\lbrace |V_{i,j} \triangle V_{i,j}^k| \rbrace \leq 81\eps m/\eta^2-2$.
Suppose not.
Then
$$
\sum_{(i,j) \in I}|V_{i,j} \triangle V_{i,j}^k| \geq |I|\frac{81\eps  m}{\eta^2} - 6\ell \stackrel{(\ref{3mln})}{\geq} (1-\eps)\frac{81\eps n}{30\eta} - 6\ell > \frac{8\eps n}{3\eta} = 2K \geq 2k,
$$
a contradiction to ($\gamma_k$).
Therefore we can choose $(i',j') \in I$ with
\begin{equation}\label{tkchoice}
|V_{i',j'} \triangle V_{i',j'}^k| \leq 81\eps m/\eta^2-2.
\end{equation}
Let $U$ be the collection of vertices in $V_{i',j'}$ with degree at least $(1/3+(\alpha+1/4)\eta)n$ in $G$.
Then
$$
|U \cap V_{i',j'}^k| \geq |U|-|V_{i',j'}^k \triangle V_{i',j'}| \stackrel{(\ref{tkchoice})}{\geq} \left(c-\frac{81\eps}{\eta^2}\right)m+2 \geq \frac{cm}{2} > 0, 
$$
so we can choose $v_1 \in U \cap V_{i',j'}^k$.
For each $(i,j) \in [\ell]\times [3]$, set
\begin{equation}\label{tdef}
V^{k+1}_{i,j} := \left\{
  \begin{array}{l l}
    V^k_{i,j} \cup \lbrace v_0 \rbrace \setminus \lbrace v_1 \rbrace & \quad \text{if $(i,j)=(i',j')$}\\
    V^k_{i,j} & \quad \text{otherwise};
  \end{array} \right.
\end{equation}
and
\begin{equation}\label{V0def}
V_0^{k+1} := V_0 \cup \lbrace v_1 \rbrace \setminus \lbrace v_0 \rbrace.
\end{equation}
 We need to check that $(\alpha_{k+1}$)--($\eps_{k+1}$) hold.
First note that $(\alpha_{k+1}$) and ($\beta_{k+1}$) follow immediately from ($\alpha_k$) and ($\beta_k$) respectively.
Property ($\gamma_{k+1}$) follows easily from ($\gamma_k$), (\ref{tkchoice}) and (\ref{tdef}).
To see ($\delta_{k+1}$), (\ref{tdef}) implies that it suffices to show that $v_0 \rightarrow V_{i',j'}$ is valid.
But this follows since $(i',j') \in I$.

It remains to prove that ($\eps_{k+1}$) holds. 
Recall that the choice of $v_0$ implies that $d_G(v_0) = (1/3+\alpha\eta)n < (2/3+\eta/2)n$.
In particular, $d'_G(v_0)=d_G(v_0)$.
Suppose first that $d'_G(v_1) = (2/3+\eta)n$.
Then $d'_G(v_1)-d'_G(v_0) > \eta n/2$.
Suppose instead that $d'_G(v_1) = d_G(v_1)$.
Then
$$
d'_G(v_1)-d'_G(v_0) \geq \left(\frac{1}{3}+ \left(\alpha+\frac{1}{4}\right)\eta\right)n - \left(\frac{1}{3}+\alpha\eta\right)n = \frac{\eta n}{4}.
$$
So this latter bound holds in both cases.
Therefore
$$
S_{k+1} \stackrel{(\ref{V0def})}{=} \sum_{v \in V_0^k}{\left( \frac{d'_G(v)}{|V_0|} \right)}  + \frac{d'_G(v_1) - d'_G(v_0)}{|V_0|}
 \geq S_k + \frac{\eta n}{4|V_0|} \geq S_k + \frac{\eta}{4\eps} \stackrel{\text{($\eps_k$)}}{\geq} \left(\frac{1}{3}+\eta\right)n + (k+1)\frac{\eta}{4\eps},
$$
as required.

So, for each $0 \leq k < K$, either the procedure has terminated, or we are able to proceed to step $k+1$.
Observe that, for all $k$, we have $S_k \leq (2/3+\eta)n$.
Moreover, $S_k = (2/3+ \eta)n$ if and only if $V_0 \subseteq V(G)_\eta \subseteq V(G)_{\eta/2}$.
Suppose that this iteration does not terminate in at most $K$ steps.
Then ($\eps_K$) implies that
$$
S_K \geq \left(\frac{1}{3} + \eta \right)n + K\eta/4\eps = \left( \frac{2}{3} + \eta \right) n,
$$
as required.
So the iteration terminates at some $p \leq 4\eps n/3\eta$.
Let $X_0 := V_0^p$ and $X_{i,j} := V_{i,j}^p$ for all $(i,j) \in [\ell] \times [3]$.
This completes the proof of the claim.

\medskip
\noindent
Now we will use the claim to prove the lemma.
For each $x \in X_0$, let
\begin{equation}\label{Sx}
S_x := \lbrace 1 \leq i \leq \ell : x \rightarrow V_{i,j}\ \text{ is valid for all }\ 1 \leq j \leq 3 \rbrace.
\end{equation}
Property (v) of the claim together with Proposition~\ref{replacelarge} imply that $|S_x| \geq \eta \ell$.
Therefore, for each $x \in X_0$ we can choose $i_x \in S_x$, such that for each $i \in [\ell]$, there are at most $|X_0|/\eta \ell$ vertices $x \in X_0$ such that $i = i_x$.
For the collection of $x \in X_0$ with $i_x=i$, choose $j_x$ as evenly as possible from $[3]$.
More precisely, for each $x \in X_0$, choose $j_x \in [3]$ so that
\begin{equation}\label{def1*}
||\lbrace x \in X_0 : i=i_x, j=j_x \rbrace| - |\lbrace x \in X_0 : i=i_x, j'=j_x \rbrace|| \leq 1 \ \ \text{ for }\ \ 1 \leq j,j' \leq 3.
\end{equation}
Define a partition $\lbrace U_{i,j} : (i,j) \in [\ell] \times [3] \rbrace$ of $V(G)$ by setting
\begin{equation}\label{def2*}
U_{i,j} := X_{i,j} \cup \lbrace x \in X_0: (i_x,j_x)=(i,j) \rbrace.
\end{equation}
Then for all $(i,j) \in [\ell] \times [3]$, part (ii) of the claim implies that
\begin{equation}\label{swap1}
0 \leq |U_{i,j}|-m \leq \frac{|X_0|}{\eta \ell} \stackrel{(\ref{3mln})}{\leq} \frac{3\eps m}{(1-\eps)\eta} \leq \frac{4\eps m}{\eta} \leq \sqrt{\eps}m.
\end{equation}
Therefore the $[\ell] \times [3]$ matrix $N = (n_{i,j})$ with $n_{i,j} := |U_{i,j}|$ is $(m,(1+\sqrt{\eps})m)$-bounded.
Moreover,
\begin{align*}
 |U_{i,j} \triangle V_{i,j}| &\leq |U_{i,j} \triangle X_{i,j}| + |X_{i,j} \triangle V_{i,j}| \stackrel{\text{(iii)}}{\leq} \frac{81\eps m}{\eta^2} + \frac{|V_0|}{\eta\ell} \stackrel{(\ref{swap1})}{\leq} \left( \frac{81\eps}{\eta^2} + \frac{4\eps}{\eta} \right)m \leq \frac{82\eps}{\eta^2}m
\stackrel{(\ref{swap1})}{\leq} \frac{82\eps}{\eta^2}|U_{i,j}|.
\end{align*}
Observe that (\ref{def1*}), (\ref{def2*}) and part (ii) of the claim imply that $N$ is $1$-balanced.
Then Proposition~\ref{dull} (with ${164\eps}/{\eta^2}$ playing the role of $\gamma$) implies that $G$ contains an
$(R,\ell,N,\eps^{1/3},d/2)$-cycle structure $\mathcal{C}'$, which is spanning.

Now the vertices in $R$ correspond to the clusters $U_{i,j}$. Since $|U_{i,j} \triangle V_{i,j}|\leq 82\eps m/\eta ^2$ and $\eps \ll c, \eta$, $d^{c/2} _{R,G}$ is $(\eta/2,3\ell)$-good with respect to $\mathcal C'$.
\end{proof}

The following easy fact is a consequence of the triangle inequality.

\begin{fact}\label{easy}
Let $a_1, \ldots, a_n \in \mathbb{R}$.
Then for all $1 \leq i \leq n$,
$$
\left|a_i-\frac{1}{n}\sum_{1 \leq j \leq n}a_j\right| \leq \frac{n-1}{n}\max_{1 \leq j< k \leq n}|a_j-a_k|.
$$
\end{fact}


In the next lemma, we make some small changes to ensure that the sizes of the colour classes in our cycle structure $\mathcal C$ are equal, i.e. the size matrix has equal columns.
Note that this is a necessary condition for $\mathcal{C}$ to be $0$-balanced.
The proof is as follows.
We assume that $M$ is $1$-balanced. 
So the sum of entries in each column is almost equal (to within $\pm \ell$).
We show that for each of the three colours (columns) $j=1,2,3$, we can find a `feeder cluster' $X_j$ of this colour which has large core degree.
Each feeder cluster has the property that it contains many vertices $x$ such that, for each $j,j'$, there are many clusters $Y_{j'}$ of colour $j'$ for which $x \rightarrow Y_{j'}$ is valid.
So if the $(j')$th column has sum which is too small, and the $j$th column has sum which is too large, we remove some vertices of large degree which lie in $X_{j}$ and add them to a cluster of colour $j'$.

\begin{lemma}\label{colourclass}
Let $n \in 3\mathbb{N}$ and $\ell,m \in \mathbb{N}$ and $0 < 1/n \ll 1/\ell \ll \eps \ll c \ll d \ll \eta < 1$.
Suppose that $G$ is a graph on $n$ vertices with a spanning $(R,\ell,M,\eps,d)$-cycle structure $\mathcal{C}$, where $M$ is $(m,(1+\eps)m)$-bounded and $1$-balanced.
Suppose further that $d^c_{R,G}$ is $(\eta/2,3\ell)$-good with respect to $\mathcal C$.
Then $G$ contains a spanning $(R,\ell,N,\eps ^{1/3},d/2)$-cycle structure $\mathcal{C}'$, where $N$ is $2\ell$-balanced, is $((1-\eps)m,(1+2\eps)m)$-bounded, and has equal columns.
\end{lemma}

\begin{proof}
Write $V_{i,j}$ for the cluster of $\mathcal{C}$ corresponding to $(i,j) \in V(R)$, and $M := (m_{i,j})$, where $m_{i,j} := |V_{i,j}|$.
As before, given a vertex $v \in V(G)$ and $(i,j) \in [\ell] \times [3]$, we say that $v \rightarrow V_{i,j}$ is \emph{valid} if $d_G(v,V_{i,j'}) \geq (d-\eps)m$ for all $j' \in [3] \setminus \lbrace j \rbrace$.

For $1 \leq j \leq 3$, let $M_j := \sum_{1 \leq i \leq \ell}m_{i,j}$ be the sum of the entries in the $j$th column of $M$.
Since $\mathcal{C}$ is spanning,
\begin{equation}\label{eq0}
M_1+M_2+M_3 = n.
\end{equation}
Since $M$ is $1$-balanced, for all $1 \leq j < j' \leq 3$ we have
$$
|M_j-M_{j'}| \leq \sum_{1 \leq i \leq \ell}|m_{i,j}-m_{i,j'}| \leq \ell.
$$
Therefore Fact~\ref{easy} applied with $3,M_j$ playing the roles of $n,a_i$ together with (\ref{eq0}) imply that
\begin{equation}\label{eq1}
\left|M_j-\frac{n}{3}\right| \leq \frac{2\ell}{3}\ \ \text{ for }\ \ 1 \leq j \leq 3.
\end{equation}
Since $d^c_{R,G}$ is $(\eta/2,|R|)$-good, 
Proposition~\ref{vertexdeg}(i) applied with $R,d^c_{R,G},V(R)$ playing the roles of $G,d'_G,X$ implies that
there exists $\mathcal{X} \subseteq V(R)$ with $|\mathcal{X}| = 2|R|/3$ and $d^c_{R,G}(X) \geq (2/3+\eta/2)|R|$ for all $X \in \mathcal{X}$.
Proposition~\ref{vertexdeg}(ii) applied with $R,d^c_{R,G},V(R) \setminus \mathcal{X}, \eta|R|/4$ playing the roles of $G,d'_G,X,k$ implies that
there exists $\mathcal{Y} \subseteq V(R) \setminus \mathcal{X}$ with $|\mathcal{Y}| \geq \eta|R|/4$ such that every $Y \in \mathcal{Y}$ has $d^c_{R,G}(Y) \geq (2/3+\eta/4)|R|+2$.
Therefore there are at least $(2/3+\eta/4)|R|$ vertices $U \in V(R)$ with $d^c_{R,G}(U) \geq (2/3+\eta/4)|R|$.
Then, for each $1 \leq j \leq 3$, there is some $i_j \in [\ell]$ and a \emph{feeder cluster} $X_j := V_{i_j,j}$ such that $d^c_{R,G}((i_j,j)) \geq (2/3+\eta/4)|R|$.
Let $I' := \lbrace i_1,i_2,i_3 \rbrace$.
By definition of core degree, there exists $C_j \subseteq X_j$ such that $|C_j| \geq c|X_j| \geq cm$ and $d_G(x) \geq (2/3+\eta/4)n$ for all $x \in C_j$.
Proposition~\ref{replacelarge} applied with $\eta/2$ playing the role of $\eta$ implies that  for $x \in C_j$, there exists $I_x(j) \subseteq [\ell]$ with $|I_x(j)| \geq \eta \ell/2$ such that, for all $i' \in I_x(j)$ and $j' \in [3]$, we have $d_G(x,V_{i',j'}) \geq (d-\eps)m$.

$M$ has equal columns if $M_1=M_2 = M_3 = n/3$.
By the observation immediately after Definition~\ref{cycstruc}, we may suppose without loss of generality that $M_1 \leq M_2 \leq M_3$.
So $M_1 \leq M_2,n/3 \leq M_3$.
In fact we will assume that
\begin{equation}\label{eq2}
M_1 \leq M_2 \leq \frac{n}{3} \leq M_3.
\end{equation}
(The other case is similar.)
We wish to move some suitable vertices from the feeder cluster $X_3$ into clusters of colours $1$ and $2$ so that the new column sums are equal.
Choose $B_{3,2} \subseteq C_3$ with
\begin{equation}\label{eq3}
|B_{3,2}| =\frac{ n}{3}-M_2 \stackrel{(\ref{eq1})}{\leq} \frac{2\ell}{3} \leq  |C_3|.
\end{equation}
For each $x \in B_{3,2}$, we can choose an arbitrary $i_x \in I_x(3) \setminus I'$ so that $x \rightarrow V_{i_x,2}$ is valid.
We have
$$
M_3-|B_{3,2}| \stackrel{(\ref{eq0}),(\ref{eq3})}{=} \frac{2n}{3} - M_1 \stackrel{(\ref{eq2})}{\geq} M_1.
$$
Choose $B_{3,1} \subseteq C_3 \setminus B_{3,2}$ with
\begin{equation}\label{eq4}
|B_{3,1}|=\frac{n}{3}-M_1 \stackrel{(\ref{eq1})}{\leq} \frac{2\ell}{3}  \stackrel{(\ref{eq3})}{\leq} |C_3 \setminus B_{3,2}|.
\end{equation}
For each $x \in B_{3,1}$, we can choose an arbitrary $i_x \in I_x(3) \setminus I'$ so that $x \rightarrow V_{i_x,1}$ is valid.

For $j=1,2$, let $X_{i,j} := V_{i,j} \cup \lbrace x \in B_{3,j} : i_x=i \rbrace$ and let $X_{i,3} := V_{i,3} \setminus (B_{3,1} \cup B_{3,2})$.
For all $(i,j) \in [\ell] \times [3]$, let $n_{i,j} := |X_{i,j}|$ and let $N := (n_{i,j})$.
Now $\lbrace X_{i,j} : (i,j) \in [\ell] \times [3] \rbrace$ is a partition of $V(G)$.
We claim that it induces a spanning cycle structure $\mathcal{C}'$.

Observe that
\begin{equation}\label{cc1}
M_1+|B_{3,1}| = M_2+|B_{3,2}|=M_3-|B_{3,1}|-|B_{3,2}| = \frac{n}{3}.
\end{equation}
So, for $j=1,2$ we have
$$
\sum_{1 \leq i \leq \ell}n_{i,j} = \sum_{1 \leq i \leq \ell}m_{i,j}+|B_{3,j}|=M_j+|B_{3,j}| \stackrel{(\ref{cc1})}{=} \frac{n}{3}
$$
and similarly $\sum_{1 \leq i \leq \ell}n_{i,3}=n/3$.
So $N$ has equal columns.
Note that $X_{i_3,1}=V_{i_3,1}$ and $X_{i_3,2}=V_{i_3,2}$ and $X_{i_3,3} = V_{i_3,3} \setminus( B_{3,1} \cup B_{3,2})$.
So
$$
|n_{i_3,j}-n_{i_3,j'}| \leq |m_{i_3,j}-m_{i_3,j'}| + |B_{3,1}| + |B_{3,2}| \stackrel{(\ref{eq3}),(\ref{eq4})}{\leq} 1+\frac{4\ell}{3} \leq 2\ell.
$$
Suppose that $i \neq i_3$.
Then $X_{i,3}=V_{i,3}$ and
$$
|n_{i,j}-n_{i,j'}| \leq |m_{i,j}-m_{i,j'}| + \max \lbrace |B_{3,1}|, |B_{3,2}| \rbrace \leq 1+\frac{2\ell}{3} \leq 2\ell.
$$
So $N$ is $2\ell$-balanced.
Similar calculations show that, for all $(i,j) \in [\ell] \times [3]$,
$$
|X_{i,j} \triangle V_{i,j}| \leq |B_{3,1}| + |B_{3,2}| \leq 2\ell.
$$
Thus,
$$
(1-\eps)m \leq m - 2\ell \leq |X_{i,j}| \leq (1+\eps)m + 2\ell \leq (1+2\eps)m.
$$
So $N$ is $((1-\eps)m,(1+2\eps)m)$-bounded.
For all $v \in X_{i,j} \setminus V_{i,j}$ we have $i \in I_v(3) \subseteq C_3$, so $d_G(v,V_{i,j'}) \geq (d-\eps)m$ for all $j' \in [3]$.
Then Proposition~\ref{dull} implies that the partition into $X_{i,j}$s induces a spanning $(R,\ell,N,\eps^{1/3},d/2)$-cycle structure $\mathcal{C}'$.
\end{proof}

The next proposition shows that $Z_\ell \subseteq R$ implies that one can slightly change the size of clusters in the same colour class in our cycle structure.
That is, given $V_{i,j}$ and $V_{k,j}$, we can increase $|V_{k,j}|$ by $b$ and decrease $|V_{i,j}|$ by $b$, so long as $b$ is not too large.
We achieve this by successively moving vertices from $V_{i,j}$ to $V_{i+1,j}$, then $V_{i+1,j}$ to $V_{i+2,j}$, and so on, until we reach $V_{k,j}$.
In terms of size matrices, this means we can redistribute the weight within a column.

\begin{proposition}\label{goround}
Let $n,\ell,m \in \mathbb{N}$ and $0 < 1/n \ll 1/\ell \ll \eps \ll d \ll \eta < 1$.
Suppose that $G$ is a graph on $n$ vertices with a spanning $(R,\ell,M,\eps,d)$-cycle structure $\mathcal{C}$, where $M$ is $((1-\eps)m,(1+2\eps)m)$-bounded.
Let $(i,j) \in [\ell] \times [3]$.
Then there exist at least $(1-8\eps)m$ vertices  $v \in V_{i,j}$ such that $d_G(v,V_{i \pm 1,j'}) \geq (d-2\eps)m$ for all $j' \in [3] \setminus \lbrace j \rbrace$ (and addition is modulo $\ell$).
\end{proposition}

\begin{proof}
Recall that, since $Z_\ell \subseteq R$ by (C2), we have that $(i,j)(i \pm 1,j') \in E(R)$ for all $j' \in [3] \setminus \lbrace j \rbrace$.
Then Fact~\ref{supereasy} implies that there exist four sets $X_{j'}^\pm \subseteq V_{i,j}$ with $|X_{j'}^\pm| \geq (1-\eps)|V_{i,j}| \geq (1-2\eps)m$ such that every $x \in X_{j'}^\pm$ has $d_G(x,V_{i\pm1,j'}) \geq (d-\eps)|V_{i\pm1,j'}| \geq (d-2\eps)m$.
Observe that the intersection of these sets has size at least $(1-8\eps)m$, and every vertex within has the required properties.
\end{proof}
Suppose that, instead of $Z_\ell \subseteq R$, we could only guarantee that $C^2 _{3\ell} \subseteq R$. Then the conclusion of the previous proposition may fail to hold. For example,
neither $(i,2)(i-1,1)$ nor $(i,2)(i+1,3)$ may be edges of $R$. Then it could be that every vertex $x \in V_{i,2}$ has $d_G (x, V_{i-1,1})=d_G (x, V_{i+1,3})=0$. So in this case no vertex in $V_{i,2}$ can be moved to $V_{i-1,2}$ or $V_{i+1,2}$.

Now, given a cycle structure that has a $2\ell$-balanced size matrix with equal columns, we repeatedly apply Proposition~\ref{goround} to obtain a $0$-balanced cycle structure.

\begin{lemma}\label{shuffle}
Let $n \in 3\mathbb{N}$ and $\ell,m \in \mathbb{N}$ and $0 < 1/n \ll 1/\ell \ll \eps \ll d \ll \eta < 1$.
Suppose that $G$ is a graph on $n$ vertices with a spanning $(R,\ell,M,\eps,d)$-cycle structure $\mathcal{C}$, where $M$ is $((1-\eps)m,(1+\eps)m)$-bounded, $2\ell$-balanced, and has equal columns.
Then $G$ has a spanning $(R,\ell,N,\eps^{1/3},d/2)$-cycle structure $\mathcal{C}'$ such that $N$ is $((1-2\eps)m,(1+2\eps)m)$-bounded and $0$-balanced.
\end{lemma}

\begin{proof}
Write $\lbrace V_{i,j} : (i,j) \in [\ell] \times [3] \rbrace$ for the collection of clusters in $\mathcal{C}$, and write $M =: (m_{i,j})$, where $m_{i,j} := |V_{i,j}|$.
Given a vertex $v \in V(G)$ and $(i,j) \in [\ell] \times [3]$, we say that $v \rightarrow V_{i,j}$ is \emph{valid} if $d_G(v,V_{i,j'}) \geq (d-2\eps)m$ for all $j' \in [3] \setminus \lbrace j \rbrace$.

We claim that, for each $1 \leq i \leq \ell$, there exists $n_i \in \mathbb{N}$ so that
\begin{equation}\label{shuf1}
|n_i - m_{i,j}| \leq 2\ell\ \ \text{ for all }\ \ j=1,2,3,\ \ \text{ and }\ \ \sum_{1 \leq i \leq \ell}n_i = \frac{n}{3}.
\end{equation}
To see this, let $\overline{m_i} := (m_{i,1}+m_{i,2}+m_{i,3})/3$.
As an initial try, take $n_i := \lceil \overline{m_i} \rceil$ for all $i$.
Then, since $\mathcal{C}$ is spanning,
$$
\frac{n}{3} = \frac{1}{3} \sum_{(i,j) \in [\ell] \times [3]}m_{i,j} \leq \sum_{1 \leq i \leq \ell}n_i  < \sum_{1 \leq i \leq \ell}(\overline{m_i}+1) = \frac{n}{3} + \ell
$$
and so $0 \leq \sum_{1 \leq i \leq \ell}n_i - n/3 \leq \ell-1$.
Since this value is less than the number of $n_i$s, we can reduce exactly $\sum_{1 \leq i \leq \ell}n_i-n/3$ of them by one.
So, for each $i$ we have $n_i \in \lbrace \lceil \overline{m_i} \rceil, \lceil \overline{m_i} \rceil-1 \rbrace$.
Therefore $|n_i-\overline{m_i}| \leq 1$ for all $1 \leq i \leq \ell$.
Recall that $M$ is $2\ell$-balanced.
Fact~\ref{easy} applied for $1 \leq j \leq 3$ with $3,m_{i,j}$ playing the roles of $n,a_j$ implies that $|m_{i,j} - \overline{m_i}| \leq 4\ell/3$.
But then, for each $1 \leq i \leq \ell$ we have
$$
|n_i - m_{i,j}| \leq |n_i - \overline{m_i}| + |\overline{m_i} - m_{i,j}| \leq 1 + 4\ell/3 \leq 2\ell,
$$
proving the claim.

In the remainder of the proof, we will adjust $\mathcal{C}$ until it has size matrix $N = (n_{i,j})$ where $n_{i,j} := n_i$ for all $(i,j) \in [\ell] \times [3]$.
Let $K := 3\ell^2$.
Suppose, for some $0 \leq k < K$, we have found for each $(i,j) \in V(R)$ subsets $V^k_{i,j} \subseteq V(G)$ such that the following hold:
\begin{itemize}
\item[($\alpha_k$)] $\lbrace V^k_{i,j} : (i,j) \in [\ell] \times [3] \rbrace$ is a partition of $V(G)$;
\item[($\beta_k$)] for all $v \in V^k_{i,j} \setminus V_{i,j}$ we have that $v \rightarrow V_{i,j}$ is valid;
\item[($\gamma_k$)] for all $(i,j) \in [\ell] \times [3]$ we have $|V^k_{i,j} \triangle V_{i,j}| \leq 2k$;
\item [($\delta_k$)] for all $1 \leq j \leq 3$ we have $\sum_{1 \leq i \leq \ell}|V_{i,j}^k| = n/3$, and $\sum_{(i,j) \in [\ell] \times [3]}||V_{i,j}^k|-n_i| \leq 6\ell^2-2k$.
\end{itemize}

Notice that we can set $V_{i,j}^0 := V_{i,j}$ for all $(i,j) \in [\ell] \times [3]$.
Indeed, since $\mathcal{C}$ is spanning, ($\alpha_0$) holds.
Properties ($\beta_0$) and ($\gamma_0$) are vacuous.
To see ($\delta_0$), note that, for all $1 \leq j \leq 3$,
$\sum_{1 \leq i \leq \ell}|V_{i,j}^0| = \sum_{1 \leq i \leq \ell}m_{i,j} = n/3$ since $M$ has equal columns.
Furthermore,
$$
\sum_{(i,j) \in [\ell] \times [3]}||V^0_{i,j}|-n_i| = \sum_{(i,j) \in [\ell] \times [3]}|m_{i,j}-n_i| \stackrel{(\ref{shuf1})}{\leq} 6\ell^2.
$$

If $|V_{i,j}^k|=n_i$ for all $(i,j) \in [\ell] \times [3]$, then we stop.
Otherwise, we will obtain sets $V_{i,j}^{k+1}$ from $V_{i,j}^k$.
Since $\sum_{(i,j) \in [\ell] \times [3]}|V_{i,j}^k| = 3\sum_{1 \leq i \leq \ell}n_i = n$ by ($\alpha_k$) and (\ref{shuf1}), and $n_i \in \mathbb{N}$,
there exists some $(i^+,j_0) \in [\ell] \times [3]$ such that $|V^k_{i^+,j_0}| \geq n_{i^+}+1$.
Now, since $\sum_{1 \leq i \leq \ell}|V_{i,j_0}^k| = n/3 = \sum_{1 \leq i \leq \ell}n_i$ by ($\delta_k$), there exists $1 \leq i^- \leq \ell$ such that $|V^k_{i^-,j_0}| \leq n_{i^-}-1$.

Proposition~\ref{goround} applied repeatedly implies that, for all integers $r \geq 0$, there exist $(1-8\eps)m$ vertices $v \in V_{i^++r,j_0}$ such that $v \rightarrow V_{i^++r+1,j_0}$ is valid (where, here and for the rest of the proof, addition is modulo $\ell$).
Let $r_0$ be the least non-negative integer such that $i^++r_0+1 \equiv i^- \mod \ell$.
So $0 \leq r \leq \ell-1$. 
Now $(1-8\eps)m - 2K = (1-8\eps)m - 6\ell^2 > m/2$ so ($\gamma_k$) implies that for each $0 \leq r \leq r_0$, we can find $x_r \in V^k_{i^++r,j_0}$ such that $x_r \rightarrow V_{i^++r+1,j_0}$ is valid.

For each $(i,j) \in [\ell]\times [3]$, set
\begin{equation*}
V^{k+1}_{i,j} := \left\{
  \begin{array}{l l}
    V^k_{i,j} \setminus \lbrace x_0 \rbrace & \quad \text{if $(i,j)=(i^+,j_0)$}\\
    V^k_{i,j} \cup \lbrace x_{i-1} \rbrace \setminus \lbrace x_i \rbrace & \quad \text{if $i^++1 \leq i \leq i^++r_0$ and $j=j_0$}\\
    V^k_{i,j} \cup \lbrace x_{i-1} \rbrace & \quad \text{if $(i,j)=(i^-,j_0)$}\\
    V^k_{i,j} & \quad \text{otherwise}.
  \end{array} \right.
\end{equation*}

The definition of $r_0$ implies that ($\alpha_{k+1}$) holds.
The choice of $x_r$ implies that ($\beta_{k+1}$) holds.
We have
$$
|V_{i,j}^{k+1} \triangle V_{i,j}| \leq |V^{k+1}_{i,j} \triangle V^k_{i,j}| + |V^k_{i,j} \triangle V_{i,j}| \stackrel{\text{($\gamma_k$)}}{\leq} 2(k+1),
$$
proving ($\gamma_{k+1}$).
Finally, observe that $||V^{k+1}_{i^\pm,j_0}|-n_{i^\pm}| =||V^k_{i^\pm,j_0}|-n_{i^\pm}|-1$ and $|V^{k+1}_{i,j}|=|V^k_{i,j}|$ for all other $(i,j)$.
Therefore
$$
\sum_{(i,j) \in [\ell] \times [3]}||V_{i,j}^{k+1}|-n_i| = \sum_{(i,j) \in [\ell] \times [3]}||V_{i,j}^{k}|-n_i| - 2 \stackrel{\text{($\delta_k$)}}{\leq} 6\ell^2-2(k+1),
$$
proving ($\delta_{k+1}$).

So, for each $0 \leq k \leq K$, either the procedure has terminated, or we are able to proceed to step $k+1$.
Therefore there is some $p \leq K = 3\ell^2$ such that $\sum_{(i,j) \in [\ell] \times [3]}||V_{i,j}^{p}|-n_i| = 0$.
So $|V_{i,j}^p|=n_i$ for all $(i,j) \in [\ell] \times [3]$.
Set $X_{i,j} := V_{i,j}^p$ for all $(i,j) \in [\ell] \times [3]$.

We claim that the partition into $X_{i,j}$s induces a spanning cycle structure $\mathcal{C}'$.
Let $N := (n_{i,j})$ where $n_{i,j} := n_i$ for all $(i,j) \in [\ell] \times [3]$.
Then $N$ is the size matrix of $\mathcal{C}'$ and is $0$-balanced by definition.
Note that, by ($\gamma_p$), for all $(i,j) \in [\ell] \times [3]$ we have
$$
(1-2\eps)m \leq (1-\eps)m - 2K \leq |X_{i,j}| \leq (1+\eps)m + 2K \leq (1+2\eps)m.
$$
So $N$ is $((1-2\eps)m,(1+2\eps)m)$-bounded.
Finally, Proposition~\ref{dull}  implies that $\mathcal{C}'$ is an $(R,\ell,N,\eps^{1/3},d/2)$-cycle structure $\mathcal{C}'$.
\end{proof}

We are now able to prove the main result of this section, Lemma~\ref{0balanced}.

\medskip
\noindent
\emph{Proof of Lemma~\ref{0balanced}.} Suppose that $G$ is a sufficiently large graph on $n$ vertices as in the statement of the lemma.
Apply Lemma~\ref{reg} (the Regularity lemma) with parameters $\eps^{100},4L'$ to obtain $L^* \in \mathbb{N}$.
Since $L^*$ depends only on $\eps$ and $L'$, which appear to the right of $L_0$ in the hierarchy, we may assume that $1/L_0 \leq 1/L^*$.
Apply Lemma~\ref{reg}  to $G$ with parameters $\eps^{100},16d,4L'$ to obtain clusters $V_1, \ldots, V_L$ of size $m$, an exceptional set $V_0$, a pure graph $G'$ and a reduced graph $R'$.
So $|R'|=L$ where $4L' \leq L \leq L^*$ and $|V_0| \leq \eps^{100} n$; and $G'[V_i,V_j]$ is $(\eps^{100},16d)$-regular whenever $ij \in E(R')$.
Lemma~\ref{reg}(iv) states that $d_{G'}(x) > d_G(x)-(d+\eps)n$ for all $x \in V(G)$.
Then Proposition~\ref{stoneage} implies that $G'$ is $(\eta/2,n)$-good.
Choose $\alpha$ such that $\eps \ll \alpha \ll d$.
By Lemma~\ref{reduced}(ii), $d_{R',G}^\alpha$ and $R'$ are $(\eta/2,L)$-good.
Further, Lemma~\ref{reduced} applied to $G'$ implies that $d_{R',G'}^\alpha$ is $(\eta/4,L)$-good.

Apply Lemma~\ref{trianglecycle} with $L,\eta/2,R',\eps^{100}$ playing the roles of $n,\eta,G,\eps$ to obtain $\ell \in \mathbb{N}$ with $(1-\eps^{100})L \leq 3\ell \leq L$ such that $R' \supseteq Z_\ell$ (where $Z_\ell$ is the $\ell$-triangle cycle).
Observe that $L' \leq L/4 \leq \ell \leq L/3 \leq L^* \leq L_0$, as required. 
Let $R := R'[V(Z_\ell)]$.
Let also $V_0'' := V_0 \cup \bigcup_{i \in V(R') \setminus V(R)}V_i$.
Then
\begin{equation}\label{V0''}
|V_0''| \leq \eps^{100} n + \eps^{100} L m \leq 2\eps^{100} n.
\end{equation}
Relabel the vertices of $Z_\ell$ (and hence $R$) in the canonical way given in Definition~\ref{ltricycle}.
So $V(R) = [\ell] \times [3]$, and for all $1 \leq i \leq \ell$ and distinct $1 \leq j,j' \leq 3$ we have $(i,j)(i,j') \in E(Z_\ell)$ and $(i,j)(i+1,j') \in E(Z_\ell)$, where addition is modulo $\ell$.
Then $R \supseteq Z_\ell \supseteq T_\ell$, where $T_\ell$ consists of the triangles $(i,1)(i,2)(i,3)$ for $1 \leq i \leq \ell$.
Given a vertex $(i,j)$ in $R$ write $V_{i,j}$ for the cluster in $G$ corresponding to $(i,j)$.

Apply Lemma~\ref{superslice} with $R,3\ell,\eps^{100},16d,G',T_\ell,2$ playing the roles of $R,L,\eps,d,G,H,\Delta$ to obtain for each $(i,j) \in V(T_\ell) = V(R)$ a subset $V_{i,j}' \subseteq V_{i,j}$ of size
\begin{equation}\label{m'}
m' := (1-\eps^{50})m
\end{equation}
such that for all $1 \leq i \leq \ell$ and $1 \leq j < j' \leq 3$, the graph $G'[V_{i,j}',V_{i,j'}']$ is $(4\eps^{50},8d)$-superregular.
Let $(i,j)(i',j') \in E(R)$ be arbitrary.
Then Proposition~\ref{superslice2}(i) with $\eps^{100},16d,\eps^{50}$ playing the roles of $\eps,d,\eps '$ implies that $G'[V'_{i,j},V'_{i',j'}]$ is $(2\eps^{50},8d)$-regular and hence $(4\eps^{50},8d)$-regular.
Let $V_0'$ be the set of all those vertices of $G$ not contained in any $V'_{i,j}$.
Then
$$
|V_0'| \stackrel{(\ref{m'})}{\leq} |V_0''| + 3\eps^{50}\ell m \stackrel{(\ref{V0''})}{\leq} (2\eps^{100} + \eps^{50})n \leq 2\eps^{50}n.
$$

Let $N_0$ be the $\ell \times 3$ matrix in which every entry is $m'$.
It is now clear that $G'$ has an $(R,\ell,N_0,4\eps^{50},8d)$-cycle structure $\mathcal{C}_0$ where the $V'_{i,j}$ are the clusters of $\mathcal C_0$ and $V'_0$ is the exceptional set.
In particular, now we view the vertices in $R$ as corresponding to the clusters $V'_{i,j}$. Recall that $d^{\alpha} _{R',G'}$ is $(\eta/4,L)$-good when we view the vertices in $R'$ 
 as corresponding to the clusters $V_{i,j}$. Thus, Proposition~\ref{stoneage} implies that $d_{R,G'}^\alpha$ is $(\eta/8,L)$-good when we view the vertices in $R$ 
 as corresponding to the clusters $V_{i,j}$. So by definition of core degree, $d^{\alpha/2} _{R,G'}$ is $(\eta/8,3\ell)$-good when we view the vertices in $R$ 
 as corresponding to the clusters $V'_{i,j}$ (i.e. $d^{\alpha/2} _{R,G'}$ is $(\eta/8,3\ell)$-good with respect to $\mathcal C_0$).

We may therefore
apply Lemma~\ref{swap} with $n,\eta/4,G',m',4\eps^{50},8d,\alpha/2$ playing the roles of $n,\eta,G,m,$ $\eps,d,c$ to show that $G'$ has a spanning $(R,\ell,N_1,\eps^{9},4d)$-cycle structure $\mathcal{C}_1$, where
$N_1$ is $(m',(1+2\eps^{25})m')$-bounded and $1$-balanced. Moreover, $d^{\alpha/4} _{R,G'}$ is $(\eta/8,3\ell)$-good with respect to $\mathcal C_1$.

Apply Lemma~\ref{colourclass} with $G',\mathcal{C}_1, \alpha /4$ playing the roles of $G,\mathcal{C},c$ to show that $G'$ has a spanning $(R,\ell,N_2,\eps^{3},2d)$-cycle structure $\mathcal{C}_2$, where
$N_2$ is $2\ell$-balanced, $((1-\eps^{9})m',(1+2\eps^{9})m')$-bounded, and has equal columns.

Finally, apply Lemma~\ref{shuffle} with $G',\mathcal{C}_2$ playing the roles of $G,\mathcal{C}$ to show that $G'$ has a spanning $(R,\ell,M,\eps,d)$-cycle structure $\mathcal{C}$ where $M$ is $0$-balanced and $((1-2\eps^{3})m',(1+2\eps^{3})m')$-bounded, and hence $((1-\eps)m,(1+\eps)m)$-bounded by (\ref{m'}).
\hfill$\square$

\medskip

\subsection{Embedding the square of a Hamilton cycle}

Given $t \in \mathbb{N}$, recall that  $C^2_{3t}$ denotes the square cycle on $3t$ vertices. In this section we will always assume implicitly that $ C^2_{3t}$ has 
 vertex set $[t] \times [3]$ such that for all $1 \leq i \leq t$ and distinct $1 \leq j,j' \leq 3$, we have $(i,j)(i,j') \in E(C^2_{3t})$ and $(i,2)(i+1,1), (i,3)(i+1,1), (i,3)(i+1,2) \in E(C^2_{3t})$, where addition is modulo $t$.
Observe that $T_t \subseteq C^2_{3t}$.

\newcommand*\num{4}

\begin{center}
\begin{figure}[H]
\begin{tikzpicture}[every node/.style={draw,circle,fill=black,inner sep=0.5mm},scale=5]

\node[draw=none,fill=none] (albert) at ($(1.5,{(1/6)*sqrt(3)})$) {};

\begin{scope}[shift=(albert)]
\draw[draw=none,fill=gray!40] ($(0,{(1/3)*sqrt(3)-0.2})$) ellipse (0.1cm and 0.25cm);

\begin{scope}[rotate=120]
\draw[draw=none,fill=gray!40] ($(0,{(1/3)*sqrt(3)-0.2})$) ellipse (0.1cm and 0.25cm);
\end{scope}

\begin{scope}[rotate=240]
\draw[draw=none,fill=gray!40] ($(0,{(1/3)*sqrt(3)-0.2})$) ellipse (0.1cm and 0.25cm);
\end{scope}

\end{scope}

\node[draw=none,fill=none] (monkey) at ($(3,{(1/3)*sqrt(3)})$) {};

\begin{scope}[shift=(monkey)]
\draw[draw=none,fill=gray!40] ($(0,{-(1/3)*sqrt(3)+0.2})$) ellipse (0.1cm and 0.25cm);

\begin{scope}[rotate=120]
\draw[draw=none,fill=gray!40] ($(0,{-(1/3)*sqrt(3)+0.2})$) ellipse (0.1cm and 0.25cm);
\end{scope}

\begin{scope}[rotate=240]
\draw[draw=none,fill=gray!40] ($(0,{-(1/3)*sqrt(3)+0.2})$) ellipse (0.1cm and 0.25cm);
\end{scope}

\end{scope}

\begin{scope}

\clip(0.75,-0.3) rectangle (3.85,1.04);

    \foreach \row in {0, 1, ...,\num} {

    \node[] (\row one) at ($\row*(1, 0)$) {};    
    \node[] (\row two) at ($\row*(1,0)+(0.5,{0.5*sqrt(3)})$) {};
        
       }
        
\node[draw=none,fill=none] at ($(1one)-(0.15,0)$) {\small$(p_i,1)$};
\node[draw=none,fill=none] at ($(2one)+(0.25,0.05)$) {\small$(p_i,3) = x_i$};
\node[draw=none,fill=none] at ($(3one)+(0.21,0.00)$) {\small$(p_i+1,2)$};
      
\node[draw=none,fill=none] at ($(1two)+(-0.15,0.04)$) {\small$(p_i,2)$};
\node[draw=none,fill=none] at ($(2two)-(0.3,0.05)$) {\small$y_{i+1}=(p_i+1,1)$};
\node[draw=none,fill=none] at ($(3two)+(0.21,-0.05)$) {\small$(p_i+1,3)$};

\node[draw=none,fill=none] at ($(1one)-(0,0.1)$) {\large$V_{i,1}$};
\node[draw=none,fill=none] at ($(2one)-(0,0.1)$) {\large$V_{i,3}$};
\node[draw=none,fill=none] at ($(3one)-(0,0.1)$) {\large$V_{i+1,2}$};
 
\node[draw=none,fill=none] at ($(1two)+(0,0.1)$) {\large$V_{i,2}$};
\node[draw=none,fill=none] at ($(2two)+(0,0.1)$) {\large$V_{i+1,1}$};
\node[draw=none,fill=none] at ($(3two)+(0,0.1)$) {\large$V_{i+1,3}$};

\node[draw=none,fill=none] at ($(1one)+(0.425,0.23)$) {\small$y_i$};
      \node[draw=none,fill=none] at ($(3two)-(0.44,0.24)$) {\small$x_{i+1}$};

\begin{scope}[]

    \foreach \x in {0,1,2,3,4} {

    \node[] (\x blah1) at ($(1,0)+(30:{\x*0.1})$) {};    
    \node[] (\x blah3) at ($(2,0)+(150:{\x*0.1})$) {};
    \node[] (\x blah2) at ($(1.5,{0.5*sqrt(3)})+(0,{\x*-0.1})$) {};
        
\draw (\x blah1) -- (\x blah2) -- (\x blah3) -- (\x blah1);

}

\draw[thick,color=red] (4blah3) -- (4blah1) -- (4blah2) -- (4blah3);

\draw (1two) -- (1blah3) -- (1one);
\draw (1blah2) -- (2blah3) -- (1blah1);
\draw (2blah2) -- (3blah3) -- (2blah1);
\draw (3blah2) -- (4blah3) -- (3blah1);

\draw (1blah2) -- (1one);
\draw (2blah2) -- (1blah1);
\draw (3blah2) -- (2blah1);
\draw (4blah2) -- (3blah1);

\end{scope}

\begin{scope}[]

    \foreach \x in {0,1,2,3,4} {

    \node[] (\x bla1) at ($(2.5,{0.5*sqrt(3)})+(-30:{\x*0.1})$) {}; 
    \node[] (\x bla2) at ($(3,0)+(0,{\x*0.1})$) {};
    \node[] (\x bla3) at ($(3.5,{0.5*sqrt(3)})+(210:{\x*0.1})$) {};
        
\draw (\x bla1) -- (\x bla2) -- (\x bla3) -- (\x bla1);

}

\draw[thick,color=red] (4bla2) -- (4bla3) -- (4bla1) -- (4bla2);

\draw (3one) -- (1bla1) -- (3two);

\draw (2bla2) -- (1bla3) -- (2bla1);
\draw (3bla2) -- (2bla3) -- (3bla1);
\draw (4bla2) -- (3bla3) -- (4bla1);

\draw (1bla2) -- (3two);

\draw (1bla2) -- (2bla1);
\draw (2bla2) -- (3bla1);
\draw (3bla2) -- (4bla1);

\end{scope}

\draw[thick,color=blue]  (2two) -- (2one);

\draw[thick,color=red] (1two) -- (2two); 
\draw[thick,color=red] (2one) -- (3one);
\draw[thick,color=red] (1two) -- (2one) -- (1one);
\draw[thick,color=red] (3two) -- (2two) -- (3one);

\draw[thick,color=red] (0blah1) -- (0blah2);
\draw[thick,color=red] (0bla2) -- (0bla3);

\begin{scope}
\clip(1,-0.3) rectangle (3.6,1.04);
\draw[thick,color=red] (4blah2) -- (0two);
\draw[thick,color=blue] (0two) -- (4blah1);
\draw[thick,color=red] (4blah1) -- (0one);
\draw[thick,color=red] (4bla2) -- (4one);
\draw[thick,color=blue] (4one) -- (4bla3);
\draw[thick,color=red] (4bla3) -- (4two);
\end{scope}

\begin{scope}
\clip(0.75,-0.3) rectangle (1,1.04);
\draw[dotted,thick,color=red] (4blah2) -- (0two) -- (4blah1) -- (0one);
\draw[dotted,thick,color=blue] (0two) -- (4blah1);
\draw[dotted,thick,color=red] (4blah1) -- (0one);
\end{scope}

\begin{scope}
\clip(3.6,-0.3) rectangle (3.85,1.04);
\draw[dotted,thick,color=red] (4bla2) -- (4one);
\draw[dotted,thick,color=blue] (4one) -- (4bla3);
\draw[dotted,thick,color=red] (4bla3) -- (4two);
\end{scope}

\end{scope}

\end{tikzpicture}

\caption{The square path $y_{i}P_ix_iy_{i+1}P_{i+1}x_{i+1}$ which forms part of the square cycle $C^2_n$; and the desired embedding into the clusters of $G$. The edges in $C^2_n[X]$ are coloured blue and the remaining edges in $J_\ell$ are coloured red.}\label{embedfig}
\end{figure}
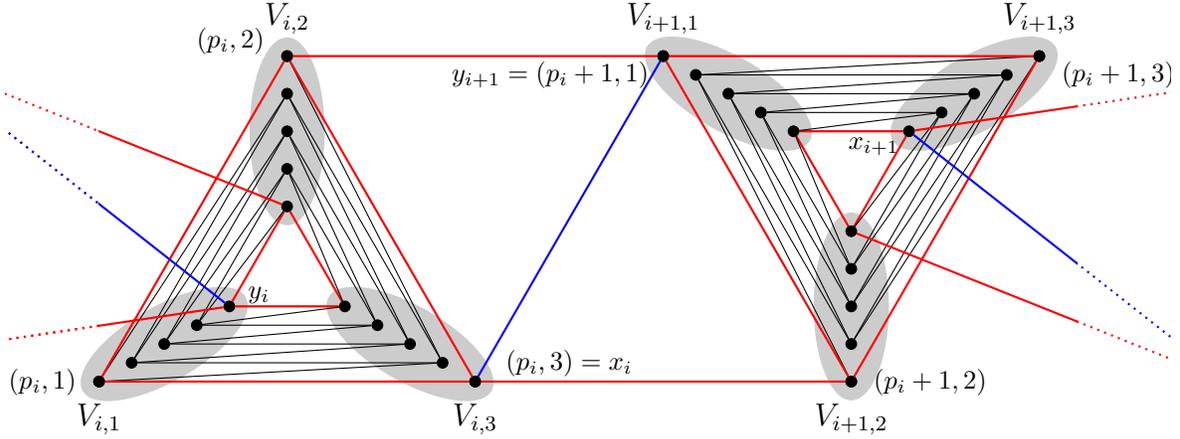
\end{center}

The following is essentially a special case of an argument in~\cite[from page 10]{3partite} and is a standard application of the Blow-up lemma, but we prove it here for completeness.

\begin{lemma}\label{embedHC2}
Let $0 < 1/n \ll 1/\ell \ll \eps \ll d \ll 1$.
Suppose that $G$ is a graph on $n$ vertices with a spanning $(R,\ell,M,\eps,d)$-cycle structure $\mathcal{C}$ such that $M$ is $((1-\eps)m,(1+\eps)m)$-bounded and $0$-balanced.
Then $G$ contains the square of a  Hamilton cycle.
\end{lemma}

\begin{proof}
By (C2), we have that $Z_\ell \subseteq R$.
Then $C^2_{3\ell} \subseteq R$.
This is all we require in the proof.
Write $\lbrace V_{i,j} : (i,j) \in [\ell] \times [3] \rbrace$ for the collection of clusters of $\mathcal{C}$, where $V_{i,j}$ corresponds to $(i,j) \in V(R)$.
So this is a partition of $V(G)$.
Since $M$ is $0$-balanced, for each $1 \leq i \leq \ell$ there exists $m_i \in \mathbb{N}$ such that
\begin{equation}\label{mi}
|V_{i,j}|=m_i = (1 \pm \eps)m\ \ \text{ for all }\ \ 1 \leq j \leq 3.
\end{equation}
Let $p_0 := 0$ and $p_i := \sum_{1 \leq r \leq i}m_r$ for all $1 \leq i \leq \ell$.
Note that
$
3p_\ell = n,
$
and, in particular, $n$ is divisible by $3$.
To prove the lemma, we will find an embedding $h : V(C^2_n) \rightarrow V(G)$, where we write
$$
C^2_n = (1,1)(1,2)(1,3)(2,1) \ldots (p_\ell,1)(p_\ell,2)(p_\ell,3).
$$
The embedding will map the first $3p_1$ vertices of $C^2_n$ to distinct vertices in $V_{1,1} \cup V_{1,2} \cup V_{1,3}$, and the $(3p_1+1)$th to $(3p_2)$th vertices of $C^2_n$ to distinct vertices in $V_{2,1} \cup V_{2,2} \cup V_{2,3}$, and so on.
For each $1 \leq i \leq \ell$, define
\begin{equation}\label{xiyi}
x_i := (p_i,3)\ \ \text{ and } \ \ y_i := (p_{i-1}+1,1).
\end{equation}
Define also
\begin{align}
\nonumber X_i &:= \left\lbrace x_i, y_{i+1} \right\rbrace\ \ \text{ and }\ \ Y_i := \lbrace (p_i,1),(p_i,2),(p_i+1,2),(p_i+1,3) \rbrace = N_{C^2_n}(X_i) \setminus X_i;\\
\label{Pi} P_i &:= (p_{i-1}+1,2)(p_{i-1}+1,3) \ldots (p_i,1)(p_i,2),
\end{align}
where $P_i$ is a square path.
Let $X := \bigcup_{1 \leq i \leq \ell}X_i$ and $Y := \bigcup_{1 \leq i \leq \ell}Y_i = N_{C^2_{n}}(X) \setminus X$.
Note further that $(P_i)^+_2 \cup (P_{i+1})^-_2 = Y_i$. 
We have that $C^2_{n} = y_1P_1x_1y_2P_2x_2y_3P_3 \ldots P_\ell x_\ell$.
(Figure~\ref{embedfig} shows the square path $y_iP_ix_iy_{i+1}P_{i+1}x_{i+1}$.)

Our strategy is as follows: first embed the vertices in $X \cup Y$ using the partial embedding lemma (Lemma~\ref{partialembed}), so that there are many choices for the embedding of each $y \in Y$.
Then, for each $1 \leq i \leq \ell$, apply the Blow-up lemma (Theorem~\ref{blowup}) to embed $P_i$ into $V_{i,1} \cup V_{i,2} \cup V_{i,3}$ in such a way that the two embeddings align.


Define $f : V(C^2_n) \rightarrow V(C^2_{3\ell})$ by $f((k,j)) = (g(k),j)$, where $g(k) \in [\ell]$ is such that 
$$
p_{g(k)-1} < k \leq p_{g(k)}.
$$ 
It is not hard to check that $f$ is a graph homomorphism, i.e. $f(x)f(y) \in E(C^2_{3\ell})$ whenever $xy \in E(C^2_n)$.
By a slight abuse of notation, we will write $V_{f((k,j))}$ for $V_{g(k),j}$.
We will find an embedding
$h : V(C^2_{n}) \rightarrow V(G)$ such that $h(x) \in V_{f(x)}$ for all $x \in C^2_n$.

For all $1 \leq i \leq \ell$, since $X_i \cup Y_i$ is a collection of $6$ consecutive vertices on $C^2_n$, we have that $J_\ell := C^2_{n}[X \cup Y]$ is a collection of $\ell$ vertex-disjoint square paths of order $6$. 
So $|J_\ell| = 6\ell \leq \eps m$ and $\Delta(J_\ell) = 4$.
Choose $c$ such that $\eps \ll c \ll d \ll 1$.
Apply Lemma~\ref{partialembed} with $C^2_{3\ell},G,\lbrace V_{i,j} : (i,j) \in [\ell] \times [3] \rbrace, J_\ell,c$ playing the roles of $R,G,\lbrace V_i : 1 \leq i \leq L \rbrace, H,c$.
Thus obtain an injective mapping $\tau : X \rightarrow V(G)$ with $\tau(x) \in V_{f(x)}$ for all $x \in X$, such that 
for all $y \in Y$ there exist sets $C_y \subseteq V_{f(y)} \setminus \tau(X)$ such that the following hold:
\begin{itemize}
\item[(i)] for all $1 \leq i \leq \ell$ (where addition is modulo $\ell$), we have that $\tau(x_i)\tau(y_{i+1}) \in E(G)$;
\item[(ii)] for all $y \in Y$ we have that $C_y \subseteq N_G(\tau(x))$ for all $x \in N_{C^2_{n}}(y) \cap X$;
\item[(iii)] $|C_y| \geq c|V_{f(y)}|$ for all $y \in Y$.
\end{itemize}

Note that for each $1 \leq i \leq \ell$, as displayed in Figure~\ref{embedfig},
$$
V_{i,1} \cap \tau(X) = \lbrace y_{i} \rbrace, \ \ V_{i,2} \cap \tau(X) = \emptyset\ \ \text{ and }\ \ V_{i,3} \cap \tau(X) = \lbrace x_i \rbrace.
$$
For all $(i,j) \in [\ell] \times [3]$, let $V'_{i,j} := V_{i,j} \setminus \tau(X)$.
So $|V'_{i,j}|=m_i-1$ for $j=1,3$; and $|V'_{i,2}|=m_i$.
Proposition~\ref{superslice2}(ii) implies that $G[V_{i,j}',V_{i,j'}']$ is $(2\eps,d/2)$-superregular for all $1 \leq i \leq \ell$ and $1 \leq j < j' \leq 3$.

Note that for each $1 \leq i \leq \ell$,
$P_i$ is a $3$-partite graph with $\Delta(P_i)=4$ and with vertex classes $W_1^i,W_2^i,W_3^i$ of sizes $m_i-1,m_i,m_i-1$ respectively, where $(k,j) \in W^i_j$ for all $(k,j) \in V(P_i)$.
Observe that $V(P_i) \cap Y = (P_i)^-_2 \cup (P_i)^+_2$.
So, by (iii), for each $y \in ((P_i)^-_2 \cup (P_i)^+_2) \cap W^i_j$, there is a set $C_y \subseteq V'_{i,j}$ with $|C_y| \geq cm_i$ that satisfies (ii).
Let $T_i$ be the triangle in $R$ spanned by $(i,1),(i,2),(i,3)$.  Let $f_i$ denote the restriction of $f$ on $P_i$. So
$f_i : V(P_i) \rightarrow V(T_i)$ where $f_i ((k,j)) = (i,j)$ for all $(k,j) \in V(P_i)$.

For each $1 \leq i \leq \ell$, apply Theorem~\ref{blowup} with $3,m_i-1,m_i,m_i-1,2\eps,V_{i,j}',T_i,d/2,P_i,W_j^i,4,f_i$ playing the roles of $k,n_1,n_2,n_3,\eps,V_j,J,d,H,W_j,\Delta,\phi$ with special vertices $y \in (P_i)^-_2 \cup (P_i)^+_2$ and associated sets $C_y$ playing the role of $S_y$.
Thus obtain an embedding of $P_i$ into $G[V'_{i,1} \cup V'_{i,2} \cup V'_{i,3}]$ such that every  $y \in (P_i)^-_2 \cup (P_i)^+_2$ is mapped to a vertex in $C_y$.
Note that, for $1 \leq i < i' \leq \ell$, every pair $z_i \in V(P_i)$ and $z_{i'} \in V(P_{i'})$ are mapped to different vertices of $G$.
By considering the union of these embeddings, we obtain a bijective mapping $\sigma : \bigcup_{1 \leq i \leq \ell}V(P_i) \rightarrow V(G) \setminus \tau(X)$ such that \begin{equation}\label{sigmaedge}
\sigma(x)\sigma(x') \in E(G)\ \ \text{ whenever }xx' \in \bigcup_{1 \leq i \leq \ell}E(P_i) \stackrel{(\ref{Pi})}{=} E(C^2_{n} \setminus X).
\end{equation}
In particular, we have that
\begin{equation}\label{wherey}
\sigma(y) \in C_y\ \ \text{ for all}\ \ y \in Y.
\end{equation}
Let $h : V(C^2_{n}) \rightarrow V(G)$ be defined by
\begin{equation}\label{overtaudef}
h(x) = \left\{
  \begin{array}{l l}
    \tau(x) & \quad \text{if $x \in X$}\\
    \sigma(x) & \quad \text{if $x \in V(C^2_{n}) \setminus X$.}
  \end{array} \right.
\end{equation}
It remains to show that $h$ is an embedding of $C^2_{n}$ in $G$.
Let $xy \in E(C^2_{n})$.
We consider three cases.
Suppose first that $x,y \in X$.
Then, without loss of generality, there is some $1 \leq i \leq \ell$ such that $x=x_i$ and $y=y_{i+1}$.
So $h(x)h(y)=\tau(x_i)\tau(y_{i+1}) \in E(G)$ by (i).
Suppose secondly that $x \in X$ and $y \in V(C^2_{n}) \setminus X$.
Then $y \in N_{C^2_{n}}(x) \setminus X \subseteq Y$, and so
$$
h(y) \stackrel{(\ref{overtaudef})}{=} \sigma(y) \stackrel{(\ref{wherey})}{\in} C_y \stackrel{\text{(ii)}}{\subseteq} N_G(\tau(x)) \stackrel{(\ref{overtaudef})}{=} N_G(h(x)),
$$
i.e. $h(x)h(y) \in E(G)$.
Suppose finally that $x,y \in V(C^2_{n}) \setminus X$.
Then $h(x)h(y)=\sigma(x)\sigma(y) \in E(G)$ by (\ref{sigmaedge}).
\end{proof}


\section{Proof of Theorem~\ref{mainthm}}\label{sectionend}

We will first prove Theorem~\ref{mainthm} for graphs whose order is divisible by three.

\begin{theorem}\label{almost}
Let $n \in 3\mathbb{N}$ and let $0 < 1/n \ll \eta \ll 1$.
Suppose that $G$ is an $\eta$-good graph on $n$ vertices.
Then $G$ contains the square of a Hamilton cycle.
\end{theorem}

\begin{proof}
Choose $L_0,L' \in \mathbb{N}$ and positive constants $\eps,d$ so that $0 \ll 1/n \ll 1/L_0 \ll 1/L' \ll \eps \ll d \ll \eta \ll 1$.
Apply Lemma~\ref{0balanced} to show that there exists a spanning subgraph $G' \subseteq G$, and $\ell \in \mathbb{N}$ with $L' \leq \ell \leq L_0$, such that $G'$ has a spanning $(R,\ell,M,\eps,d)$-cycle structure such that $M$ is $0$-balanced and $((1-\eps)m,(1+\eps)m)$-bounded.
Now apply Lemma~\ref{embedHC2} to show that $G'$, and hence $G$, contains the square of a Hamilton cycle.
\end{proof}

The proof of Theorem~\ref{mainthm} is now a short step away.

\medskip
\noindent
\emph{Proof of Theorem~\ref{mainthm}.}
Let $\eta > 0$.
Without loss of generality, we may assume that $\eta \ll 1$.
Choose $n_0 \in \mathbb{N}$ so that $0 < 1/n_0 \ll \eta$ and the conclusion of Theorem~\ref{almost} holds whenever $n \geq n_0-2$ and $\eta/2$ plays the role of $\eta$.
Let $G$ be a graph on $n \geq n_0$ vertices, whose degree sequence $d_1 \leq \ldots \leq d_n$ satisfies
$$
d_i \geq n/3 + i + \eta n\ \ \text{ for all }\ \ i \leq n/3.
$$
Note firstly that $G$ is $(2\eta/3)$-good.
Then (\ref{etagood}) with $2\eta/3$ playing the role of $\eta$ implies that we can find vertex-disjoint edges $x_1y_1, x_2y_2 \in E(G)$ such that $x_1,y_1, x_2,y_2 \in V(G)_{2\eta/3}$.

Let $k$ be the least non-negative integer such that $n \equiv k \mod 3$.
So $k \in \lbrace 0,1,2 \rbrace$.
Let $G'$ be the graph obtained as follows.
If $k=0$, set $G' := G$.
Otherwise, we let $z_j$ be a new vertex for each $1 \leq j \leq k$, and set
$$
V(G') := V(G) \cup \lbrace z_j : 1 \leq j \leq k \rbrace \setminus \lbrace x_j,y_j : 1 \leq j \leq k \rbrace
$$
and
$$
E(G') := E(G \setminus \lbrace x_j,y_j : 1 \leq j \leq k \rbrace) \cup \lbrace vz_j : 1 \leq j \leq k \ \text{ and }\ v \in N^2_G(x_j,y_j) \rbrace.
$$

Note that, for all $1 \leq j \leq k$ we have
$$
d_{G'}(z_j)  = |N_G^2(x_j,y_j)| \geq (1/3+\eta)n
$$
by Proposition~\ref{2ndnbrhd}(i).
It is easy to see that $G'$ is an $(\eta/2)$-good graph and $|G'| = n-k \equiv 0 \mod 3$.
Furthermore, $|G'| \geq n-2 \geq n_0-2$.
Then Theorem~\ref{almost} implies that $G'$ contains the square of a Hamilton cycle $C'$.
Since every neighbour of $z_j$ in $G'$ is a neighbour of both $x_j$ and $y_j$ in $G$, the cycle $C$ obtained from $C'$ by replacing each $z_j$ with the edge $x_jy_j$ (in any order) gives the square of a Hamilton cycle in $G$.
\hfill$\square$


\section{Concluding remarks}\label{concsec}

Recall that in Lemma~\ref{trianglecycle}, we showed that a graph $G$ as in Theorem~\ref{mainthm} contains an almost spanning copy of a so-called triangle cycle $Z_\ell$.
We then used this framework to embed the square of a Hamilton cycle.
(Roughly speaking, by \emph{framework} we mean a structure in the reduced graph which enables us to embed a subgraph into $G$.)
Frameworks similar to $Z_\ell$ have been used previously for embedding other spanning structures.

In~\cite{maxplanar}, K\"uhn, Osthus and Taraz showed that any graph $G$ on $n$ vertices and with $\delta(G) \geq (2/3 + o(1))n$ contains a spanning triangulation, i.e.~a maximal planar graph.
To embed the triangulation, the framework they used was the square of a Hamilton path.
(The error term $o(n)$ here was subsequently removed in~\cite{koplanar}, yielding an exact result.)

We say a graph $H$ on $n$ vertices has \emph{bandwidth} $b$ if there exists an ordering of the vertices $1, \ldots, n$ so that $|i-j| < b$ whenever $ij$ is an edge of $H$.
In~\cite{3partite}, B\"ottcher, Schacht and Taraz considered the more general problem of embedding (possibly spanning) graphs $H$ with small bandwidth.
They showed that any graph $G$ on $n$ vertices with $\delta(G) \geq (2/3 + o(1))n$ contains every $3$-chromatic graph $H$ on at most $n$ vertices and of bounded maximum degree and bandwidth $o(n)$.
Again, the framework used here was the square of a Hamilton path.
In later work~\cite{bst} they generalised this to $r$-chromatic graphs $H$ and used an analogue of $Z_\ell$ for their framework.

We believe that a graph as in Theorem~\ref{mainthm} contains a spanning triangulation.

\begin{conjecture}\label{conjplanar}
Given any $\eta >0$ there exists an $n_0 \in \mathbb N$ such that the following holds. If $G$ is a graph on $n \geq n_0$ vertices whose degree sequence $d_1\leq \dots \leq d_n$ satisfies
$$ d_i \geq n/3+i+\eta n \ \ \text{ for all } \ \ i \leq n/3,$$ then $G$ contains a spanning triangulation.
\end{conjecture}

If true, Conjecture~\ref{conjplanar} implies the aforementioned result of K\"uhn, Osthus and Taraz.
One approach to prove Conjecture~\ref{conjplanar} could be to use $Z_\ell$ as a framework for embedding (i.e. apply Lemma~\ref{trianglecycle}).
This approach could also be fruitful in attacking the following more general conjecture.

\begin{conjecture}\label{conjband}
Given any $\eta >0$ and $\Delta \in \mathbb{N}$, there exists a $\beta > 0$ and an $n_0 \in \mathbb N$ such that the following holds. Suppose that $H$ is a $3$-chromatic graph on $n \geq n_0$ with $\Delta(H) \leq \Delta$ and bandwidth at most $\beta n$.
If $G$ is a graph on $n$ vertices whose degree sequence $d_1\leq \dots \leq d_n$ satisfies
$$ d_i \geq n/3+i+\eta n \ \ \text{ for all } \ \ i \leq n/3,$$ then $G$ contains $H$.
\end{conjecture}

We conclude the paper by discussing degree sequence conditions that force the $k$th power of a Hamilton cycle in a graph. (The $k$th power of a Hamilton cycle $C$ is obtained from $C$ by adding an edge between every pair of vertices of distance at most $k$ on $C$.) 
A conjecture of Seymour~\cite{seymour} states that  every graph $G$ on $n$ vertices  with $\delta (G) \geq kn/(k+1)$ contains the $k$th power of a Hamilton cycle. Thus, Seymour's conjecture is a generalisation of Conjecture~\ref{posaconj}. Koml\'os,  S\'ark\"ozy and  Szemer\'edi~\cite{kss} proved Seymour's conjecture for sufficiently large graphs $G$. 
In light of Theorem~\ref{mainthm}, we believe the following degree sequence version of Seymour's conjecture is true.

\begin{conjecture}\label{conjnew}
Given any $\eta >0$ and $k \geq 2$ there exists an $n_0 \in \mathbb N$ such that the following holds. If $G$ is a graph on $n \geq n_0$ vertices whose degree sequence $d_1\leq \dots \leq d_n$ satisfies
$$ d_i \geq \frac{(k-1)n}{k+1}+i+\eta n \ \ \text{ for all } \ \ i \leq \frac{ n}{k+1},$$ then $G$ contains the $k$th power of a Hamilton cycle.
\end{conjecture}
If true, Conjecture~\ref{conjnew} would be essentially best possible. Indeed, the example in Proposition~17 in~\cite{bkt} shows that one cannot replace the term $\eta n$ in the degree sequence condition here with $-1$.
Note that  a necessary
condition for a graph $G$ to contain the $k$th power of a Hamilton cycle is that $G$
contains a perfect $K_{k+1}$-packing:
In~\cite{triangle} it was shown that the hypothesis of Conjecture~\ref{conjnew} indeed ensures that $G$ contains a perfect $K_{k+1}$-packing. 

We believe that most of the proof of Theorem~\ref{mainthm} naturally generalises to $k$th powers of Hamilton cycles. The main difficulty in proving Conjecture~\ref{conjnew} appears to be in proving a `connecting lemma' (i.e. an analogue of Lemma~\ref{abcd}). In particular,  the methods we use to prove Lemma~\ref{abcd} seem somewhat tailored to the case of the square of a Hamilton cycle.


\section*{Acknowledgements}

We are grateful to Daniela K\"uhn and Deryk Osthus for helpful comments on the manuscript, and to the referees for their careful and helpful reviews.
The first author was partially supported by ERC grant~306493 and the second author was supported by EPSRC grant EP/M016641/1.

\medskip

{\footnotesize \obeylines \parindent=0pt

\begin{tabular}{lll}
Katherine Staden                 &\ &  Andrew Treglown\\
Mathematics Institute						    &\ &  School of Mathematics \\
University of Warwick 					&\ & University of Birmingham   	 \\
Coventry                     &\ &  Birmingham  \\
CV4 2AL														&\ & B15 2TT	\\
UK																	&\ &  UK
\end{tabular}
}

\begin{flushleft}
{\it{E-mail addresses}:
\tt{k.l.staden@warwick.ac.uk, a.c.treglown@bham.ac.uk}}
\end{flushleft}

\end{document}